\documentclass[tikz,border=6pt]{article}
\usepackage[letterpaper,top=2cm,bottom=2cm,left=3cm,right=3cm,marginparwidth=1.75cm]{geometry}
\usepackage{amsmath}
\usepackage{amssymb}
\usepackage{amsthm}
\usepackage{placeins}
\usepackage{mathtools}
\usepackage{xcolor}
\usepackage{caption}
\usepackage{subcaption}
\usepackage{verbatim}
\usepackage{listings}
\usepackage{hyperref}
\usepackage{algorithm}
\usepackage{algpseudocode}
\newtheorem{theorem}{Theorem}
\newtheorem{lemma}[theorem]{Lemma}

\newtheorem{corollary}[theorem]{Corollary}

\newtheorem{definition}[theorem]{Definition}
\newtheorem{remark}[theorem]{Remark}
\newtheorem{proposition}[theorem]{Proposition}
\newtheorem{model}{Model}
\newcommand{\norm}[1]{\left\lVert#1\right\rVert}
\DeclareMathOperator{\SPAN}{span}
\DeclareMathOperator{\fl}{fl}

\DeclareMathOperator{\RANGE}{R}
\DeclareMathOperator{\EXP}{\mathop{{}\mathbb{E}}}

\DeclareMathOperator{\REAL}{Re}
\DeclareMathOperator{\gammar}{\gamma_\text{reg}}
\DeclareMathOperator{\gammae}{\gamma_\text{exact}}
\DeclareMathOperator{\gammaesquared}{\gamma^2_\text{exact}}
\DeclareMathOperator{\Aext}{A_{\epsilon}}
\DeclareMathOperator{\Aexthat}{\widehat{A}_{\epsilon}}
\DeclareMathOperator{\Calg}{C_\text{alg}}
\DeclareMathOperator{\Stail}{S_\text{tail}}

\DeclareMathOperator{\machepsdp}{u^\text{dp}}
\DeclareMathOperator{\machepssp}{u^\text{sp}}

\DeclareMathOperator*{\argmin}{arg\,min}

\DeclareMathOperator*{\esssup}{ess\,sup}

\newcounter{cntr}
\newcommand{\from}{\colon}
\usepackage{graphicx} 
\usepackage{tikz}
\usetikzlibrary{arrows.meta,positioning,calc,fit}
\usepackage{tikz-3dplot}
\tdplotsetmaincoords{75}{130}

\usepackage{mleftright}
\mleftright

\title{Function Approximation in Numerically Rank-Deficient Bases}
\author{Astrid Herremans\setcounter{cntr}{\value{footnote}}\footnote{Department of Computer Science, KU Leuven, 3001 Leuven, Belgium (\url{astrid.herremans@kuleuven.be, daan.huybrechs@kuleuven.be})} \kern3pt and Daan Huybrechs\setcounter{footnote}{\value{cntr}}\footnotemark}
\date{}

\begin{document}

\maketitle

\begin{abstract}
    We study linear function approximation in a finite basis under finite-precision arithmetic. In a highly non-orthogonal basis, certain directions are only weakly represented, so that rounding errors can significantly distort the effectively spanned space. In the first part of the paper, we formalize this phenomenon through the notion of numerically rank-deficient bases and analyze their associated numerical span. Using a novel model for the rounding errors involved, we prove that approximation in the numerical span behaves like approximation in exact arithmetic subject to an additional penalty proportional to the size of the expansion coefficients and the unit roundoff. A key implication is that numerical orthogonalization—that is, recombination of the rounded basis functions—cannot mitigate the effects induced by finite-precision arithmetic. The framework also provides a theoretical justification for $\ell^2$-regularized approximation by showing that it achieves near-optimal numerical accuracy. Moreover, regularization controls the amplification of rounding errors in the computation of expansion coefficients—an essential requirement for the existence of practical algorithms. In the second part of the paper, we address sampling for function approximation in the presence of numerical rank-deficiency. We demonstrate that regularization has another fundamental benefit: it relaxes the conditions required for accurate least squares approximation from sampled data. This effect is made concrete through an analysis of randomized sampling based on a regularized variant of the Christoffel function. The resulting sample complexity bounds depend on an effective dimension that measures the number of directions that remain useful after finite-precision rounding. We also show that regularization renders the Christoffel function computable in contrast to the standard Christoffel function, whose numerical evaluation may require arbitrarily high precision in the presence of numerical rank-deficiency. We apply the derived theory to obtain new results for the discretization of univariate Fourier extension frames. In particular, we show that whereas uniform sampling is a suboptimal sampling strategy in exact arithmetic, it is near-optimal in finite precision.
\end{abstract} \vspace{2mm}


\section{Introduction}
We consider the least squares approximation of a function $f \colon X \to \mathbb{C}$, $X \subseteq \mathbb{R}^d$, in the linear span of a finite set of functions $\Phi = \{\phi_1, \phi_2, \dots, \phi_n\}$. That is, we seek coefficients $c = (c_1,\dots,c_n) \in \mathbb{C}^n$ minimizing
\[
    \norm{\sum_{i=1}^n c_i \phi_i - f}_{L^2_\rho(X)}
\]
where $\rho$ is a prescribed non-atomic probability measure on $X$. While such problems are well understood in exact arithmetic, practical computations are necessarily carried out in finite precision. The purpose of this paper is to provide a theoretical framework that describes how rounding errors modify the convergence behaviour and the discretization of such problems, and to exploit the resulting structure.

We work under the standard relative-error model of floating-point arithmetic. Let $\fl \colon \mathbb{C} \to \mathbb{C}$ denote rounding to a finite-precision number. There exists a unit roundoff $u>0$ such that
\begin{equation} \label{eq:introfl}
\forall y \in \mathbb{C},\; \exists \delta \in \mathbb{C} \text{ with } |\delta| \le u: \quad \fl(y) = y(1+\delta).
\end{equation}
Applying this model pointwise to function evaluations, we seek coefficients $c \in \mathbb{C}^n$ such that the resulting numerical error
\[
    \norm{\sum_{i=1}^n c_i \fl(\phi_i) - \fl(f)}_{L^2_\rho(X)}
\]
is small.

Although rounding of $f$ has a negligible effect on the error in the sense that
\[
\|f-\fl(f)\|_{L^2_\rho} \le u \|f\|_{L^2_\rho},
\]
rounding of $\Phi$ can have a pronounced effect. A central object in our analysis is the \emph{numerical span}, defined as
\[
\SPAN(\{\fl(\phi_i)\}_{i=1}^n),
\]
which represents the space that is effectively accessible in finite-precision arithmetic. When the functions in $\Phi$ are strongly non-orthogonal, the numerical span may differ substantially from the exact span. A geometric illustration is given in Figure~\ref{fig:intro}. 

\begin{figure}
    \centering
    \begin{subfigure}{.49\linewidth}
        \begin{tikzpicture}[tdplot_main_coords, scale=1.5]

        \tikzset{
        axis/.style={->, thick},
        planeXY/.style={fill=blue!25, opacity=0.6},
        planeTilt/.style={fill=red!25, opacity=0.6},
        edge/.style={thick},
        vector/.style={->, very thick, black, line width=1pt},
        point/.style={circle, fill=black, inner sep=1.5pt}
        }

        \draw[axis] (0,0,0) -- (3.5,0,0) node[anchor=north] {$x$};
        \draw[axis] (0,0,0) -- (0,3,0) node[anchor=north] {$y$};

        \coordinate (A) at (0,0,0);
        \coordinate (B) at (3,0,0);
        \coordinate (C) at (3,2.5,0);
        \coordinate (D) at (0,2.5,0);

        \fill[planeXY] (A) -- (B) -- (C) -- (D) -- cycle;
        \draw[edge]    (A) -- (B) -- (C) -- (D) -- cycle;

        \node[rotate=-10] at (0,1.5,0.2) {$\SPAN(\{\phi_1,\phi_2\})$};

        \draw[vector] (0,0,0) -- (1.8,0,0) node[below] {$\phi_1$};
        \draw[vector] (0,0,0) -- (1.8,0.1,0) node[below right] {$\phi_2$};

        \node[point] (P1) at (2.6,0,0) {};
        \node[above left=0pt of P1] {$f_1$};

        \node[point] (P2) at (1,2,0) {};
        \node[left=0pt of P2] {$f_2$};

        \end{tikzpicture}
        \subcaption{Exact arithmetic}
    \end{subfigure}
    \begin{subfigure}{.49\linewidth}
        \begin{tikzpicture}[tdplot_main_coords, scale=1.5]

        \tikzset{
        axis/.style={->, thick},
        planeXY/.style={fill=blue!25, opacity=0.6},
        planeTilt/.style={fill=red!25, opacity=0.6},
        edge/.style={thick},
        vector/.style={->, very thick, black, line width=1pt},
        point/.style={circle, fill=black, inner sep=1.5pt}
        }

        \draw[axis] (0,0,0) -- (3.5,0,0) node[anchor=north] {$x$};
        \draw[axis] (0,0,0) -- (0,3,0) node[anchor=north] {$y$};
        \draw[axis] (0,0,0) -- (0,0,2) node[anchor=north east, align=center, xshift=-4pt, yshift=10pt] {\footnotesize new direction \\[-4pt] \footnotesize introduced by rounding};

        \coordinate (A) at (0,0,0);
        \coordinate (B) at (3,0,0);
        \coordinate (C) at (3,2.5,0);
        \coordinate (D) at (0,2.5,0);

        \fill[planeXY] (A) -- (B) -- (C) -- (D) -- cycle;
        \draw[edge]    (A) -- (B) -- (C) -- (D) -- cycle;

        \coordinate (E) at (0,0,0);
        \coordinate (F) at (3,0,0);  
        \coordinate (G) at (3,3,2.5);  
        \coordinate (H) at (0,3,2.5);

        \fill[planeTilt] (E) -- (F) -- (G) -- (H) -- cycle;
        \draw[edge]      (E) -- (F) -- (G) -- (H) -- cycle;

        \node[rotate=18] at (0,1.7,2.15) {$\SPAN(\{\fl(\phi_1),\fl(\phi_2)\})$};

        \node[point] (P1) at (2.6,0,0) {};
        \node[above left=2pt of P1] {$\fl(f_1)$};

        \node[point] (P2) at (1,2,0) {};
        \node[left=0pt of P2] {$\fl(f_2)$};

        \draw[vector] (0,0,0) -- (1.82,0,0.035) node[below, yshift=-1pt, xshift=1pt] {$\fl(\phi_1)$};
        \draw[vector] (0,0,0) -- (1.8,0.08,0.1) node[above right, xshift=2pt, yshift=4pt] {$\fl(\phi_2)$};

        \node[point, inner sep=0.4pt] (P3) at (1,1.18,0.98) {};
        \draw[dashed, line width=1pt] (P2) -- (P3) node[anchor=west, align=center,xshift=14pt, yshift=-14pt] {\footnotesize best numerical \\[-4pt] \footnotesize approximation error};
        \def\ra{0.1}
        \draw ($(P3) + (-\ra-0.25,0,0)$) -- ($(P3) + (-0.15,0,0)$) -- ($(P3) + (-0.15,\ra,-1.2*\ra)$);

        \end{tikzpicture}
        \subcaption{Finite-precision arithmetic}
    \end{subfigure}
    \caption{Geometric illustration of the difference between the exact and the numerical span. The basis $\Phi = \{\phi_1, \phi_2\}$ spans the $xy$-plane. However, the $y$-direction is only weakly represented in $\Phi$ and, hence, finite-precision rounding perturbs this direction significantly, while directions that are strongly represented—the $x$-direction—are essentially unaffected. As a result, the numerical span deviates substantially from the exact span; it contains new directions\protect\footnotemark that mostly consist of noise and do not contribute significantly to function approximation. It follows that numerical approximation of functions with a significant component along the $y$-direction, such as $f_2$, is poor. Weak representation corresponds to the need for large expansion coefficients; for example, $f_2 = c_1 \phi_1 + c_2 \phi_2$, where $\sqrt{c_1^2 + c_2^2}$ is large. This connection is central to our analysis of the numerical span and explains why $\ell^2$-regularization of the expansion coefficients damps rounding-sensitive directions.}
    \label{fig:intro}
\end{figure}
\footnotetext{Because the rounding of well-represented directions as well as of the functions $f_1$ and $f_2$ is negligible, only one new direction is drawn here for simplicity of exposition.}

Rounding sensitivity is related to representation strength, which can be characterized using the synthesis operator
\[
\mathcal{T} \from \mathbb{C}^n \to L^2(X,\rho), \qquad 
c \mapsto \sum_{i=1}^n c_i \phi_i,
\]
and its singular value decomposition
\begin{equation} \label{eq:introsvd}
\mathcal{T} = \sum_{i=1}^{\hat{n}} \sigma_i u_i v_i^*,
\end{equation}
where $\hat{n}$ denotes the dimension of $\SPAN(\Phi)$. The singular values $\sigma_i$ quantify how strongly the directions $u_i \in \SPAN(\Phi)$ are represented by $\Phi$. Directions with small singular values are weakly represented; they are associated with large expansion coefficients since
\[
    u_i = \mathcal{T}c \quad \Rightarrow \quad \|c\|_2 \geq 1 / \sigma_i.
\]
This makes them sensitive to finite-precision perturbations. In analogy with the matrix setting, we say that $\Phi$ is \emph{numerically rank-deficient} if the condition number of its synthesis operator satisfies
\[
\kappa(\mathcal{T}) = \frac{\sigma_1}{\sigma_{\hat{n}}} \geq \frac{1}{u},
\]
in which case at least one direction in $\SPAN(\Phi)$ is substantially perturbed by relative rounding of the order of $u$.

For a broad class of functions $\Phi$, we develop a characterization of the numerical span that leads to error bounds of the form
\begin{equation} \label{eq:introbound}
\inf_{v \in \SPAN(\{\fl(\phi_i)\}_{i=1}^n)} \|v - f\|_{L^2_\rho}
\;\asymp\;
\inf_{c \in \mathbb{C}^n}
\left( \|\mathcal{T}c - f\|_{L^2_\rho} + C u \|c\|_2 \right)
\end{equation}
with $C>0$, where the upper bound holds deterministically and the lower bound holds in expectation under suitable stochastic assumptions. These bounds formalize the phenomenon depicted in Figure~\ref{fig:intro}: functions with components along weakly represented directions experience large numerical errors, and these directions are associated with a need for large expansion coefficients. Whereas the upper bound in~\eqref{eq:introbound} is a consequence of~\eqref{eq:introfl}, novel and much more advanced modeling techniques are required to obtain a meaningful lower bound. An important consequence of the lower bound is that numerical orthogonalization, i.e., recombination of the set $\{\fl(\phi_i)\}_{i=1}^n$, does not mitigate the effects of finite-precision arithmetic described in this paper.

A natural approach when approximating in the presence of numerical rank-deficiency is to introduce $\ell^2$-regularization, which penalizes large coefficients and, hence, suppresses weakly represented directions. Leveraging our characterization of the numerical span, we show that $\ell^2$-regularization on the order of the unit roundoff yields near-best numerical errors. Moreover, regularization is indispensable in practice, as it prevents the amplification of rounding errors during the computation of the expansion coefficients. Together, these results provide a theoretical justification for the use of regularization.

Regularization has further fundamental implications for discretization: it relaxes the conditions required for accurate approximation from sampled data. Specifically, we show that accurate discrete approximation is guaranteed whenever the data seminorm 
\[ 
    \|\cdot\|_\mathcal{M} = \|\mathcal{M}\cdot\|_2
\] 
induced by a sampling operator $\mathcal{M}$, satisfies
\[
    \gamma \left( \|\mathcal{T}c \|_{L^2_\rho} + \epsilon \|c\|_2 \right) \leq \|\mathcal{T}c\|_\mathcal{M} + \epsilon \|c\|_2, \qquad \forall c \in \mathbb{C}^n
\]
for some $0 < \gamma \leq 1$, where $\epsilon$ is a regularization parameter of the order of the unit roundoff $u$. In exact arithmetic, the same norm inequality is required with $\epsilon = 0$. The regularized inequality reflects representation strength; for instance, when $\gamma = 1/2$, it is automatically satisfied for coefficient vectors $v_i$ associated with small singular values $\sigma_i \leq \epsilon$. Thus, in the presence of numerical rank-deficiency, near-best numerical approximations may be obtained from fewer measurements of $f$ than would be required in exact arithmetic.

In the light of this relaxation, we analyze the interaction between regularization and randomized discretization strategies, focusing on Christoffel sampling~\cite{cohen2017,adcockOptimalSamplingLeastsquares2024}. Motivated by related results in the fully discrete setting, we introduce a regularized inverse Christoffel function
\[
    k^\epsilon(x) = \sum_{i=1}^{\hat{n}} \frac{\sigma_i^2}{\sigma_i^2 + \epsilon^2} \lvert u_i(x) \rvert^2,
\]
where $\sigma_i$ and $u_i$ are defined in~\eqref{eq:introsvd}. In contrast to analyses in exact arithmetic based on the inverse Christoffel function $k$ (corresponding to $\epsilon = 0$ in the definition above), the regularized formulation again incorporates representation strength. We find that this has a significant influence on the pointwise behaviour of $k$ versus $k^\epsilon$. Since discretization conditions depend on quantities such as $\|k^\epsilon\|_{L^\infty}$ relative to the chosen sampling density, these pointwise differences can significantly affect the number of samples required for accurate numerical approximation. Beyond yielding new theoretical insight, regularization also renders $k^\epsilon$ computable in finite precision, whereas the numerical evaluation of $k$ may require arbitrarily high precision in the presence of numerical rank-deficiency.

The regularized Christoffel function naturally gives rise to an effective dimension
\[
    \hat{n}^\epsilon = \int_X k^\epsilon \, d\rho = \sum_{i=1}^{\hat{n}} \frac{\sigma_i^2}{\sigma_i^2 + \epsilon^2} \leq \hat{n}.
\] 
In our context, this (non-integer) quantity measures the number of directions in the numerical span that remain significant after finite-precision rounding. We obtain sampling bounds for numerical weighted least squares approximation that depend on $\hat{n}^\epsilon$ rather than $\hat{n}$, establishing a precise link between representation strength and required data.

Finally, we consider randomized discretization of Fourier extension frames, which are used for the approximation of smooth functions on irregular domains. We derive new results for both weighted and unweighted least squares approximation in the univariate case. In particular, we characterize how the effective dimension $\hat{n}^\epsilon$ depends on the extension parameter, and we show that uniform sampling is a near-optimal sampling strategy for numerical approximation, even though it is suboptimal for approximation in exact arithmetic.

\subsection{Related work and sources of numerical rank-deficiency}
The impact of numerical rank-deficiency on the regularized approximation error is investigated in~\cite{fna1,adcockFramesNumericalApproximation2020} in the specific context of frames. In particular, it is shown that if one truncates an infinite-dimensional frame, the finite subsequence becomes numerically rank-deficient for large enough $n$. The truncated singular value decomposition (TSVD) is proposed as a model for practical least squares solvers, though without theoretical justification. The accuracy of TSVD-based approximations is characterized in~\cite[Theorem 1.3]{adcockFramesNumericalApproximation2020} and depends on a regularization parameter, on the existence of accurate approximations with small expansion coefficients, and on two constants, $\kappa_{M,N}^\epsilon$ and $\lambda_{M,N}^\epsilon$, which quantify the influence of discretization and regularization. While these constants play a central role in the analysis, their dependence on representation strength is non-transparent. In the current work, the effects of finite-precision are made explicit, resulting in a full analysis in a more general setting.

A well-studied setting in which numerical rank-deficiency arises is function approximation on irregular domains, where orthonormal or Riesz bases are not readily available. A standard strategy in such settings is to restrict an orthonormal basis defined on a surrounding, regular domain to the target domain, which introduces redundancy and results in a numerically rank-deficient basis~\cite{matthysenFunctionApproximationArbitrary2018a}. Note that the target domain may sometimes be unknown, as is often the case for high-dimensional approximation methods arising in uncertainty quantification~\cite{adcock2020approximating}. Other settings in which numerical rank-deficiency has been analyzed (to some extent) include settings where bases are combined, enriched or weighted to better capture specific features of the function being approximated~\cite{herremansEfficientFunctionApproximation2024,papadopoulos2025frame}, radial basis function approximation~\cite{PiretFrames}, rational approximation~\cite{herremansResolutionSingularitiesRational2023}, Trefftz methods~\cite{huybrechs2019oversampled}, (deep) neural networks~\cite{adcock2022cas4dl} and the Method of Fundamental Solutions~\cite{barnett2008stability}. 

While the examples above concern settings in which numerical rank-deficiency is explicitly identified, related phenomena arise in many other numerical methods. In these settings, the analysis developed in this paper may provide a useful perspective and the proposed rounding-aware sampling strategies could potentially be beneficial. Among others, ill-conditioned representations occur in reduced basis methods for parameterized PDEs~\cite{quarteroni2015reduced}; overparameterized learning problems with highly correlated features~\cite{hastie2022surprises}; radial basis functions and kernel methods~\cite{fornberg2008stable}; finite element methods based on enrichment or locally adapted approximation spaces, such as enriched, multiscale and cut-cell FEM~\cite{babuvska1997partition,efendiev2013generalized,burman2025cut}; boundary integral and potential methods~\cite{kress1989linear}; moment problems~\cite{gautschiOrthogonalPolynomialsComputation2004}; quantum chemistry using nearly linearly dependent atomic orbital bases~\cite{lehtola2019review}; and redundant representation in signal processing~\cite{mallat1999wavelet}.

\subsection{Paper overview}
After introducing notation in Section~\ref{sec:notation}, the paper is organized into two main parts. The first develops the theoretical framework for approximation in finite-precision arithmetic, while the second applies this framework to the analysis of randomized sampling strategies.

The first part comprises Sections~\ref{sec:bestapprox}–\ref{sec:numericalexamples}. The core theory is presented in Sections~\ref{sec:bestapprox} and~\ref{sec:discretization}, of which an overview is provided in Figure~\ref{fig:diagram}. In Section~\ref{sec:bestapprox}, we introduce the notions of numerically rank-deficient bases and near-best numerical approximation. We show in particular that regularized approximation achieves near-best numerical accuracy. Section~\ref{sec:discretization} establishes complementary benefits of regularization, namely that it relaxes the conditions required for accurate discrete approximation and that it suppresses errors arising in the computation of expansion coefficients. Section~\ref{sec:numericalexamples} illustrates these concepts through numerical experiments.

The second part begins with Section~\ref{sec:randomizedsampling}, where we analyze randomized sampling through a regularized variant of the Christoffel function. We discuss the influence of regularization on randomized sampling and prove that regularization renders the Christoffel function computable under backward stable computation. Finally, in Section~\ref{sec:fourierextension}, we apply this theory to derive new results on the random discretization of the univariate Fourier extension frame.

\tikzset{
  every picture/.style={line width=0.75pt},
  box/.style={
    draw,
    minimum height=2.5em,
    minimum width=10em,
    align=center,
    inner sep=0.6em
  },
  arrow/.style={->, >=Stealth}
}
\begin{figure}
    \centering
    \resizebox{.92\textwidth}{!}{%
    \begin{tikzpicture}[
    node distance=.5cm and 1cm
    ]

    \node[box] (L1) {Upper numerical error bound \\ (Prop.~\ref{prop:continuous-upper})};
    \node[box, below=of L1] (L2) {Lower numerical error bound \\(Thm.~\ref{thm:lowerbound})};

    \node[box, right=of L1, yshift=-0.85cm] (M1) {Near-best numerical approximation \\ (Def.~\ref{def:nearbest})};
    \node[box, below=1.7cm of M1, xshift=0cm] (M2) {\textbf{$\boldsymbol{\ell^2}$-regularized approximation}};

    \node[box, below=1cm of M2, xshift=-3.8cm] (B1) {Relaxed discretization condition \\ (Thm.~\ref{thm:discreteerrorbound})};
    \node[box, below=1cm of M2, xshift=3.2cm] (B2) {Accuracy under backward stable computation \\ (Thm.~\ref{thm:backwardstab2})};

    \draw[arrow] (L1.east) -- (M1.west);
    \draw[arrow] (L2.east) -- (M1.west);

    \draw[-] ($ (M1.south) + (0,0) $) -- ($ (M2.north) + (0,0) $) node[midway, right, xshift=0.1cm] (T1) {(Thm.~\ref{thm:regnearbest})};;

    \draw[arrow] ($ (M2.south) + (-1cm,0) $) -- ($ (B1.north) + (1cm,0) $);
    \draw[arrow] ($ (M2.south) + (1cm,0) $) -- ($ (B2.north) + (-0.4cm,0) $);

    \coordinate (line-text-right) at ($ (M1.south) + (1.5cm,-1.1cm) $);

    \node[draw, dashed, inner sep=1em, line width=0.1pt, fit=(L1) (L2) (M1) (line-text-right) (M2) (T1)] (upperbox) {};
    \node[draw, dashed, inner sep=1em, line width=0.1pt, fit=(B1) (B2) (M2)] (lowerbox) {};

    \node[anchor=north east, font=\small] at (upperbox.north east) {\textit{Section~\ref{sec:bestapprox}: $L^2_\rho$-approximation}};

    \node[anchor=north east, font=\small, align=right] at (lowerbox.north east) {\textit{Section~\ref{sec:discretization}: Discrete} \\ \textit{approximation}};
    \end{tikzpicture}%
    }
    \caption{Logical dependencies among the main results of Part I.}
    \label{fig:diagram}
\end{figure}

\section{Notation} \label{sec:notation}
For vectors and matrices, $\|\cdot\|_2$ denotes the Euclidean and spectral norms, respectively, and $\|\cdot\|_F$ denotes the Frobenius norm. For Hermitian matrices $A$ and $B$, we write $A \preceq B$ to denote the Loewner order, meaning that $B-A$ is positive semidefinite. Throughout the paper, $X$ will be a measurable subset of $\mathbb{R}^d$, where $d \in \mathbb{N}_0$. For functions defined on $X$, $\|\cdot\|_{L^2_\rho}$ denotes the $L^2(X,\rho)$-norm with respect to a prescribed non-atomic probability measure $\rho$ on $X$ (so that $\rho(X) = 1$), with associated inner product
\[
    \langle f, g \rangle_{L^2_\rho} = \int_X \overline{f} \, g \, d\rho.
\]
Whenever $\rho$ is the Lebesgue measure, we write $L^2(X,\rho) = L^2(X)$ and $\|\cdot\|_{L^2_\rho} = \|\cdot\|_{L^2}$. 

We let $\Phi = \{\phi_1,\dots,\phi_n\} \subset L^2(X,\rho)$ be a finite set of functions. For brevity and convenience, we refer to $\Phi$ as a basis, but we don't assume linear independence. We denote by $\hat{n}$ the dimension of $\SPAN(\Phi)$, which satisfies $\hat n \le n$. Associated to $\Phi$ is the finite-rank, bounded, linear \textit{synthesis operator}
\[
    \mathcal{T} \from \mathbb{C}^n \to L^2(X,\rho), \qquad c \mapsto \sum_{i=1}^n c_i \phi_i.
\]
We define
\[
    \|\mathcal{T}\|_{2,L^2_\rho} = \sup_{\substack{c \in \mathbb{C}^n \\ c \neq 0}} \frac{\|\mathcal{T}c\|_{L^2_\rho}}{\|c\|_2}.
\]
The symbol $^\dagger$ denotes the Moore-Penrose inverse and we write
\[
\kappa(A) = \|A\|_2\|A^\dagger\|_2,
\qquad
\kappa(\mathcal{T}) = \|\mathcal{T}\|_{2,L^2_\rho}\|\mathcal{T}^\dagger\|_{2,L^2_\rho}.
\]
We refer to the matrix $G \in \mathbb{C}^{n \times n}$ as the Gram matrix representing the operator $\mathcal{T}^*\mathcal{T}$, so that 
\[
    (G)_{i,j} = \langle \phi_i, \phi_j \rangle_{L^2_\rho}.
\]

Throughout, we denote by $\fl \from \mathbb{C} \to \mathbb{C}$ rounding to floating-point arithmetic, satisfying
\begin{equation}\label{eq:rounding}
    \forall y \in \mathbb{C},\; \exists \delta \in \mathbb{C} \text{ with } |\delta| \le u: \quad \fl(y) = y(1+\delta),
\end{equation}
where $u > 0$ is the unit roundoff. For simplicity, we neglect errors due to under- or overflow. If other assumptions are made about the operator $\fl$, these will be made explicit. For a given matrix $A \in \mathbb{C}^{m \times n}$ and vector $b \in \mathbb{C}^n$, by $\fl(A)$ and $\fl(b)$ we mean the matrix and vector obtained by applying $\fl(\cdot)$ entry-wise. For a given function $v \in L^2(X,\rho)$, finite-precision evaluation is denoted by $\fl(v) = \fl \circ \, v$. Note that due to~\eqref{eq:rounding}, $\fl(v)\in L^2(X,\rho)$ as well. Furthermore, we define the \textit{numerical synthesis operator} associated with $\Phi$ as
\[
    \fl(\mathcal{T}) \from
    \mathbb{C}^n \to L^2(X,\rho),
    \qquad c \mapsto \sum_{i=1}^n c_i \fl(\phi_i).
\]
Its range equals $\SPAN(\{\fl(\phi_i)\}_{i=1}^n)$ and is called the \textit{numerical span} of $\Phi$\footnote{One might argue that any definition of numerical span should also account for the finite-precision representation of the expansion coefficients. However, this has no significant effect, as it merely induces a slight renormalization of the functions in $\Phi$.}.

\section*{Part I: Approximation in finite precision}

\section{Characterization of (near-)best numerical approximation} \label{sec:bestapprox} 
In this section, we analyze the properties of the numerical span in order to define the notion of near-best numerical approximation. The derived bounds show that the norm of the expansion coefficients plays a crucial role in the analysis of the numerical error. While upper bounding the numerical error is straightforward using a standard model for finite-precision rounding, obtaining a lower bound is much more involved. We conclude the section by showing that regularized approximation yields near-best numerical errors, which theoretically justifies its use in the presence of numerical rank-deficiency.

\subsection{Upper bound on the numerical error}\label{sec:upperbound}
As a reference and for later use, we first consider a fully discrete setting. Let $A \in \mathbb{C}^{m \times n}$ and $b \in \mathbb{C}^m$. For any coefficient vector $c \in \mathbb{C}^n$, the numerical residual can be decomposed as
\[
    \underbrace{\fl(A)c - \fl(b)}_{\text{numerical residual}}
    =
    \underbrace{Ac - b\vphantom{\fl(A)}}_{\text{residual}}
    +
    \underbrace{(\fl(A)-A)c}_{\text{rounding of $A$}}
    +
    \underbrace{\fl(b)-b}_{\text{rounding of $b$}}.
\]
This decomposition separates the error that comes from the approximation problem itself and the error that is purely due to finite-precision rounding of $A$ and $b$.

\begin{proposition}\label{prop:discrete-upper}
For any $A \in \mathbb{C}^{m \times n}$, $b \in \mathbb{C}^m$ and $c \in \mathbb{C}^n$, we have
\begin{equation}\label{eq:discrete-upper}
    \norm{\fl(A)c - \fl(b)}_2
    \le
    \norm{Ac - b}_2
    +
    u\left(\sqrt{n}\norm{A}_2\norm{c}_2 + \norm{b}_2\right).
\end{equation}
\end{proposition}

\begin{proof}
Using the triangle inequality on the decomposition above, we get
\[
    \norm{\fl(A)c - \fl(b)}_2 \le \norm{Ac - b}_2 + \norm{(\fl(A)-A)c}_2 + \norm{\fl(b)-b}_2.
\]
By the floating-point model \eqref{eq:rounding}, we have
$\norm{\fl(b)-b}_2 \le u \norm{b}_2$ and
$\norm{(\fl(A)-A)c}_2 \le \norm{\fl(A)-A}_F \norm{c}_2 \le u \|A\|_F \norm{c}_2 \le u \sqrt{n} \norm{A}_2 \norm{c}_2$,
which yields the stated bound \eqref{eq:discrete-upper}.
\end{proof}

A particularly important case is the minimum-norm least squares solution
$\hat{c} = A^\dagger b$. Using $\norm{A^\dagger b}_2 \le \norm{A^\dagger}_2 \norm{b}_2$, Prop.~\ref{prop:discrete-upper} implies
\begin{equation}\label{eq:discrete-upper-pinv}
    \norm{\fl(A)\hat{c} - \fl(b)}_2
    \le
    \underbrace{\norm{AA^\dagger b-b}_2}_{\text{best residual}}
    +
    u(\sqrt{n}\kappa(A)+1)\norm{b}_2.
\end{equation}
When $\kappa(A) \geq 1/u$, this bound no longer guarantees any accuracy of the numerical approximation; rounding errors completely perturb certain directions in the range of $A$. In low rank matrix theory, such matrices are said to be numerically rank-deficient~\cite{golubRosetakDocumentRank1977}.

\begin{definition} \label{def:numrankdefdisc}
A matrix $A$ is said to be \emph{numerically rank-deficient} if $\kappa(A) \ge 1/u$.
\end{definition}

We can formulate an analogous result without discretization. Namely, the numerical error admits the same decomposition as in the discrete case:
\[
    \underbrace{\fl(\mathcal{T})c - \fl(f)}_{\text{numerical error}}
    =
    \underbrace{\mathcal{T}c - f \vphantom{\fl(\mathcal{T})}}_{\text{approximation error}}
    +
    \underbrace{(\fl(\mathcal{T})-\mathcal{T})c}_{\text{rounding of $\Phi$}}
    +
    \underbrace{\fl(f)-f}_{\text{rounding of $f$}}.
\]

\begin{proposition}\label{prop:continuous-upper}
For any $\Phi \subset L^2(X,\rho)$, $f \in L^2(X,\rho)$ and $c \in \mathbb{C}^n$, we have
\begin{equation}\label{eq:continuous-upper}
    \norm{\fl(\mathcal{T})c - \fl(f)}_{L^2_\rho}
    \le
    \norm{\mathcal{T}c - f}_{L^2_\rho}
    +
    u\left(\sqrt{n}\norm{\mathcal{T}}_{2,L^2_\rho}\norm{c}_2 + \norm{f}_{L^2_\rho}\right).
\end{equation}
\end{proposition}
\begin{proof}
Applying the triangle inequality to the decomposition above, we obtain
\[
    \norm{\fl(\mathcal{T})c - \fl(f)}_{L^2_\rho} \le \norm{\mathcal{T}c - f}_{L^2_\rho} + \norm{(\fl(\mathcal{T})-\mathcal{T})c}_{L^2_\rho} + \norm{\fl(f)-f}_{L^2_\rho}.
\]
By the floating-point model \eqref{eq:rounding}, we have $\|\fl(f)-f\|_{L^2_\rho} \le u \|f\|_{L^2_\rho}$ and $$\|(\fl(\mathcal{T})-\mathcal{T})c\|_{L^2_\rho} \le \sqrt{n} u \|\mathcal{T}\|_{2,L^2_\rho} \|c\|_2,$$ which yields the bound~\eqref{eq:continuous-upper}.
\end{proof}
\begin{remark}
    The factor $\norm{\mathcal{T}}_{2,L^2_\rho}$ naturally appears in this bound and, as a consequence, in other bounds throughout the paper. Note that it can easily be controlled, that is, it is good practice to use normalized basis functions $\|\phi_i\|_{L^2_\rho} = 1$ so that $\norm{\mathcal{T}}_{2,L^2_\rho} = 1$.
\end{remark}

For the minimum-norm least squares solution $\hat{c} = \mathcal{T}^\dagger f$, we obtain
\begin{equation}\label{eq:continuous-upper-pinv}
    \norm{\fl(\mathcal{T})\hat{c} - \fl(f)}_{L^2_\rho}
    \le
    \underbrace{\norm{\mathcal{T}\mathcal{T}^\dagger f-f}_{L^2_\rho}}_{\text{best approximation error}}
    +
    u(\sqrt{n}\kappa(\mathcal{T})+1)\norm{f}_{L^2_\rho}.
\end{equation}
Hence, for orthonormal bases where $\kappa(\mathcal{T}) = 1$, the numerical error is close to the best approximation error. However, this is not the case when $\kappa(\mathcal{T})$ is large. Similarly to the discrete case, we define:

\begin{definition} \label{def:numrankdef}
A set of functions $\Phi = \{\phi_i\}_{i=1}^n$ is said to be \emph{numerically rank-deficient} if its synthesis operator satisfies $\kappa(\mathcal{T}) \ge 1/u$.
\end{definition}

\subsection{Lower bound on the numerical error using a stochastic model}\label{sec:lowerbound}
In order to obtain a lower bound on the numerical error, additional information on the rounding errors is required beyond that provided by~\eqref{eq:rounding}. Indeed, if all rounding errors were zero, the numerical error would coincide exactly with the approximation error. In this subsection, we introduce two ingredients that together enable the derivation of such a lower bound.
\begin{itemize}
    \item We develop a novel analysis to derive a lower bound on the perturbation of each basis function, assuming continuity and a nontrivial range.
    \item We model the rounding errors associated with different basis functions as uncorrelated random variables, allowing us to lower bound the expected error in the numerical span.
\end{itemize}
Importantly, the numerical span is invariant under changes of basis by linear recombination of the rounded functions, such as numerical orthogonalization. Consequently, the derived lower bound has implications for a broad class of methods. By assuming that an orthonormal basis of the exact span is available or can be computed, this point is frequently overlooked in the literature.

\subsubsection{Lower bound on the perturbation of a function}

Since any set of floating-point numbers is discrete, continuously varying functions are necessarily perturbed when represented in floating-point arithmetic. This can be formalized into a lower bound on the perturbation of a function. For simplicity, we consider real-valued univariate functions and restrict to the Lebesgue measure, i.e., $L^2(X)$-bounds. However, similar ideas hold for complex-valued, multivariate functions and other non-atomic measures $\rho$.

In what follows, we'll require that a bound like~\eqref{eq:deterlowerboundpert} holds for any basis function $\phi_i \in \Phi$ with a constant $C > 0$ independent of $i$. Moreover, in order to characterize convergence in a sequence of approximation spaces, the same should hold for an infinite sequence of functions $\{\phi_i\}_{i=1}^\infty$. A key ingredient in obtaining such a bound for a broad class of functions is the use of a non-injective change-of-variables~\cite[Theorem 1.16-2]{ciarlet2025linear}, which takes into account the cardinality of the range of $\phi_i$.

We restrict to round-to-nearest floating point systems. That is, there exists a finite set of representables $\mathbb{F} \subset \mathbb{R}$ such that $\forall r \in \mathbb{F}$:
\begin{equation} \label{eq:roundtonearest}
    \forall y \in I_r = \{ y \, : \, \lvert y - r \rvert < u \lvert r \rvert \}: \qquad \fl(y) = r
\end{equation}
where $u > 0$ is the unit roundoff. Restricting to a round-to-nearest strategy entails no loss of generality, since other rounding schemes, such as directed or stochastic rounding, necessarily introduce larger rounding errors. Note that the lower bound derived in the theorem below holds irrespective of the rounding strategy used at the borders of the cells $I_r$. Furthermore, under- and overflow is implicitly taken into account in this model, as we do not require that $\bigcup_{r \in \mathbb{F}} \overline{I_r} = \mathbb{R}$.

\begin{theorem} \label{thm:determlowerbound}
Let $\phi \in C^1(X)$, where $X \subseteq \mathbb{R}$, and let $\mathrm{fl}:\mathbb{R}\to\mathbb{R}$ be a floating-point format that satisfies~\eqref{eq:roundtonearest}. Assume that there exists an open and connected interval $J$ such that
\[
    J \subseteq \phi(X) \cap \bigcup_{r \in \mathbb{F}} \overline{I_r}.
\]
Furthermore, assume that $m=\inf_{y\in J}|y|>0$, $\lvert J \rvert \geq 4 u M$ with $M=\sup_{y\in J}|y|$ and $L = \sup_{x \in \phi^{-1}(J)} \lvert \phi'(x)\rvert < \infty$. Then, there exists a constant $C > 0$ (independent of $u$) such that \sloppy
\begin{equation} \label{eq:deterlowerboundpert}
    \norm{ \fl(\phi)-\phi}_{L^2(X)}^2 \ge \norm{ \fl(\phi)-\phi}_{L^2(\phi^{-1}(J))}^2 \geq C u^2,
\end{equation}
where $\phi^{-1}(\cdot)$ denotes the pre-image. In particular, one may take 
\[
    C = \frac{\lvert J \rvert}{6 L}\,\frac{m^{3}}{M} \min_{y \in J} \text{card}(\phi^{-1}(y))
\]
where $\text{card}(\phi^{-1}(y))$ is the cardinal of the set $\phi^{-1}(y)$.
\end{theorem}
\begin{proof}
Fix a full cell $I_r\subset J$ and define $\Delta_r = u \lvert r \rvert$ and
$A_r=\phi^{-1}(I_r)$. By~\eqref{eq:roundtonearest}, $\mathrm{fl}(\phi(x))=r$ on $A_r$, so
\[
    \int_{A_r}\bigl(\mathrm{fl}(\phi(x))-\phi(x)\bigr)^2\,dx=\int_{A_r}(r-\phi(x))^2\,dx \geq \frac{1}{L}\int_{A_r} (r-\phi(x))^2 \lvert \phi'(x)\rvert \,dx,
\]
where the last inequality is due to $|\phi'(x)|\le L, x\in\phi^{-1}(J)$. Applying~\cite[Theorem 1.16-2]{ciarlet2025linear} gives
\[
    \frac{1}{L} \int_{A_r}(r-\phi(x))^2\,|\phi'(x)|\,dx\ = \frac{1}{L} \int_{I_r}(r-y)^2 \text{card}(\phi^{-1}(y)) \,dy \geq \frac{\min_{y \in J} \text{card}(\phi^{-1}(y))}{L} \int_{I_r}(r-y)^2 \,dy.
\]
The integral equals
\[
    \int_{I_r}(r-y)^2 \,dy = \int_{-\Delta_r}^{\Delta_r} z^2 \, dz = \frac{2}{3} \Delta_r^3 \geq \frac{2}{3} u^3 m^3,
\]
using the fact that $\Delta_r = u \lvert r \rvert \geq u m$ for all cells $I_r \subset J$.
Moreover, since the width of each cell $I_r \subset J$ satisfies $2\Delta_r \leq 2uM$, we know that $\Big\lfloor\frac{\lvert J \rvert}{2u\,M}\Big\rfloor-1\ \ge\ \frac{\lvert J \rvert}{4u\,M}$ (since $\lvert J \rvert\ge 4uM$) is a lower bound on the number of full cells in $J$. Summing over all cells, we obtain
\[
    \int_{X}\bigl(\mathrm{fl}(\phi)-\phi\bigr)^2\,dx \geq \frac{\lvert J \rvert}{4uM}\frac{\min_{y \in J} \text{card}(\phi^{-1}(y))}{L} \frac{2}{3}u^3m^3 = \frac{\lvert J \rvert}{6L}\frac{m^{3}}{M} \min_{y \in J} \text{card}(\phi^{-1}(y)) u^{2}.\qedhere
\]
\end{proof}

This theorem provides an explicit lower bound on the magnitude of the perturbation in $L^2(X)$ for a broad class of functions. For example, consider $\phi(x) = x$ on $X = [0,1]$, which varies continuously from $0$ to $1$ over its domain. For the choice $J = (0.5,1) \subset \phi([0,1])$, the bound yields 
\[
    \| \mathrm{fl}(\phi)-\phi\|_{L^2(X)}^2 \ge u^2/96.
\]
Consider now $\phi_k(x) = \cos(k \pi x)$ on $X = [0,1]$ for $k = 1,2,\dots$ and both $J^- = (-1,-0.5)$ and $J^+ = (0.5,1)$. We get 
\begin{align*}
    \| \mathrm{fl}(\phi)-\phi\|_{L^2(X)}^2 &\ge \| \mathrm{fl}(\phi)-\phi\|_{L^2(\phi_k^{-1}(J^-))}^2 + \| \mathrm{fl}(\phi)-\phi\|^2_{L^2(\phi_k^{-1}(J^+))} \\ &\geq u^2 /(96\pi)
\end{align*}
using $L \leq k \pi$ and $\min_{y \in J^-} \text{card}(\phi^{-1}(y)) + \min_{y \in J^+} \text{card}(\phi^{-1}(y)) = k$. As a result, we obtain a lower bound that is independent of $k$.

An important exception is that of functions with a trivial range. For instance, for $\phi(x) = 1$ the perturbation is identically zero. In contrast, although not covered by the theorem, the perturbation associated with $\phi(x) = \sqrt{2}$ is typically nonzero. Nevertheless, we expect the contribution of one such basis function to be negligible and the effects described in the next section to remain relevant when other basis functions are perturbed.

The bound applies only to one-dimensional real-valued differentiable functions and the unweighted least squares norm. However, we hope that the central message of the derivation is clear: numerical errors are unavoidable when a continuously varying function is rounded to the discrete set of finite-precision numbers.

\subsubsection{Probabilistic model for approximation in the numerical span}\label{sec:probabilisticmodel}

In this section, we study the approximation power of linear combinations of perturbed basis functions, i.e., the numerical span. Even though we have established that the perturbations are nonzero for a broad class of functions, constructing a realistic model for approximation in the numerical span is subtle. Deterministically, one cannot exclude pathological scenarios; for example, the perturbations of the basis functions may accidentally lie in the exact span, so that the span is invariant under rounding. In an even more extreme case, the perturbations could improve the approximation of the target function, leading to a numerical error smaller than the approximation error.

In practice, such situations are highly unlikely, as they would require that the rounding perturbations associated with distinct basis functions are correlated with one another or with the target function. To model this lack of correlation, we adopt a stochastic framework with minimal assumptions. Note that a similar stochastic approach to modeling finite-precision arithmetic has appeared in~\cite{higham2019new}. In what follows, each random function $\delta_i$ is a model for the behaviour of the (deterministic) relative perturbation of $\phi_i$ in finite-precision arithmetic.

\begin{model}[Stochastic perturbation model]\label{model}
Let $(\Omega,\mathcal F,\mathbb P)$ be a probability space and 
let $\{\delta_i\}_{i}$ be a countable set of random functions $\delta_i \in L^2(\Omega;L^\infty(X))$\footnote{This means that each $\delta_i : \Omega \to L^\infty(X)$ is Bochner square-integrable, i.e., $\EXP\left[ \norm{\delta_i}_{L^\infty}^2 \right] < \infty$. Throughout the paper we suppress the dependence on $\omega\in\Omega$ and write $\delta_i(x)$ instead of $\delta_i(\omega)(x)$.}.
\begin{enumerate}
    \item \textbf{Zero mean.} For every $x\in X$ and every $i$,
    \[
        \EXP\left[\delta_i(x)\right] = 0.
    \]

    \item \textbf{Uncorrelatedness.} For every $x\in X$ and all $i\neq j$,
    \[
        \EXP\left[\overline{\delta_i(x)}\,\delta_j(x)\right] = 0.
    \]
\end{enumerate}
\end{model}

\begin{remark}
    Our setting differs from that of~\cite{higham2019new} in several ways. First, the role of stochasticity is different: in~\cite{higham2019new} it is introduced to sharpen constants in upper bounds, whereas in our case it is essential for obtaining a meaningful lower bound. Second, the analysis in~\cite{higham2019new} models rounding errors arising from operations on fully discrete quantities, while we consider rounding of functions. Finally,~\cite[Model 2.1]{higham2019new} assumes independence of all rounding errors, whereas we only require that perturbations associated with distinct functions are uncorrelated, which is a very mild assumption.
\end{remark}

The following theorem uses this model to derive a lower bound on the expected error in the numerical span. The numerical error is expected to be larger than the approximation error by a term proportional to the size of the expansion coefficients and the unit roundoff. This mirrors the upper bound of Prop.~\ref{prop:continuous-upper}. For simplicity, and in line with the derivations in \S\ref{sec:upperbound}, we neglect rounding of $f$ as its influence is insignificant. Observe that condition~\eqref{eq:expcondition} reflects the deterministic lower bound of Thm.~\ref{thm:determlowerbound}.
\begin{theorem} \label{thm:lowerbound}
    Let $\{\delta_i\}_{i=1}^n$ satisfy the conditions of Model~\ref{model} and define $\widetilde{\Phi} = \{\widetilde{\phi}_i\}_{i=1}^n$ via
    \[
        \widetilde{\phi}_i = \phi_i (1+\delta_i), 
    \]
    where $\Phi = \{\phi_i\}_{i=1}^n \subset L^2(X,\rho)$. Furthermore, assume that there exists a constant $C > 0$ such that
    \begin{equation} \label{eq:expcondition}
        \EXP \left[ \norm{\widetilde{\phi_i} - \phi_i}^2_{L^2_\rho} \right] \geq C u^2, \qquad 1 \leq i \leq n.
    \end{equation}
    Then, for every $c \in \mathbb{C}^n$ and $f \in L^2(X,\rho)$, it holds that
    \begin{equation} \label{eq:thm8}
        \EXP\left[\|\widetilde{\mathcal{T}}c - f\|_{L^2_\rho}^2\right] \geq \|\mathcal{T}c-f\|_{L^2_\rho}^2 + C u^2 \|c\|^2_2,
    \end{equation}
    where $\widetilde{\mathcal{T}} \, : \, c \mapsto \sum_{i=1}^n c_i \widetilde{\phi}_i$ is the synthesis operator associated with $\widetilde{\Phi}$.
\end{theorem}
\begin{proof}
    Due to $\delta_i \in L^2(\Omega,L^\infty(X))$ and $\phi_i \in L^2(X,\rho)$, it holds that $\widetilde{\mathcal{T}} \from \mathbb{C}^n \to L^2(\Omega, L^2(X,\rho))$, so that we can consider
    \begin{align*}
        \EXP\left[\|\widetilde{\mathcal{T}}c - f\|_{L^2_\rho}^2\right] &= \EXP\left[\|\mathcal{T}c-f + (\widetilde{\mathcal{T}} - \mathcal{T}) c \|_{L^2_\rho}^2\right] \\
        &= \EXP\left[\|\mathcal{T}c-f\|_{L^2_\rho}^2 + 2 \REAL \left(\langle \mathcal{T}c-f, (\widetilde{\mathcal{T}} - \mathcal{T}) c\rangle_{L^2_\rho}\right) + \|(\widetilde{\mathcal{T}} - \mathcal{T}) c \|_{L^2_\rho}^2 \right] \\
        &= \|\mathcal{T}c-f\|_{L^2_\rho}^2 + 2 \REAL \left(\langle \mathcal{T}c-f, \EXP\left[(\widetilde{\mathcal{T}} - \mathcal{T}) c \right] \rangle_{L^2_\rho}\right) + \EXP\left[\|(\widetilde{\mathcal{T}} - \mathcal{T}) c \|_{L^2_\rho}^2\right] \\
        &= \|\mathcal{T}c-f\|_{L^2_\rho}^2 + \EXP\left[\|(\widetilde{\mathcal{T}} - \mathcal{T}) c \|_{L^2_\rho}^2\right]
    \end{align*}
    where in the last step we use that $\EXP\left[(\widetilde{\mathcal{T}} - \mathcal{T}) c\right] = 0$ due to the first property of Model~\ref{model}. Furthermore, 
    \begin{align*}
        \EXP\left[\|(\widetilde{\mathcal{T}} - \mathcal{T}) c \|_{L^2_\rho}^2\right] &= \EXP\left[\int_X \left| \sum_{i=1}^n c_i \phi_i(x) \delta_i(x) \right|^2 d\rho \right] \\
        &= \EXP\left[\sum_{i=1}^n \lvert c_i\rvert^2 \int_X \lvert \phi_i(x) \delta_i(x)\rvert^2 d\rho + \sum_{i=1}^{n} \sum_{\substack{j=1\\j \neq i}}^{n} \, \overline{c_i} c_j \int_X \overline{\phi_i(x)} \overline{\delta_i(x)} \phi_j(x) \delta_j(x) d\rho \right] \\
        &= \sum_{i=1}^n \lvert c_i\rvert^2 \EXP\left[\int_X \lvert\phi_i(x) \delta_i(x)\rvert^2 d\rho\right] + \sum_{i=1}^{n} \sum_{\substack{j=1\\j \neq i}}^{n} \, \overline{c_i} c_j \int_X \overline{\phi_i(x)} \phi_j(x) \EXP\left[\overline{\delta_i(x)} \delta_j(x)\right] d\rho \\
        &\geq Cu^2\|c\|_2^2,
    \end{align*}
    due to the second property of Model~\ref{model} and due to~\eqref{eq:expcondition}.
\end{proof}

As before in Section~\ref{sec:upperbound}, of particular interest is the minimum-norm least squares solution $\hat{c} = \mathcal{T}^\dagger f$:
\begin{equation} \label{eq:numerrorbestapprox}
    \EXP(\|\widetilde{\mathcal{T}}\hat{c} - f\|_{L^2_\rho}^2) \geq \underbrace{\|\mathcal{T}\mathcal{T}^\dagger f-f\|_{L^2_\rho}^2}_{\text{best approximation error}} + C u^2 \|\mathcal{T}^\dagger f\|^2_2.
\end{equation}
Even though this choice minimizes the approximation error in the lower bound, the expansion coefficients $\|\hat{c}\|_2^2$ can grow very large—of the order of $1/u^2$—when $\Phi$ is numerically rank-deficient. Consequently, even disregarding computability considerations, the coefficients of the best approximation in exact arithmetic generally have little relevance in the presence of numerical rank-deficiency.

\subsubsection{Accuracy barrier and numerical orthogonalization}

The result of Thm.~\ref{thm:lowerbound} allows us to think about the best possible error in the numerical span. We find that for solutions with a small numerical error to exist, there must exist coefficients with both a small approximation error, $\|\mathcal{T}c-f\|_{L^2_\rho}$, and small norm, $u\sqrt{C}\|c\|_2$.

\begin{corollary} \label{corr:accuracy barrier}
    Given the setting of Thm.~\ref{thm:lowerbound}, the best root mean square error of $f$ in $\SPAN(\widetilde{\Phi})$ satisfies
    \[
        \inf_{v \in \SPAN(\widetilde{\Phi})} \sqrt{\EXP\|v - f\|_{L^2_\rho}^2} \geq \inf_{c \in \mathbb{C}^n} \frac{1}{\sqrt{2}} \left( \|\mathcal{T}c-f\|_{L^2_\rho} + u \sqrt{C} \|c\|_2 \right).
    \]
\end{corollary}
\begin{proof}
    This follows immediately from taking the square root and the infimum over $c \in \mathbb{C}^n$ on both sides of~\eqref{eq:thm8}, and using $\sqrt{x^2 + y^2} \geq (x+y)/\sqrt{2}$ for $x,y > 0$.
\end{proof}

An important observation is that this accuracy barrier cannot be mitigated through a change-of-basis by recombination of the rounded functions. Indeed, if one were to define a new set of basis functions $\{q_i\}_{i=1}^n$, each one being a linear combination of the old basis functions, that is, they are evaluated via
\[
    q_i(x) = \sum_{j=1}^n a_{i,j} \fl(\phi_j(x)), \qquad a_{i,j}\in\mathbb{C},
\]
then $\SPAN(\{q_i\}_{i=1}^n) \subseteq \SPAN(\{\fl(\phi_i)\}_{i=1}^n)$. Therefore, one cannot expect higher accuracy of the numerical approximation, irrespective of the properties of $\{q_i\}_{i=1}^n$. In particular, straightforward numerical orthogonalization cannot mitigate the effects of finite-precision arithmetic described in this paper. 

\subsection{Near-best numerical approximation and regularization} \label{sec:nearbestreg}
Consider an infinite sequence of functions $\{\phi_i\}_{i=1}^\infty$ and, for each $n \in \mathbb{N}$, let $\mathcal{T}_n \from \mathbb{C}^n \to L^2(X,\rho)$ denote the synthesis operator associated with the finite set $\{\phi_i\}_{i=1}^n$. We assume that each function $\phi_i$ varies continuously, in such a way that finite-precision rounding satisfies
\begin{equation} \label{eq:lowerboundsequence}
    \exists C > 0: \qquad \norm{ \fl(\phi_i) - \phi_i }_{L^2_\rho}^2 \geq C u^2, \qquad 1 \leq i < \infty,
\end{equation}
as motivated by the results of Thm.~\ref{thm:determlowerbound}. Building on the probabilistic modeling results of Corr.~\ref{corr:accuracy barrier}, it is then expected that
\begin{equation} \label{eq:nearbestassumption}
    \exists C_1, C_2 > 0: \qquad \inf_{c \in \mathbb{C}^n}  \|\fl(\mathcal{T}_n) c - f \|_{L^2_\rho} \geq C_1 \inf_{c \in \mathbb{C}^n}  \left( \|\mathcal{T}_n c-f\|_{L^2_\rho} + u C_2 \|c\|_2 \right)
\end{equation}
for each $n \in \mathbb{N}$ and for all $f \in L^2(X,\rho)$. 

In this setting, and in light of Prop.~\ref{prop:continuous-upper}, we define the notion of near-best numerical approximation. Notably, this definition does not require that the functions $\phi_i$ are numerically rank-deficient. Nevertheless, it is in such settings that the penalty proportional to the norm of the expansion coefficients plays a substantial role.

\begin{definition} \label{def:nearbest}
    Let $\{\phi_i\}_{i=1}^\infty \subset L^2(X,\rho)$ and $\fl \from \mathbb{C} \to \mathbb{C}$ satisfy~\eqref{eq:nearbestassumption}. An approximation procedure $\{\mathcal{A}_n\}_{n \in \mathbb{N}}$, with $$\mathcal{A}_n \from L^2(X,\rho) \to \mathbb{C}^n, $$
    is said to yield \emph{near-best numerical approximations} if there exist constants $A_1, A_2 > 0$ such that
    \begin{equation} \label{eq:nearbest}
        \|\fl(\mathcal{T}_n) \mathcal{A}_n(f) - \fl(f)\|_{L^2_\rho} \leq A_1  \inf_{c \in \mathbb{C}^n} \left(\|\mathcal{T}_nc - f\|_{L^2_\rho} + u A_2 \sqrt{n} \|\mathcal{T}_n\|_{2,L^2_\rho} \|c\|_2 \right) + u \|f\|_{L^2_\rho},
    \end{equation}
    for each $n \in \mathbb{N}$ and for all $f \in L^2(X,\rho)$.
\end{definition}

Here, we permit $\mathcal{A}_n \from L^2(X,\rho) \to \mathbb{C}^n$. For the method to be numerically realizable, however, one must account for rounding errors—both in the input data and those induced by the algorithm—as well as the fact that a computer operates on discrete information. These issues will be addressed in Section~\ref{sec:discretization}. We first analyze which idealized procedures yield near-best numerical approximations.

Since rounding errors are amplified by large expansion coefficients, it is natural to consider approximations that penalize the norm of the coefficients. This corresponds precisely to penalizing directions that are only weakly represented by $\Phi$. In the following theorem, it is shown that approximations that (nearly) minimize the \textit{$\ell^2$-regularized $L^2_\rho$-error}
\begin{equation} \label{eq:contl2regobj}
    E_\epsilon(c) = \|\mathcal{T} c - f\|_{L^2_\rho} + \epsilon \|c\|_2,
\end{equation}
yield near-best numerical errors\footnote{The unorthodox choice of symbol $\epsilon$ for the regularization parameter emphasizes that it is small and that regularization is purely due to numerical effects.}. This renders $\ell^2$-regularization with a regularization parameter of the order of the unit roundoff an important tool in the presence of numerical rank-deficiency.

\begin{theorem} \label{thm:regnearbest}
    Given $f \in L^2(X,\rho)$, let $\hat{c} \in \mathbb{C}^n$ satisfy
    \begin{equation} \label{eq:regnearbesteq1}
        E_\epsilon(\hat{c}) \leq C \inf_{c \in \mathbb{C}^n} E_\epsilon(c)
    \end{equation}
    for some $C \geq 1$ and $\epsilon \geq u \sqrt{n} \|\mathcal{T}\|_{2,L^2_\rho}$. Then,
    \[
        \|\fl(\mathcal{T}) \hat{c} - \fl(f)\|_{L^2_\rho} \leq C \inf_{c \in \mathbb{C}^n} \left(\|\mathcal{T} c - f\|_{L^2_\rho} + \epsilon \|c\|_2 \right) + u \|f\|_{L^2_\rho}.
    \]
\end{theorem}
\begin{proof}
    We have 
    \begin{align*}
        \|\fl(\mathcal{T})\hat{c} - \fl(f)\|_{L^2_\rho} &\leq \|\mathcal{T}\hat{c} - f\|_{L^2_\rho} + u(\sqrt{n} \|\mathcal{T}\|_{2,L^2_\rho} \|\hat{c}\|_2 + \|f\|_{L^2_\rho}) \quad &\text{(Prop.~\ref{prop:continuous-upper})} \\
        &\leq \|\mathcal{T}\hat{c} - f\|_{L^2_\rho} + \epsilon \|\hat{c}\|_2 + u \|f\|_{L^2_\rho} \quad &(\epsilon \geq u \sqrt{n}\|\mathcal{T}\|_{2,L^2_\rho}) \\
        &\leq C \inf_{c \in \mathbb{C}^n} (\|\mathcal{T}c - f\|_{L^2_\rho} + \epsilon \|c\|_2) + u \|f\|_{L^2_\rho} \quad &\text{\text{(using~\eqref{eq:regnearbesteq1})}}
    \end{align*}
\end{proof}

Regularization on the order of the unit roundoff is standard in a fully discrete, numerically rank-deficient setting, i.e., in low rank matrix theory. For completeness, we show that two popular methods indeed satisfy condition~\eqref{eq:regnearbesteq1}. 

\begin{proposition} \label{lm:regu}
    The Tikhonov regularized approximation
    \[
        \hat{c} = \argmin_{c\in\mathbb{C}^n} \|\mathcal{T} c - f\|_{L^2_\rho}^2 + \epsilon^2 \|c\|_2^2
    \]
    satisfies~\eqref{eq:regnearbesteq1} for $C = \sqrt{2}$. The truncated singular value decomposition (TSVD) approximation 
    \[
        \hat{c} = (\mathcal{T})_\epsilon^\dagger f,
    \]
    where $(\mathcal{T})_\epsilon ^\dagger$ denote the Moore-Penrose inverse of $\mathcal{T}$ after setting its singular values below $\epsilon$ to zero, satisfies~\eqref{eq:regnearbesteq1} for $C = 2$.
\end{proposition}
\begin{proof}
    The first result follows immediately from taking the square root on both sides of
    \[
        \|\mathcal{T} \hat{c} - f\|_{L^2_\rho}^2 + \epsilon^2 \|\hat{c}\|_2^2 \leq \inf_{c \in \mathbb{C}^n} \|\mathcal{T} c - f\|_{L^2_\rho}^2 + \epsilon^2 \|c\|_2^2,
    \]
    combined with $(x+y)/\sqrt{2} \leq \sqrt{x^2 + y^2} \leq x+y$, for all $x,y > 0$.
    For the second result, observe that
    \[
        \mathcal{T}(\mathcal{T}_\epsilon)^\dagger = \mathcal{T}_\epsilon(\mathcal{T}_\epsilon)^\dagger \qquad \text{and} \qquad (\mathcal{T}_\epsilon)^\dagger\mathcal{T} = (\mathcal{T}_\epsilon)^\dagger \mathcal{T}_\epsilon
    \]
    project orthogonally onto $\SPAN(\{u_i\}_{\sigma_i \geq \epsilon})$ and $\SPAN(\{v_i\}_{\sigma_i \geq \epsilon})$, respectively, where $\mathcal{T} = \sum_{i=1}^{\hat{n}} \sigma_i u_i v_i^*$ is the singular value decomposition of $\mathcal{T}$. Using $f = (f - \mathcal{T}c) + \mathcal{T}c$ for every $c \in \mathbb{C}^n$, we obtain 
    \begin{align*}
        \|f-\mathcal{T}\hat{c}\|_{L^2_\rho} = \|f - \mathcal{T}(\mathcal{T}_\epsilon)^\dagger f \|_{L^2_\rho} &= \|(\mathcal{I} - \mathcal{T}(\mathcal{T}_\epsilon)^\dagger) ((f - \mathcal{T}c) + \mathcal{T}c) \|_{L^2_\rho} \\
        &\leq \|\mathcal{I} - \mathcal{T}(\mathcal{T}_\epsilon)^\dagger \|_{L^2_\rho,L^2_\rho} \|f - \mathcal{T}c\|_{L^2_\rho} + \|\mathcal{T} - \mathcal{T}(\mathcal{T}_\epsilon)^\dagger \mathcal{T}\|_{2,L^2_\rho} \|c\|_2 \\
        &\leq \|f - \mathcal{T}c\|_{L^2_\rho} + \epsilon  \|c\|_2,
    \end{align*}
    where $\mathcal{I} \from L^2(X,\rho) \to L^2(X,\rho)$ denotes the identity operator. Similarly,
    \begin{align*}
        \epsilon \|\hat{c}\|_2 = \epsilon \|(\mathcal{T}_\epsilon)^\dagger f\|_{2} &= \epsilon \|(\mathcal{T}_\epsilon)^\dagger ((f - \mathcal{T}c) + \mathcal{T}c)\|_{2} \\
        &\leq \epsilon \|(\mathcal{T}_\epsilon)^\dagger\|_{L^2_\rho,2} \|f - \mathcal{T}c\|_{L^2_\rho} + \epsilon \|(\mathcal{T}_\epsilon)^\dagger \mathcal{T}\|_{2,2} \|c\|_2 \\
        &\leq \|f - \mathcal{T}c\|_{L^2_\rho} + \epsilon \|c\|_2.\qedhere
    \end{align*}
\end{proof}

Finally, we briefly compare our setting with the classical theory of inverse problems, where regularization is standard~\cite{vogel2002computational}. Although closely related, the emphases differ. In our setting, regularization is needed due to perturbations of the operator $\mathcal{T}$, while noise in the right-hand side is insignificant. Moreover, inverse problems aim to approximate an underlying solution $c_{\mathrm{true}}$, whereas our objective is to compute coefficients $c$ that yield a small numerical residual, without assigning intrinsic meaning to them. Accordingly, the regularization parameter in our analysis is dictated by the unit roundoff, in contrast to the classical setting, where it is selected to balance stability against data noise and thus depends on the (typically unknown) noise level. In that context, determining a suitable value of the regularization parameter is often crucial. In the context of numerical rank-deficiency, it is not.

\section{Discrete approximation with numerical rank-deficiency} \label{sec:discretization}
It follows from the previous section that appropriately regularized approximations achieve near-best numerical errors in the presence of numerical rank-deficiency. To obtain a computable approximation method, we still need to address two issues: discretization, discussed in Subsection~\ref{sec:discretizationcondition}, and rounding errors arising in the computation of the expansion coefficients, discussed in Subsection~\ref{sec:computation}. An important result is that, for regularized approximations, the discretization condition is relaxed compared to the unregularized setting. 

\subsection{Relaxed discretization condition} \label{sec:discretizationcondition}

We define a sampling operator
\begin{equation*}
    \mathcal{M} \from L^\infty(X) \to \mathbb{C}^m, \quad f \mapsto \{l_j(f)\}_{j=1}^m, 
\end{equation*}
consisting of $m$ linear functionals $l_j$. We allow that the sampling operator is only defined on $L^\infty(X)$ to accommodate for the important case of pointwise sampling, where
\begin{equation} \label{eq:pointwisesampling}
    \mathcal{M} f = \left\{\sqrt{w_j} f(x_{j})\right\}_{j=1}^m 
\end{equation}
for some points $x_{j} \in X$ and weights $w_j > 0$. We assume that $\mathcal{M}$ is bounded with respect to $\|\cdot\|_{L^\infty}$:
\[ 
    \|\mathcal{M}\|_{L^\infty,2} \coloneq \sup_{v \in L^\infty(X), v \neq 0} \frac{\|\mathcal{M} v\|_2}{\|v\|_{L^\infty}} < \infty.
\] 
Note that in the case of~\eqref{eq:pointwisesampling}, one has $\|\mathcal{M}\|_{L^\infty,2}^2 = \sum_{j=1}^m w_j$.
The operator $\mathcal{M}$ defines a seminorm via $\|\cdot\|_{\mathcal{M}} \coloneq \|\mathcal{M} \cdot \|_2$. Furthermore, we assume that $\Phi \subset L^\infty(X)$.

Following Thm.~\ref{thm:regnearbest}, approximations that (nearly) minimize the $\ell^2$-regularized $L^2_\rho$-error $E_\epsilon$ result in near-best numerical errors. A similar result holds after discretization, given that the discretization is sufficiently rich. In this case, the quantity of interest is the \textit{$\ell^2$-regularized discrete error}
\begin{equation} \label{eq:discl2regobj}
    e_\epsilon(c) = \|\mathcal{T} c - f\|_\mathcal{M} + \epsilon \|c\|_2.
\end{equation}

\begin{theorem} \label{thm:discreteerrorbound}
    Given $\Phi \subset L^\infty(X)$, let $\mathcal{M}\from L^\infty(X) \to \mathbb{C}^m$ satisfy
    \begin{equation} \label{thm:norminequality}
        \gammar \left( \|\mathcal{T}c\|_{L^2_\rho} + \epsilon \|c\|_2 \right) \leq \|\mathcal{T}c\|_{\mathcal{M}} + \epsilon \|c\|_2, \qquad \forall c \in \mathbb{C}^n,
    \end{equation}
    for some $0 < \gammar \leq 1$ and $\epsilon \geq u \sqrt{n} \|\mathcal{T}\|_{2,L^2_\rho}$. Furthermore, given $f \in L^\infty(X)$, let $\hat{c} \in \mathbb{C}^n$ satisfy
    \begin{equation} \label{thm:discreteregulariedcond}
        e_\epsilon(\hat{c}) \leq C \inf_{c \in \mathbb{C}^n} e_\epsilon(c)
    \end{equation}
    for some $C \geq 1$. Then
    \begin{align} \label{thm:finalerrorbound}
        \begin{split} 
            \|\fl(\mathcal{T})\hat{c} - \fl(f)\|_{L^2_\rho} &\leq \inf_{c \in \mathbb{C}^n} \left(E_\epsilon(c) + \frac{1+C}{\gammar} e_\epsilon(c)\right) + u\|f\|_{L^2_\rho} \\
            &\leq \inf_{c \in \mathbb{C}^n} \left( \left( 1 + \|\mathcal{M}\|_{L^\infty,2} \frac{1+C}{\gammar} \right) \|\mathcal{T}c-f\|_{L^\infty} + \left( 1 + \frac{1+C}{\gammar} \right) \epsilon \|c\|_2 \right) + u \|f\|_{L^2_\rho}.
        \end{split}
    \end{align}
\end{theorem}
\begin{proof}
    It holds that 
    \begin{align*}
        \|\fl(\mathcal{T})\hat{c} - \fl(f)\|_{L^2_\rho} &\leq \|\mathcal{T}\hat{c} - f\|_{L^2_\rho} + u(\sqrt{n}\|\mathcal{T}\|_{2,L^2_\rho} \|\hat{c}\|_2 + \|f\|_{L^2_\rho}) \quad &\text{(Prop.~\ref{prop:continuous-upper})} \\
        &\leq \|\mathcal{T}\hat{c} - f\|_{L^2_\rho} + \epsilon \|\hat{c}\|_2 + u \|f\|_{L^2_\rho} \quad &(\epsilon \geq u \sqrt{n} \|\mathcal{T}\|_{2,L^2_\rho}) \\
        &\leq \|\mathcal{T}(\hat{c}-c)\|_{L^2_\rho} + \epsilon \|\hat{c}-c\|_2 + E_\epsilon(c) + u \|f\|_{L^2_\rho} \quad &\text{($\triangle$-inequality)}
    \end{align*}
    for any $c \in \mathbb{C}^n$.
    Furthermore, using~\eqref{thm:norminequality} on the set of coefficients $\hat{c}-c\in\mathbb{C}^n$, we get
    \begin{align*}
        \|\mathcal{T}(\hat{c}-c)\|_{L^2_\rho} + \epsilon \|\hat{c}-c\|_2 &\leq \frac{1}{\gammar} \left( \|\mathcal{T}(\hat{c}-c)\|_{\mathcal{M}} + \epsilon \|\hat{c}-c\|_2 \right) \\ 
        &\leq \frac{1}{\gammar} \left( e_\epsilon(\hat{c}) + e_\epsilon(c) \right)\quad &\text{($\triangle$-inequality)} \\
        &\leq \frac{1+C}{\gammar} e_\epsilon(c) \quad &\text{\text{(using~\eqref{thm:discreteregulariedcond})}}
    \end{align*}
    The final inequality follows from $\|\mathcal{T}c-f\|_{L^2_\rho(X)} \leq \|\mathcal{T}c-f\|_{L^\infty(X)}$ (since $\rho(X) = 1$) and $\|\mathcal{T}c-f\|_{\mathcal{M}} \leq \|\mathcal{M}\|_{L^\infty,2}\|\mathcal{T}c-f\|_{L^\infty(X)}$.
\end{proof}

\begin{remark}
    After discretization, we limit ourselves to error bounds that depend on the optimal error measured in the $L^\infty(X)$-norm. Bounds in terms of the $L^2(X,\rho)$-norm are tighter, yet require additional effort. For instance, probabilistic error bounds of this kind can be obtained in the case of random pointwise sampling, as demonstrated in~\cite{cohen2017,dolbeault2022optimal}.
\end{remark}

We can compare this result to the standard setting of least squares fitting in exact arithmetic, see for example~\cite[Lemma 5.1]{adcockOptimalSamplingLeastsquares2024}. 

\begin{proposition} \label{prop:discreteerror}
    Given $\Phi \subset L^\infty(X)$, let $\mathcal{M}\from L^\infty(X) \to \mathbb{C}^m$ satisfy
    \begin{equation} \label{thm:analytical:norminequality}
        \gammae \|v\|_{L^2_\rho} \leq \|v\|_{\mathcal{M}}, \qquad \forall v \in \SPAN(\Phi),
    \end{equation}
    for some $0 < \gammae \leq 1$. Furthermore, given $f \in L^\infty(X)$, let $c \in \mathbb{C}^n$ satisfy
    \[
        \hat{c} = \argmin_{c \in \mathbb{C}^n} \|\mathcal{T}c - f\|_{\mathcal{M}}^2.
    \]
    Then,
    \begin{align}
        \begin{split}
            \|\mathcal{T}\hat{c} - f\|_{L^2_\rho} &\leq \inf_{v \in \SPAN(\Phi_n)} \left( \|f - v\|_{L^2_\rho} + \frac{1}{\gammae} \|f - v\|_{\mathcal{M}} \right) \\
            &\leq \left(1 + \frac{\|\mathcal{M}\|_{L^\infty,2}}{\gammae} \right) \inf_{v \in \SPAN(\Phi_n)} \|f-v\|_{L^\infty}.
        \end{split}
    \end{align}
\end{proposition}

\begin{remark}
The condition~\eqref{thm:analytical:norminequality} is related to (the lower bound of) sampling discretizations~\cite{kashin2022sampling} and Marcinkiewicz-Zygmund inequalities~\cite{temlyakov2018marcinkiewicz,grochenig2020sampling}, which are typically defined for a sequence of approximations spaces and sampling operators.
\end{remark} 

If one compares the conditions~\eqref{thm:norminequality} and~\eqref{thm:analytical:norminequality}, it immediately becomes clear that the regularized inequality depends on the basis $\Phi$, whereas the standard inequality only depends on the spanned space. On the other hand, both inequalities are equivalent whenever $\epsilon = 0$ and/or $\gamma = 1$. In the case where $\epsilon > 0$ and $\gamma < 1$,~\eqref{thm:norminequality} is a strict relaxation of~\eqref{thm:analytical:norminequality}. This follows from rewriting them as
\begin{align*}
    &\eqref{thm:norminequality}: \qquad  \gammar \|\mathcal{T} c\|_{L^2_\rho} \leq \|\mathcal{T} c\|_{\mathcal{M}} + (1 - \gammar) \epsilon \|c\|_2& \\
    &\eqref{thm:analytical:norminequality}: \qquad \gammae \|\mathcal{T} c\|_{L^2_\rho} \leq \|\mathcal{T} c\|_{\mathcal{M}}&
\end{align*}
Hence, the largest $\gammar$ and $\gammae$ satisfying~\eqref{thm:norminequality} and~\ref{thm:analytical:norminequality}, respectively, satisfy
\[
    \gammar \geq \gammae,
\]
meaning that less information may be needed for numerical approximation compared to approximation in exact arithmetic.

One might wonder whether the relaxation is significant for small $\epsilon \gtrsim u \sqrt{n} \|\mathcal{T}\|_{2,L^2_\rho}$. To investigate this question, we consider the singular value decomposition of $\mathcal{T}$:
\[
    \mathcal{T} = \sum_{i=1}^{\hat{n}} \sigma_i u_i v_i^*,
\]
where $\{u_i\}_{i=1}^{\hat{n}}$ is an orthonormal basis for $\SPAN(\Phi)$ and $\{v_i\}_{i=1}^{\hat{n}}$ is an orthonormal basis for $\mathbb{C}^{\hat{n}}$. Directions associated with small singular values are weakly represented in $\Phi$ and damped significantly by regularization. Therefore, less data is required from these directions.
\begin{proposition}
    For every $$c \in \SPAN(\{v_i \; \vert \; \sigma_i \leq (1/\gammar-1)\epsilon\}),$$ condition~\eqref{thm:norminequality} is satisfied, irrespective of the choice of sampling operator.
\end{proposition}
\begin{proof}
    One has
    \[
        \gammar \|\mathcal{T}c\|_{L^2_\rho} \leq \gammar(1/\gammar-1)\epsilon \|c\|_2 = (1-\gammar)\epsilon \|c\|_2,
    \]
    which implies that~\eqref{thm:norminequality} is satisfied.
\end{proof}

Whenever the basis is numerically rank-deficient in the sense of Def.~\ref{def:numrankdef}, at least one singular value satisfies $\sigma_i \leq u \|\mathcal{T}\|_{2,L^2_\rho}$. Consequently, for $\gammar$ bounded away from $1$ and $\epsilon \gtrsim u \sqrt{n} \|\mathcal{T}\|_{2,L^2_\rho}$, the subspace $\SPAN(\{v_i\}_{\sigma_i \leq (1/\gammar-1)\epsilon})$ is non-empty. In contrast, when $\Phi$ is an orthonormal basis, all singular values are equal to $1$ and the relaxation becomes negligible. This is consistent with the fact that, in an orthonormal basis, every direction is well resolved, essentially unaffected by rounding errors and practically undamped by regularization.

\subsection{Rounding errors arising in computation} \label{sec:computation}
In this section, we address how the expansion coefficients $c$ are computed such that they satisfy the discrete regularization requirement~\eqref{thm:discreteregulariedcond}, taking into account rounding errors arising during their computation. Since these results may be valuable outside of our specific setting in function approximation, we first consider an abstract discrete setting. That is, for any $A \in \mathbb{C}^{m \times n}$ and $b \in \mathbb{C}^m$, we look for ways to obtain coefficients c such that the numerical residual 
\[
    \|\fl(A) c - \fl(b)\|_2
\]  
is small. At the end of the section, we use these results to obtain an end-to-end error bound for function approximation. The proofs of this section can be found in Appendix~\ref{sec:app1}.

\subsubsection{Insufficiency of backward stability} \label{sec:backwardstability}
The coefficients minimizing the numerical residual equal
\begin{equation} \label{eq:bestnumapprox}
    c = \fl(A)^\dagger \fl(b).
\end{equation}
In practice, however, these can never be computed exactly: rounding errors arise not only in the data $\fl(A)$ and $\fl(b)$ but also during the execution of the numerical algorithm itself. This raises the question of what guarantees one can reasonably expect from the computed output.

In numerical analysis, it is standard to require that a numerical method is \textit{backward stable}. Specifically, \textit{a method for computing $y = g(x)$ is called backward stable if, for any $x$, it produces a computed $\widehat{y}$ with a small backward error, that is, $\widehat{y} = g(x + \Delta x)$ for some small $\Delta x$}~\cite[section 1.5]{highamAccuracyStabilityNumerical2002}. Here, the perturbation $\Delta x$ models not only the rounding of the input data---which is already taken into account in~\eqref{eq:bestnumapprox}---but also the cumulative effect of rounding errors incurred during the computation of $g(x)$. 

Consequently, the coefficients computed by a backward stable algorithm for the mapping
\begin{equation} \label{eq:coefforthproj}
   (A,b) \mapsto A^\dagger b \in \argmin_{c \in \mathbb{C}^n} \|Ac-b\|_2^2
\end{equation}
satisfy
\[
    \hat{c} = (A+\Delta A)^\dagger (b+\Delta b),
\]
where, in general, $A + \Delta A \neq \fl(A)$ and $b + \Delta b \neq \fl(b)$. Typically, $\|\Delta A\|_2 \lesssim u \|A\|_2$ and $\|\Delta b\|_2 \lesssim u \|b\|_2$, where the implied constants depend on the algorithm and typically on the problem size. We find that the output of a backward stable algorithm is guaranteed to have a small numerical residual only when $A$ is well-conditioned.

\begin{theorem} \label{thm:backwardstab1}
    Given $A\in\mathbb{C}^{m\times n}$ and $b \in \mathbb{C}^m$, let $\hat{c} \in \mathbb{C}^n$ satisfy
    \[
        \hat{c} = (A+\Delta A)^\dagger (b+\Delta b),
    \]
    for some $\|\Delta A\|_2 \leq \Calg u \|A\|_2$ and $\|\Delta b\|_2 \leq \Calg u \|b\|_2$ with $\Calg > 0$. Then,
    \[
        \|A\hat{c} - b\|_2 \leq \|Ac - b\|_2 + \Calg u \left( \|A \|_2 (\|\hat{c}\|_2 + \|c\|_2) + 2\|b\|_2 \right), \qquad \forall c \in \mathbb{C}^n.
    \]
    Furthermore, if $A$ satisfies $\kappa(A) > (\Calg u)^{-1}$, then 
    \[
        \|A\hat{c} - b\|_2 \leq \|AA^\dagger b - b\|_2 + \Calg u \left( \frac{(1+ \Calg u) \|A\|_2 }{\sigma_{\min}(A) - \Calg u \|A\|_2} + \kappa(A) + 2\right) \|b\|_2
    \]
    and
    \[
        \|\hat{c}\|_2 \leq \frac{(1+ \Calg u)\|b\|_2}{\sigma_{\min}(A) - \Calg u \|A\|_2},
    \]
    where $\sigma_{\min}$ denotes the smallest nonzero singular value.
\end{theorem}

When $A$ is well-conditioned, Thm.~\ref{thm:backwardstab1} provides strong guarantees. In particular, the computed coefficients $\hat{c}$ have a near-optimal residual and controlled norm. Combined with Prop.~\ref{prop:discrete-upper}, this guarantees a small numerical residual.

In contrast, no such guarantees are available when $A$ is numerically rank-deficient. Indeed, a backward stable algorithm applied to~\eqref{eq:coefforthproj} may produce coefficients of arbitrarily large norm. More precisely, there exist perturbations $\Delta A$ satisfying $\|\Delta A \|_2 \lesssim u \|A\|_2$ such that the smallest singular value of $A + \Delta A$ becomes arbitrarily small. In this situation, the pseudo-inverse $(A + \Delta A)^\dagger$ may have arbitrarily large norm and, consequently, the computed coefficients $\hat{c}$ can be unbounded.

\begin{remark}
    One might argue that this worst-case scenario is unlikely if the perturbations $\Delta A$, $\Delta b$ are quasi-random. In the spirit of Section~\ref{sec:lowerbound}, one could attempt to define a stochastic model to capture this behaviour. However, we proceed by showing that backward stability does suffice when combined with regularization.
\end{remark}

\subsubsection{Sufficiency of backward stability with regularization}

Once again, damping perturbed directions through regularization is a natural idea to avoid this blow-up of coefficients. Hence, we consider backward stable algorithms for mappings of the following kind:
\begin{equation} \label{eq:coeffreg}
    (A,b) \mapsto \alpha(A,b) \qquad \text{where } e_\epsilon(\alpha(A,b); A, b) \leq C \inf_{c \in \mathbb{C}^n} e_\epsilon(c; A, b)
\end{equation}
for some $\epsilon > 0$ and $C \geq 1$ with $e_\epsilon(c; A, b) = \|Ac - b\|_2 + \epsilon \|c\|_2$. 

\begin{theorem} \label{thm:backwardstab2}
    Given $A\in\mathbb{C}^{m\times n}$ and $b \in \mathbb{C}^m$, let $\hat{c} \in \mathbb{C}^n$ satisfy
    \begin{equation} \label{thm:backwardstab2:eq}
        e_\epsilon(\hat{c}; A+ \Delta A, b + \Delta b) \leq C \inf_{c \in \mathbb{C}^n} e_\epsilon(c; A+\Delta A, b+ \Delta b)
    \end{equation}
    for some $\|\Delta A\|_2 \leq \Calg u \|A\|_2$ and $\|\Delta b\|_2 \leq \Calg u \|b\|_2$ with $\Calg > 0$ and $C \geq 1$. Then, if $\epsilon  \geq \Calg u \|A\|_2$, it holds that
    \[
        e_\epsilon(\hat{c}; A, b) = \|A\hat{c} - b\|_2 + \epsilon \|\hat{c}\|_2 \leq 2C\|Ac - b\|_2 + 3C \epsilon \|c\|_2  + u\Calg(1+2C)\|b\|_2, \qquad \forall c \in \mathbb{C}^{n}.
    \]
\end{theorem}

Combining this result with Prop.~\ref{prop:discrete-upper}, we find that small numerical residuals are guaranteed when $\epsilon \geq u \|A\|_2 \max(\Calg, \sqrt{n})$. This reveals an important reason for regularization: it guarantees accuracy under backward stable computation, even in the presence of numerical rank-deficiency.

Regularizing the problem as in~\eqref{eq:coeffreg} is by no means new. In low rank matrix theory it is standard practice to penalize the $\ell^2$-norm of the solution vector $c$. This is commonly achieved either via truncated singular value decomposition (TSVD) or via Tikhonov regularization. Entirely analogously to Prop.~\ref{lm:regu}, one can show the following. For the TSVD, this result also follows from combining~\cite[Lemma 3.3]{coppe2020az} and~\cite[Theorem 3.8]{adcockFramesNumericalApproximation2020}.

\begin{proposition}
    The Tikhonov regularized solution
    \[
        \hat{c} = \argmin_{c\in\mathbb{C}^n} \|A c - f\|_2^2 + \epsilon^2 \|c\|_2^2
    \]
    satisfies~\eqref{eq:coeffreg} for $C = \sqrt{2}$. The truncated singular value decomposition (TSVD) solution 
    \[
        \hat{c} = A_\epsilon^\dagger f,
    \]
    where $A_\epsilon ^\dagger$ denote the Moore-Penrose inverse of $A$ after setting its singular values below a threshold $\epsilon$ to zero, satisfies~\eqref{eq:coeffreg} for $C = 2$.
\end{proposition}

\subsubsection{End-to-end error bound for function approximation}

We are now able to combine Thm.~\ref{thm:backwardstab2} for $A = \mathcal{M}\mathcal{T}$ and $b = \mathcal{M}f$ with Thm.~\ref{thm:discreteerrorbound}. This results in an error bound that takes into account rounding of the functions in $\Phi$ and $f$, rounding during the computation of the expansion coefficients and discretization errors.
\begin{corollary}
    Given $\mathcal{M} \from L^\infty(X) \to \mathbb{C}^m$, $\Phi \subset L^\infty(X)$ and $f \in L^\infty(X)$, let $A = \mathcal{M}\mathcal{T}$ and $b = \mathcal{M}f$. Furthermore, let
    \begin{itemize}
        \item $\hat{c}\in\mathbb{C}^n$ be computed by a backward stable algorithm for an appropriately regularized mapping, i.e.,~\eqref{thm:backwardstab2:eq} is satisfied with $\|\Delta A\|_2 \leq \Calg u \|A\|_2$ and $\|\Delta b\|_2 \leq \Calg u \|b\|_2$ for some $\Calg > 0$, $C \geq 1$ and $\epsilon \geq \max(u\sqrt{n} \|\mathcal{T}\|_{2,L^2_\rho}, u \Calg \|A\|_2)$;
        \item the discretization be sufficiently rich, i.e.,~\eqref{thm:norminequality} is satisfied for some $0 < \gammar \leq 1$ and the same $\epsilon$.
    \end{itemize}
    Then, 
    \begin{align} \label{thm:finalerrorbound2}
        \begin{split}
            \|\fl(\mathcal{T})\hat{c} - \fl(f)\|_{L^2_\rho} \leq \inf_{c \in \mathbb{C}^n} \left( E_\epsilon(c) + \frac{1+3C}{\gammar} e_\epsilon(c) \right) + u \Calg (1+2C) \|\mathcal{M}\|_{L^\infty,2} \|f\|_{L^\infty} + u\|f\|_{L^2_\rho}
        \end{split}
    \end{align}
    with $E_\epsilon$ and $e_\epsilon$ as defined in~\eqref{eq:contl2regobj} and~\eqref{eq:discl2regobj}, respectively.
\end{corollary}
\begin{proof}
    From the proof of Thm.~\ref{thm:discreteerrorbound}, we obtain for every $c \in \mathbb{C}^n$ that
    \[
        \|\fl(\mathcal{T})\hat{c} - \fl(f)\|_{L^2_\rho} \leq E_\epsilon(c) + \frac{1}{\gammar} \left( e_\epsilon(\hat{c}) + e_\epsilon(c) \right) + u\|f\|_{L^2_\rho}.
    \]
    Furthermore, Thm.~\ref{thm:backwardstab2} results in $e_\epsilon(\hat{c}) = e_\epsilon(\hat{c}; A, b) \leq 3C e_\epsilon(c) + u \Calg (1+2C) \|\mathcal{M}\|_{L^\infty,2} \|f\|_{L^\infty}$.
\end{proof}

\begin{remark}
    We find that the lower bound of the regularization parameter $\epsilon$ depends on $\Calg$, of which the behaviour is well-studied across different algorithms; see, for example, \cite{highamAccuracyStabilityNumerical2002}.
\end{remark}

\section{Numerical examples} \label{sec:numericalexamples}
\subsection{Illustration of the accuracy barrier} \label{sec:numexample1}
We reproduce~\cite[Example~3]{brubeckVandermondeArnoldi2021} and consider the $L^2$-approximation of
\begin{equation}\label{eq:numexampleapprox}
    f(x) = \frac{1}{10 - 9x} \approx \sum_{k=-n}^{n} c_k \exp(i k \pi x / 2) = \sum_{k=-n}^{n} c_k \phi_k,
\end{equation}
on $[-1,1]$ using a Fourier extension on $[-2,2]$. For background on Fourier extensions, see Section~\ref{sec:fourierextension}. Figure~\ref{fig:comparison_svds} shows that the basis $\{\phi_k\}_{k=-n}^n$ becomes numerically rank-deficient around $n=20$. 

Discretizing in $1000$ Chebyshev points $x_j \in [-1,1]$ yields the linear system
\[
    Ac = b, \qquad (A)_{j,k} = \phi_k(x_j), \quad (b)_j = f(x_j).
\]
Figure~\ref{fig:comparison} compares several approaches for computing the coefficients $c$ in MATLAB. Whenever regularization is employed, we take $\epsilon = 10^{-14}$. Details on the implementation can be found in~\cite{githubrepo}.

In exact arithmetic, one expects exponential convergence~\cite{huybrechsFourierExtensionNonperiodic2010}. However, all methods—except for one discussed below—exhibit slower convergence once the basis becomes numerically rank-deficient, regardless of how the coefficients are computed. This includes approximations computed in a numerically orthogonalized basis via a QR decomposition. These observations illustrate the accuracy barrier characterized in Corr.~\ref{corr:accuracy barrier}; see~\cite{adcockNumericalStabilityFourier2014} for a detailed discussion of the regularized convergence behaviour.

Among the tested methods, MATLAB’s backslash operator performs particularly poorly. As documented, when $A$ is numerically rank-deficient, $A \backslash b$ does not necessarily return a small-norm solution. Consequently, large coefficient norms may arise, leading to poor numerical residuals. The built-in function \texttt{lsqminnorm} provides a robust alternative.

The only method that achieves exponential convergence is the Vandermonde with Arnoldi (VwA) algorithm~\cite{brubeckVandermondeArnoldi2021,zhu2025convergence}. This does not contradict Corr.~\ref{corr:accuracy barrier}, since the VwA algorithm does not compute approximations in the numerical span of the Fourier extension functions. Instead, it constructs an alternative basis $\{q_k\}_{k=0}^n$\footnote{For real valued functions $f$, it suffices to consider non-negative modes, since the coefficients satisfy $c_{-k} = \overline{c_k}$; see~\cite[Example~3]{brubeckVandermondeArnoldi2021}.} via an iterative Arnoldi-type procedure:
\[
    q_0(x) = 1, \qquad q_k(x) = \text{orth}\bigl( \exp(i \pi x / 2)\, q_{k-1}(x) \bigr),
\]
where orth corresponds to discrete orthogonalization with respect to the previous basis functions $q_0, \dots, q_{k-1}$. Hence, it uses the additional knowledge that multiplying an exponential function by $\exp(i\pi x /2)$ increases its frequency, instead of evaluating $\exp(i k \pi x /2)$ directly. This procedure is a discrete analogue of the Stieltjes process for generating orthogonal polynomials~\cite{gautschiOrthogonalPolynomialsComputation2004}. 

Empirically, we find that the VwA basis does strongly represent all directions needed for exponential convergence. A detailed analysis of why this is the case lies beyond the scope of this paper; see~\cite{stylionopoulosArnoldiGramSchmidtProcess2010,stylianopoulos2013strong}. Notably, this also means that the sample points should satisfy the full discretization condition~\eqref{thm:analytical:norminequality} for VwA, while the relaxed discretization condition~\eqref{thm:norminequality} suffices for regularized approximation. Finally, we note that the VwA algorithm is limited to Vandermonde-type bases (e.g., monomials and exponentials).

\begin{figure}
    \centering 
    \includegraphics*[width=\linewidth]{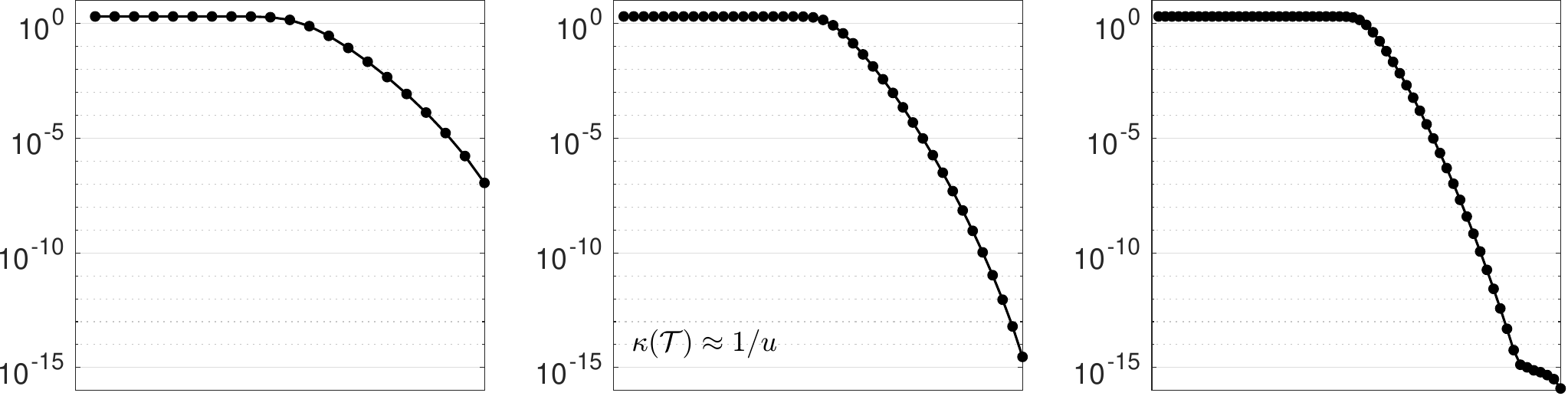}
    \caption{Singular value profile of the synthesis operator associated with $\{\exp(i k \pi x /2)\}_{k=-n}^n$ for $n = 10$ (left), $n = 20$ (middle), and $n = 30$ (right), computed in double precision using the singular value decomposition of a quasimatrix in Chebfun~\cite{driscollChebfunGuide2014}. In exact arithmetic, the singular values decrease exponentially toward zero as $n$ goes to infinity~\cite{adcockNumericalStabilityFourier2014}.}
    \label{fig:comparison_svds}
\end{figure}

\begin{figure}
    \centering 
    \includegraphics*[height=7.5cm]{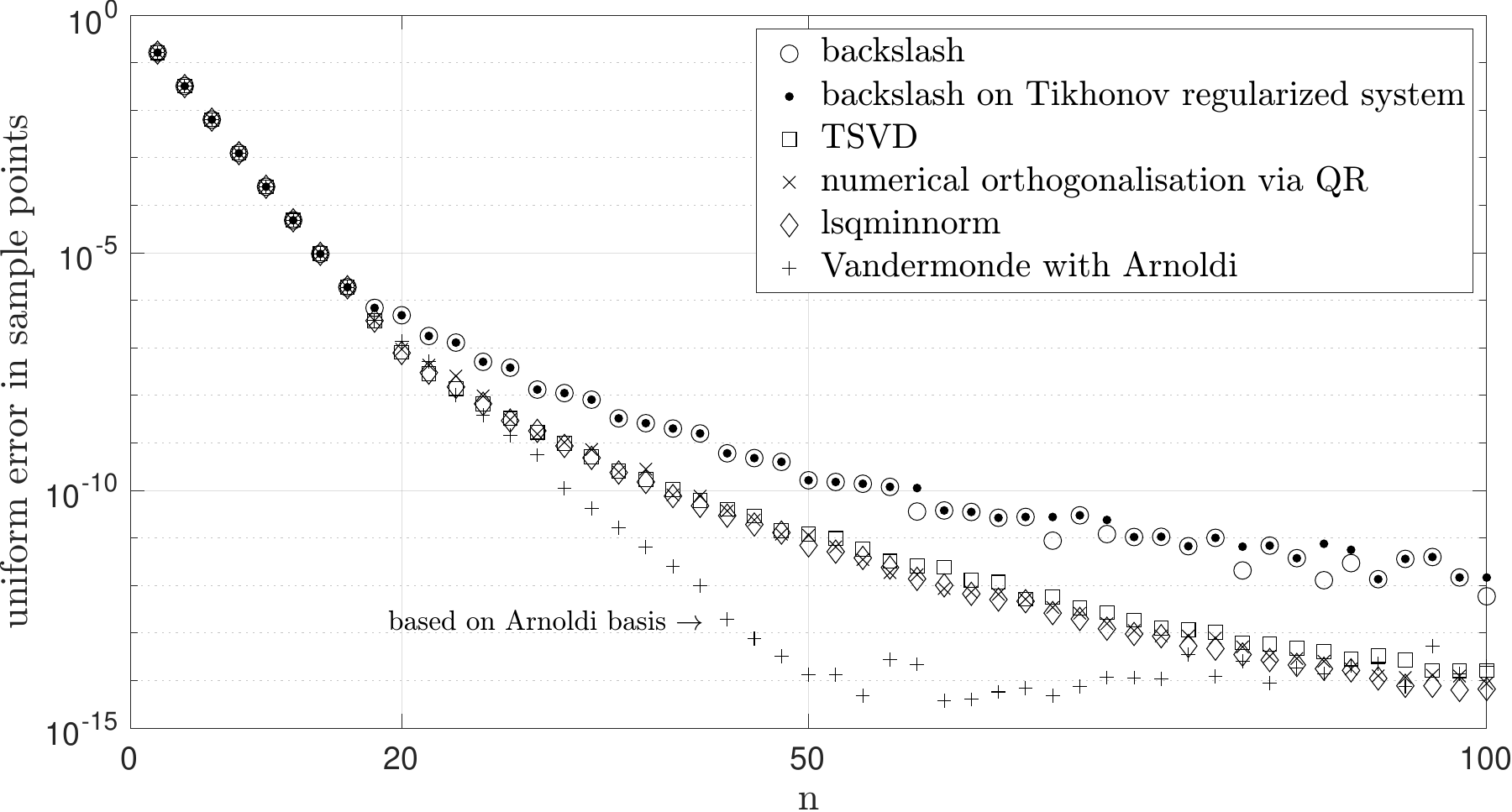}
    \caption{A comparison of different methods to compute expansion coefficients for the approximation of $f(x) = 1/(10-9x)$ in a Fourier extension frame on $[-2,2]$. All methods except the Vandermonde with Arnoldi algorithm compute approximations within the numerical span of the Fourier extension frame and are, therefore, subject to the accuracy barrier as described in Corr.~\ref{corr:accuracy barrier}.}
    \label{fig:comparison}
\end{figure}

\subsection{Illustration of the relaxed discretization condition} \label{sec:numexample2}
To illustrate the relaxed discretization condition~\eqref{thm:norminequality}, we consider an enriched basis. In many science and engineering problems, expert knowledge about the structure of the target function is known—often through asymptotic analysis—and can be incorporated in the approximation method by augmenting the approximation space with additional basis functions. For a sufficiently large number of enrichments, this leads to numerical rank-deficiency.

We again consider a simple model problem, namely the $L^2$-approximation of
\begin{equation} \label{eq:f}
    f(x) = J_{1/2}(x + 1) + \frac{1}{x^2 + 1}, \qquad x \in [-1,1]
\end{equation}
where $J_{1/2}$ denotes the Bessel function of the first kind of order $1/2$. It is known that $J_{1/2}(x + 1) \sim \sqrt{x + 1}$ as $x \to -1$, which can be taken into account by approximating in
\begin{equation} \label{eq:sumframe}
    \Phi = \{\varphi_{i}\}_{i=1}^{n/2} \cup \{\psi_{i}\}_{i=1}^{n/2}, \qquad \text{ where } \psi_i = w \varphi_i,
\end{equation}
assuming $n$ is even. The functions $\varphi_i$ are standard (smooth) basis functions, while the singular behaviour is represented by the weight function $w$. For $w \in L^\infty(-1,1)$, this spanning set is a subsequence of an overcomplete frame and, hence, $\Phi$ is numerically redundant for sufficiently large $n$; see \S7.1. For the approximation of $f$~\eqref{eq:f}, we choose $\varphi_i$ equal to the Legendre polynomial of degree $i-1$ and $w(x) = \sqrt{(x+1)/2}$.

We can evaluate the quality of a specific choice of sample set for this problem by explicitly computing the constants $\gammar$ and $\gammae$ involved in~\eqref{thm:norminequality} and~\eqref{thm:analytical:norminequality}, respectively. If $\gamma$ is close to $1$, the approximation is near-optimal, whereas if $\gamma$ is close to zero the approximation is suboptimal. Numerically computing the constants $\gamma$ can be done by solving a generalized eigenvalue problem, possibly in high precision. For more details, we refer to~\S\ref{sec:christoffelfuns} and our implementation~\cite{githubrepo}.

We use sample points related to the smooth behaviour of $\varphi_i$ and sample points related to the singular behaviour of $w$. For the first we use Legendre points, while for the latter we use sample points that are exponentially clustered towards $x = -1$. These have been found effective for least squares approximation of functions with branch point singularities in the specific case of rational approximation~\cite{herremansResolutionSingularitiesRational2023}. 

The two-parameter plots on Figure~\ref{fig:A_weightedlegendre} show $1/\gammar$ and $1/\gammae$, which appear in the error bounds for the regularized (Thm.~\ref{thm:discreteerrorbound}) and the exact case (Prop.~\ref{prop:discreteerror}), respectively, using a varying number of Legendre points and exponentially clustered sample points. In the regularized case, we plot $\gammar \approx \sqrt{\lambda_{\min}^r}$, which is accurate up to a factor $\sqrt{2}$, as described in Prop.~\ref{prop:eigenvalue}. We compare this to the accuracy of the TSVD approximation ($\epsilon = 10^{-14}$) of $f$. The error is evaluated on an independent, dense grid. We can conclude that $\gammar$ accurately predicts the behaviour of the numerical error, while $\gammae$ overestimates the required number of sample points.

\begin{figure}
    \centering
    \includegraphics[width=.495\linewidth]{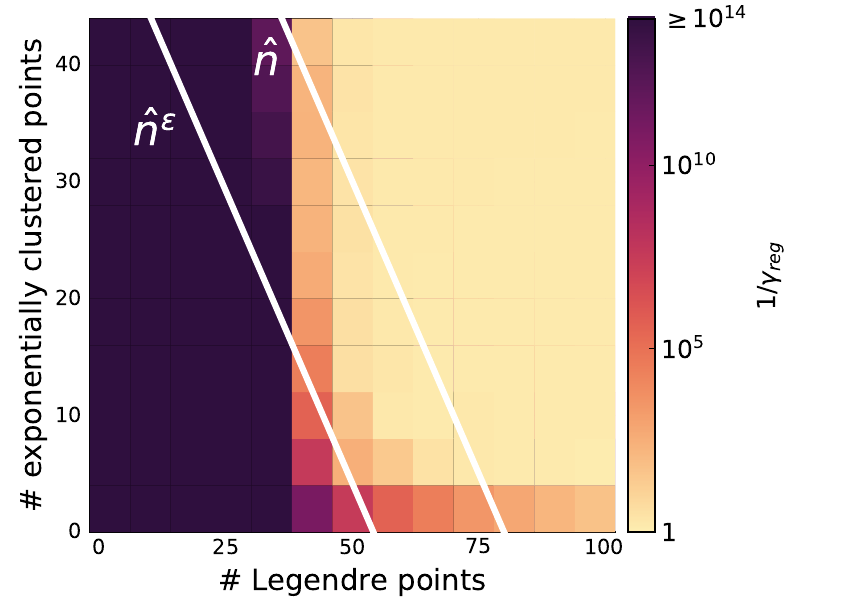}\hfill
    \includegraphics[width=.495\linewidth]{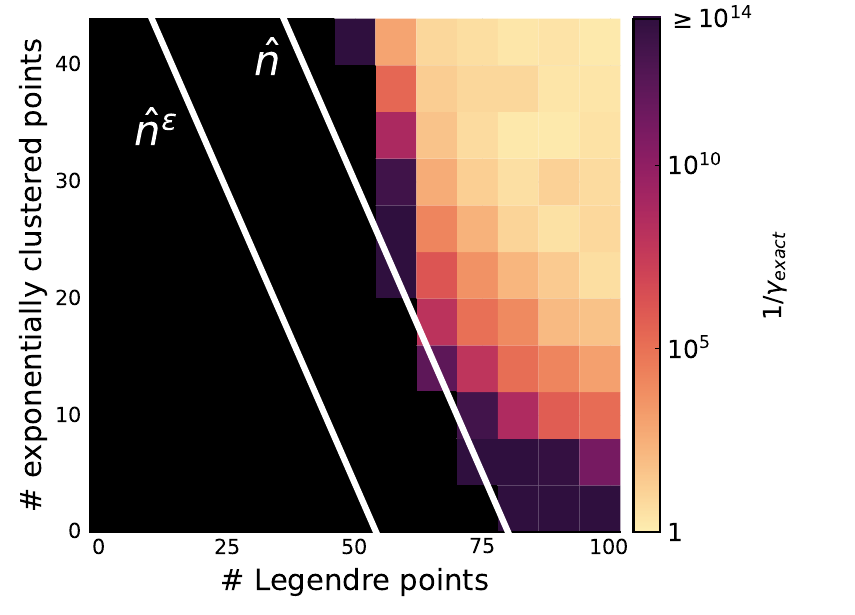} 
    \includegraphics[width=.495\linewidth]{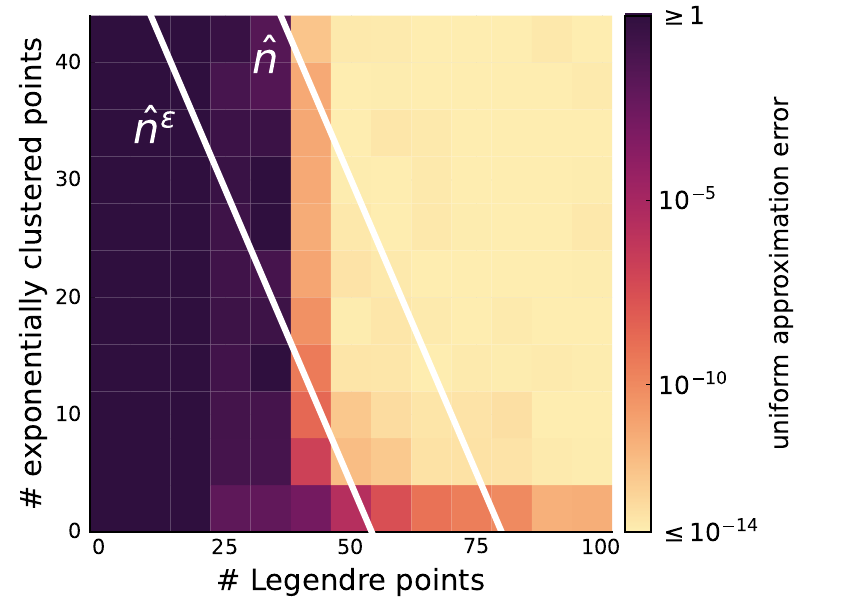}
    \caption{Top: comparison between $1/\gammar$ for $\epsilon = 10^{-14}$ (left) and $1/\gammae$ (right) for the Legendre + weighted Legendre basis~\eqref{eq:sumframe} with $n = 80$, for a varying number of Legendre points and exponentially clustered sample points. Bottom: the uniform approximation error—evaluated on an independent, dense grid—of the TSVD approximation of $f$~\eqref{eq:f} with threshold $\epsilon = 10^{-14}$. The black squares correspond to $\gammae = 0$. The white lines mark where the total number of sample points equals $\hat{n} = 80$ or $\hat{n}^\epsilon \approx 54.8$, i.e., the effective dimension introduced in the next section. We find that the top left figure is a good model for the actual error.}
    \label{fig:A_weightedlegendre}
\end{figure}

\section*{Part II: Random sampling in finite precision}

\section{Christoffel sampling with numerical rank-deficiency} \label{sec:randomizedsampling}
In this section, we analyze random sampling for least squares approximation using the Christoffel function. While this is well-studied in exact arithmetic, we focus on the effects of finite-precision arithmetic in the presence of numerical rank-deficiency.

\subsection{Regularized Christoffel sampling} \label{sec:christoffelfuns}
The analysis of random sampling becomes tractable after observing that the norm inequality~\eqref{thm:analytical:norminequality} is equivalent to a spectral inequality, that is,
\begin{align} \label{eq:spectralbound} 
    \begin{split} 
        \gammae \|v\|_{L^2_\rho} \leq \|v\|_{\mathcal{M}}, \quad \forall v \in \SPAN(\Phi) \qquad &\Leftrightarrow \qquad \gammae \|\mathcal{T}c\|_{L^2_\rho} \leq \|\mathcal{T}c\|_{\mathcal{M}}, \quad \forall c \in \mathbb{C}^n \\ &\Leftrightarrow \qquad \gammaesquared G \preceq G_d
    \end{split}
\end{align}
where $G_d, G \in \mathbb{C}^{n \times n}$ represent the operators $(\mathcal{M}\mathcal{T})^*\mathcal{M}\mathcal{T}$ and $\mathcal{T}^*\mathcal{T}$, respectively~\cite{cohen2013stability,cohen2017}. The same equivalence holds for any set spanning the space $\SPAN(\Phi)$. In particular, consider an $L^2_\rho(X)$-orthonormal basis $\{u_i\}_{i=1}^{\hat{n}}$ with synthesis operator $\mathcal{U}$, then $\mathcal{U}^*\mathcal{U} = I$ and
\[
    \gammae \|v\|_{L^2_\rho} \leq \|v\|_{\mathcal{M}}, \quad \forall v \in \SPAN(\Phi) \qquad \Leftrightarrow \qquad \gammaesquared I \preceq G_{d,\mathcal{U}}
\]
where $G_{d,\mathcal{U}}$ represents $(\mathcal{M}\mathcal{U})^*\mathcal{M}\mathcal{U}$. Hence, good sampling operators control the smallest eigenvalue of the discrete Gram matrix $G_{d,\mathcal{U}}$.

Consider now random pointwise sampling:
\begin{equation} \label{eq:pointwisesamplingop}
    \mathcal{M} \, :\, f \mapsto \{\sqrt{w_j} f(x_j)\}_{j=1}^m \qquad \text{where}\qquad  x_j \overset{\mathrm{iid}}{\sim} \mu, \; w_j = w(x_j)/m \qquad \text{with} \qquad w d\mu = d\rho,
\end{equation}
where $\|w^{-1}\|_{L^1_\rho} = 1$. Then, the smallest eigenvalue of $G_{d,\mathcal{U}}$ can be characterized via tools from random matrix theory. An important quantity turns out to be the the \textit{inverse Christoffel function}\footnote{The inverse of $k$ has been studied in the context of orthogonal polynomials long before its use in discrete least squares problems. In classical analysis, $k^{-1}$ is known as the Christoffel function~\cite{nevai1986geza}.}
\begin{equation} \label{eq:christoffelfunction}
    k(x) = \sum_{i=1}^{\hat{n}} \lvert u_i(x) \rvert^2,
\end{equation}
of which the definition does not depend on the choice of orthonormal basis. An important property is
\[
    \int_X k \, d\rho = \int_X \sum_{i=1}^{\hat{n}} \lvert u_i \rvert^2 \, d\rho = \sum_{i=1}^{\hat{n}} \int_X \lvert u_i \rvert^2 \, d\rho = \hat{n}.
\]
The following sampling result is well-known; see e.g.~\cite{cohen2013stability,hampton2015coherence,cohen2017,dolbeault2022optimal,adcockOptimalSamplingLeastsquares2024}. To accommodate $w = \hat{n}/k$, we assume that $k > 0$ almost everywhere. This is a rather mild and standard assumption.

\begin{proposition} \label{prop:christoffel}
    Let $0 < \delta < 1$ and let the sampling operator be defined by~\eqref{eq:pointwisesamplingop}. If
    \begin{equation} \label{eq:samplinglowerbound}
        m \geq 2.48 \|wk\|_{L^\infty} \log(\hat{n}/\delta),
    \end{equation}
    then~\eqref{thm:analytical:norminequality} is satisfied for $\gammae = 1/2$ with probability at least $1-\delta$. In particular, for $w = \hat{n}/k$, the condition becomes 
    \begin{equation} \label{eq:samplinglowerboundchrist}
        m \geq 2.48 \hat{n} \log(\hat{n}/\delta).
    \end{equation}
\end{proposition}
\begin{proof}
    Similar to the proof of~\cite[Theorem 2.1]{cohen2017}, we invoke the matrix Chernoff bound~\cite[Theorem 1.1]{tropp2012user} to bound the spectrum of $G_{d,\mathcal{U}}$. We require only that its smallest eigenvalue is larger than $1/4$ with probability at least $1-\delta$ to ensure that~\eqref{thm:analytical:norminequality} holds for $\gammae = 1/2$. This leads to
    \[
        m \geq \|w k\|_{L^\infty} \left(3/4 + \left( 1-3/4\right) \log \left( 1-3/4\right)\right)^{-1} \log(\hat{n}/\delta).\qedhere
    \]
\end{proof}

The sampling distribution associated with $w = \hat{n}/k$ requires only $\mathcal{O}(\hat{n} \log(\hat{n}))$ samples for accurate least squares fitting with high probability. This is close to interpolation, corresponding to $m = \hat{n}$. For a summary of optimality benchmarks and probabilistic error analyses, we refer to~\cite[\S1 and 2]{dolbeault2022optimal}. Also, we note that considerable efforts have been made to obtain $\mathcal{O}(\hat{n})$ sampling strategies instead of $\mathcal{O}(\hat{n} \log(\hat{n}))$; see \cite[\S8]{adcockOptimalSamplingLeastsquares2024} and references therein.

We now turn to random sampling for least squares approximation in the presence of numerical rank-deficiency. In the previous sections, we have found that $\ell^2$-regularized approximation—with a regularization parameter $\epsilon$ on the order of the unit round-off $u$—yields near-best numerical errors. Furthermore, we showed that regularization relaxes the discretization condition for approximation based on sampled data. Observe now that the relaxed discretization condition
\begin{equation} \label{eq:norminequalitylater}
    \gammar \left( \|\mathcal{T}c\|_{L^2_\rho} + \epsilon \|c\|_2 \right) \leq \|\mathcal{T}c\|_{\mathcal{M}} + \epsilon \|c\|_2, \qquad \forall c \in \mathbb{C}^n
\end{equation}
can also be linked to a spectral inequality, as described in the following proposition. This spectral inequality does depend on the choice of the representation $\Phi$.

\begin{proposition} \label{prop:eigenvalue}
    Consider the largest $\alpha$ so that $\alpha (G + \epsilon^2 I) \preceq G_d + \epsilon^2 I$ and the largest $\gammar$ satisfying~\eqref{eq:norminequalitylater}, then
    \[
        \sqrt{\alpha / 2} \leq \gammar \leq \sqrt{2 \alpha}.
    \]
\end{proposition}
\begin{proof}
    We have
    \[
        \alpha = \min_{c \in \mathbb{C}^n} \frac{c^* G_d c + \epsilon^2 c^* c}{c^* G c + \epsilon^2 c^* c} = \min_{c \in \mathbb{C}^n} \frac{\|\mathcal{T}c\|_{\mathcal{M}}^2 + \epsilon^2 \|c\|_2^2}{\|\mathcal{T}c\|_{L^2_\rho}^2 + \epsilon^2 \|c\|_2^2} \qquad \text{and} \qquad \gammar = \min_{c \in \mathbb{C}^n} \frac{\|\mathcal{T}c\|_{\mathcal{M}} + \epsilon \|c\|_2}{\|\mathcal{T}c\|_{L^2_\rho} + \epsilon \|c\|_2}.
    \]
    The final result now follows from $(x+y)/\sqrt{2}\leq \sqrt{x^2 + y^2} \leq x+y$, for all $x,y > 0$. 
\end{proof}

Once again, if we consider a random pointwise sampling operator as in~\eqref{eq:pointwisesamplingop}, we can characterize the associated smallest generalized eigenvalue using random matrix theory. In this case, the important quantity turns out to be the \textit{inverse regularized Christoffel function}
\begin{equation} \label{eq:christoffelfunctionreg}
    k^\epsilon(x) = \sum_{i=1}^{\hat{n}} \frac{\sigma_i^2}{\sigma_i^2 + \epsilon^2} \lvert u_i(x) \rvert^2 \qquad \text{where} \qquad \sum_{i=1}^{\hat{n}} \sigma_i u_i v_i^* \text{ is the SVD of $\mathcal{T}$}.
\end{equation}
Comparing this function to the inverse Christoffel function~\eqref{eq:christoffelfunction} shows that directions $u_i$ associated with small singular values $\sigma_i$ (relative to the regularization parameter $\epsilon$) are damped.

Associated with the function $k^\epsilon$ is an \textit{effective dimension} defined by 
\begin{equation} \label{eq:effectivedof}
    \hat{n}^\epsilon = \int_X k^\epsilon \, d\rho = \int_X \sum_{i=1}^{\hat{n}} \frac{\sigma_i^2}{\sigma_i^2 + \epsilon^2} \lvert u_i \rvert^2 \, d\rho = \sum_{i=1}^{\hat{n}} \frac{\sigma_i^2}{\sigma_i^2 + \epsilon^2} \int_X \lvert u_i \rvert^2 \, d\rho  = \sum_{i=1}^{\hat{n}} \frac{\sigma_i^2}{\sigma_i^2 + \epsilon^2}.
\end{equation}
Observe that $\hat{n}^\epsilon \leq \hat{n}$. The effective dimension is a measure for the number of well-represented directions in $\Phi$, which remain useful after finite-precision rounding. The following theorem shows that the effective dimension $\hat{n}^\epsilon$ serves a similar role as the dimension $\hat{n}$ in the presence of numerical rank-deficiency. We again assume that $k^\epsilon > 0$ almost everywhere.

\begin{theorem} \label{thm:christoffel2}
    Let $0 < \delta < 1$, let the sampling operator be defined by~\eqref{eq:pointwisesamplingop}, and let $\|\mathcal{T}\|_{L^2_\rho} \geq \epsilon > 0$. If 
    \begin{equation} \label{eq:samplinglowerbound2}
        m \geq (9.34 \|wk^\epsilon\|_{L^\infty} + 1.34) \log(8 \hat{n}^\epsilon/\delta),
    \end{equation}
    then~\eqref{thm:norminequality} is satisfied for $\gammar = 1/2$ with probability at least $1 - \delta$. In particular, for $w = \hat{n}^\epsilon/k^\epsilon$, the condition becomes 
    \begin{equation} \label{eq:samplinglowerboundchrist2}
        m \geq (9.34 \hat{n}^\epsilon + 1.34) \log(8 \hat{n}^\epsilon/\delta).
    \end{equation}
\end{theorem}
\begin{proof}
    From Prop.~\ref{prop:eigenvalue}, it follows that~\eqref{eq:norminequalitylater} holds for $\gammar = 1/2$ if $\alpha \geq 1/2$. Hence, we require that
    \[
        \frac{1}{2}(G + \epsilon^2 I) \preceq G_d + \epsilon^2 I
    \]
    holds with probability at least $1-\delta$. This expression is equivalent to 
    \[
        G - G_d \preceq \frac{1}{2}(G + \epsilon^2 I) \quad \Leftrightarrow \quad Y \preceq \frac{1}{2}I, \quad \text{where} \quad Y = S^{-1} Q^* \left( G - G_d \right) Q S^{-1},
    \]
    using the eigenvalue decomposition of $G + \epsilon ^2 I = Q S^2 Q^*$. Define the random Hermitian matrices
    \[
        X_j = \frac{1}{m} S^{-1} Q^* \left( G - w(x_j) \phi(x_j)^* \phi(x_j) \right) Q S^{-1}, \qquad 1 \leq j \leq m,
    \]
    where $\phi(x) = \begin{bmatrix} \phi_1(x) & \dots & \phi_n(x) \end{bmatrix}$ and $x_j \overset{\mathrm{iid}}{\sim} \mu$ with $w d\mu = d\rho$. Then,
    \[
        \sum_{j=1}^m X_j = Y \qquad \text{and} \qquad \EXP\left[X_j\right] = 0.
    \]
    For simplicity of notation, we introduce $\psi(x) = \phi(x) Q S^{-1}$ and 
    \[ 
        D = S^{-1} Q^* G Q S^{-1} = I - \epsilon^2 S^{-2} = \text{diag}(\sigma_1^2/(\sigma_1^2+\epsilon^2), \dots, \sigma_n^2 / (\sigma_n^2 + \epsilon^2)),
    \]
    where $\sigma_i$ denotes the $i$-th singular value of the synthesis operator $\mathcal{T}$ (if applicable, including those equal to zero, i.e., $\sigma_i$ for $\hat{n} < i \leq n$). For every $1 \leq j \leq m$, we bound
    \begin{align*}
        \lambda_{\max}(X_j) = \lambda_{\max}\left( \frac{1}{m}(D-w(x_j) \psi(x_j)^*\psi(x_j) )\right) &\leq \frac{1}{m} \left( \|D\|_2 + \|w(x_j) \psi(x_j)^* \psi(x_j) \|_2 \right) \\
        &= \frac{1}{m} \left(\frac{\sigma_1^2}{\sigma_1^2 + \epsilon^2} + w(x_j) \psi(x_j) \psi(x_j)^* \right) \\
        &\leq \frac{1 + \|wk^\epsilon\|_{L^{\infty}}}{m} \eqcolon L
    \end{align*}
    almost surely, using $\psi(x)\psi(x)^* = k^\epsilon(x)$ due to Prop.~\ref{prop:grammatrixform} below. Moreover, we have
    \begin{align*}
        \EXP\left[ X_j^2 \right] &= \EXP\left[ (\EXP\left[w(x_j)\psi(x_j)^*\psi(x_j)\right] - w(x_j)\psi(x_j)^*\psi(x_j))^2/m^2 \right] \\ &\preceq \EXP\left[ (w(x_j)\psi(x_j)^*\psi(x_j))^2/m^2 \right] \\
        &= \EXP\left[ w(x_j)^2\psi(x_j)^*\psi(x_j)\psi(x_j)^*\psi(x_j)/m^2 \right] \\
        &= \EXP\left[ w(x_j)^2 k^\epsilon(x_j) \psi(x_j)^*\psi(x_j) / m^2 \right] \\
        &\preceq D \|wk^\epsilon\|_{L^\infty} / m^2,
    \end{align*}
    almost surely, using $D = \EXP[w(x_j)\psi(x_j)^*\psi(x_j)]$. This leads to
    \[
        \EXP\left[ Y^2 \right] = \sum_{j=1}^m \EXP\left[X_j^2\right] \preceq D \frac{\|wk^\epsilon\|_{L^\infty}}{m} \eqcolon V.
    \]
    Using 
    \[
        v = \|V\|_2 = \frac{\|wk^\epsilon\|_{L^\infty}}{m} \sigma_1^2/(\sigma_1^2 + \epsilon^2) \quad \text{and} \quad d = \text{intdim}(V) = \frac{\text{tr}(V)}{\|V\|_2} = \frac{\hat{n}^\epsilon}{\sigma_1^2/(\sigma_1^2 + \epsilon^2)},
    \]
    it follows from~\cite[Thm.\ 7.7.1]{tropp2015introduction} and $1/2 \leq \sigma_1^2/(\sigma_1^2 + \epsilon^2) \leq 1$ (due to $\|\mathcal{T}\|_{L^2_\rho} \geq \epsilon$) that
    \begin{equation} \label{eq:probinproof}
        \Pr \left( \lambda_{\max}(Y) \geq 1/2 \right) \leq 4 d \exp\left( \frac{-1/8}{v + L/6}\right)
        \leq 8 \hat{n}^\epsilon \exp\left( \frac{-m/8}{(7\|wk^\epsilon\|_{L^\infty} + 1)/6} \right).
    \end{equation}
    In order for this probability to be smaller than $\delta$, we require that
    \begin{equation} \label{eq:probinproof2}
        m \geq ((28/3)\|wk^\epsilon\|_{L^\infty} + (4/3))\log(8\hat{n}^\epsilon / \delta).
    \end{equation}
    Following~\cite[Thm.\ 7.7.1]{tropp2015introduction},~\eqref{eq:probinproof} holds whenever $1/2 \geq \sqrt{v} + L/3$, which is satisfied when $m$ satisfies~\eqref{eq:probinproof2}, using $\hat{n}^\epsilon \geq 1/2$ (due to $\|\mathcal{T}\|_{L^2_\rho} \geq \epsilon$).
\end{proof}

\begin{remark}
    The difference in constants between Prop.~\ref{prop:christoffel} and Thm.~\ref{thm:christoffel2} arises for two main reasons. First, obtaining a bound that depends on the effective dimension requires the use of a matrix concentration inequality that depends on the intrinsic\footnote{The intrinsic dimension of a matrix $M$ equals $\text{tr}(M)/\|M\|_2$.} rather than the ambient dimension. This necessitates the use of a matrix Bernstein inequality, since the corresponding matrix Chernoff bound for the smallest eigenvalue is not currently available; see~\cite[Ch.7]{tropp2015introduction}. A second factor is the gap between $\gammar$ and the associated spectral inequality, as described in Prop.~\eqref{prop:eigenvalue}.
\end{remark}

In a fully discrete setting, sampling with a density proportional to the inverse Christoffel function is equivalent to sampling according to the \textit{(statistical) leverage scores} arising in least squares problems. Indeed, for row subsampling of a tall-and-skinny matrix, leverage score sampling yields a subspace embedding with high probability once $m \gtrsim \hat{n} \log(\hat{n}/\delta)$~\cite{drineas2012fast}. The regularized function $k^\epsilon$ plays the same role for $\ell^2$-regularized least squares and corresponds to the notion of \textit{ridge leverage scores}~\cite{alaoui2015fast}. The associated trace quantity $\hat{n}^\epsilon$ is also commonly referred to as the \textit{effective degrees of freedom} or the \textit{statistical dimension}. Thm.~\ref{thm:christoffel2} can be viewed as a continuous counterpart of ridge leverage score sampling results. Similar continuous counterparts have appeared, for example, in the setting of kernel ridge regression~\cite{avron2017random}. 

\subsection{Understanding the influence of regularization}
To illustrate the influence of regularization, Figure~\ref{fig:christoffel} compares $k$ and $k^\epsilon$ for the examples of Section~\ref{sec:numericalexamples}. Together with Prop.~\ref{prop:christoffel} and Thm.~\ref{thm:christoffel2}, Figure~\ref{fig:christoffel} can be used to assess the sampling requirements for regularized and unregularized approximation. Notably, the discretization condition associated with regularized approximation is automatically satisfied whenever the corresponding condition for unregularized approximation holds; see, e.g., Section~\ref{sec:discretizationcondition}. Consequently, Prop.~\ref{prop:christoffel} also applies to the regularized setting.

Consider regularized approximation with the specific choices $w = \hat{n}/k$ and $w = \hat{n}^\epsilon/k^\epsilon$. The required number of samples then scales with $\hat{n}$ and $\hat{n}^\epsilon$, respectively (ignoring logarithmic factors). In our examples, accounting for regularization—i.e., choosing the latter weight—yields only a constant-factor improvement, since the effective dimension $\hat{n}^\epsilon$ typically scales proportionally with $\hat{n}$ as $n$ varies. Moreover, the leading constant in Thm.~\ref{thm:christoffel2} is larger than that in Prop.~\ref{prop:christoffel}. 

Nevertheless, incorporating regularization into the analysis remains important for two main reasons. First, even if $\hat{n}^\epsilon=\|k^\epsilon\|_{L^1_\rho}$ is close to $\hat{n} = \|k\|_{L^1_\rho}$, the pointwise behavior of $k^\epsilon$ and $k$ may differ substantially. This can strongly affect the number of samples required for a prescribed sampling distribution. An example is provided by the Fourier extension frame considered in Figure~\ref{fig:christoffel}(a) and Section~\ref{sec:fourierextension}. For uniformly random samples (constant $w$), the two quantities
\[
    \|k\|_{L^\infty} = \mathcal{O}(n^2), \qquad \|k^\epsilon\|_{L^\infty} = \mathcal{O}(n \log(n))
\]
are substantially different. We find that the required number of samples for accurate numerical approximation is $\mathcal{O}(n\log(n)^2)$, whereas an analysis ignoring regularization would predict $\mathcal{O}(n^2\log(n))$.

Second, the sampling distribution associated with the Christoffel function is often computed numerically. In Section~\ref{sec:computabilitychrist}, we show that $k^\epsilon$ can be evaluated accurately using a backward stable algorithm, whereas no such guarantees can be made for $k$ in the presence of numerical rank-deficiency. Consequently, the regularized Christoffel function is highly relevant for the practical computation of sampling densities.

Finally, the influence of regularization can be further understood through an extremal characterization of the Christoffel function. It is well known~\cite{nevai1986geza,dolbeault2022optimal} that
\[
    k(x) = \max_{\substack{v \in \SPAN(\Phi) \\ v \neq 0}} \frac{\vert v(x)\vert^2}{\|v\|_{L^2_\rho}^2}.
\]
This characterization shows that a point $x\in X$ becomes important for sampling when there exists a function in the approximation space that is large at $x$ while remaining small elsewhere. Intuitively, this ensures that every function in the approximation space is visible on the sampling grid.

A similar property holds for the regularized Christoffel function, with the difference that functions with large expansion coefficients are penalized. In this sense, weakly represented directions may be less visible on the sampling grid than well-represented directions. The proof of the following proposition is given in Appendix~\ref{sec:app2}.

\begin{proposition} \label{prop:extremalproperty}
    Let $\Phi \subset L^2(X,\rho)$ and $\epsilon > 0$. It holds that
    \begin{equation} \label{eq:extremalprop}
        k^\epsilon(x) = \max_{\substack{c \in \mathbb{C}^n \\ c \neq 0}} \frac{\lvert \phi(x) c \rvert^2}{\|\mathcal{T}c \|_{L^2_\rho}^2 + \epsilon^2 \|c\|_2^2} \, ,
    \end{equation}
    where $\phi(x) = \begin{bmatrix} \phi_1(x) & \dots & \phi_n(x) \end{bmatrix} \in \mathbb{C}^{1 \times n}$.
\end{proposition}

\begin{figure}
    \begin{subfigure}{\linewidth}
        \includegraphics[width=.49\linewidth]{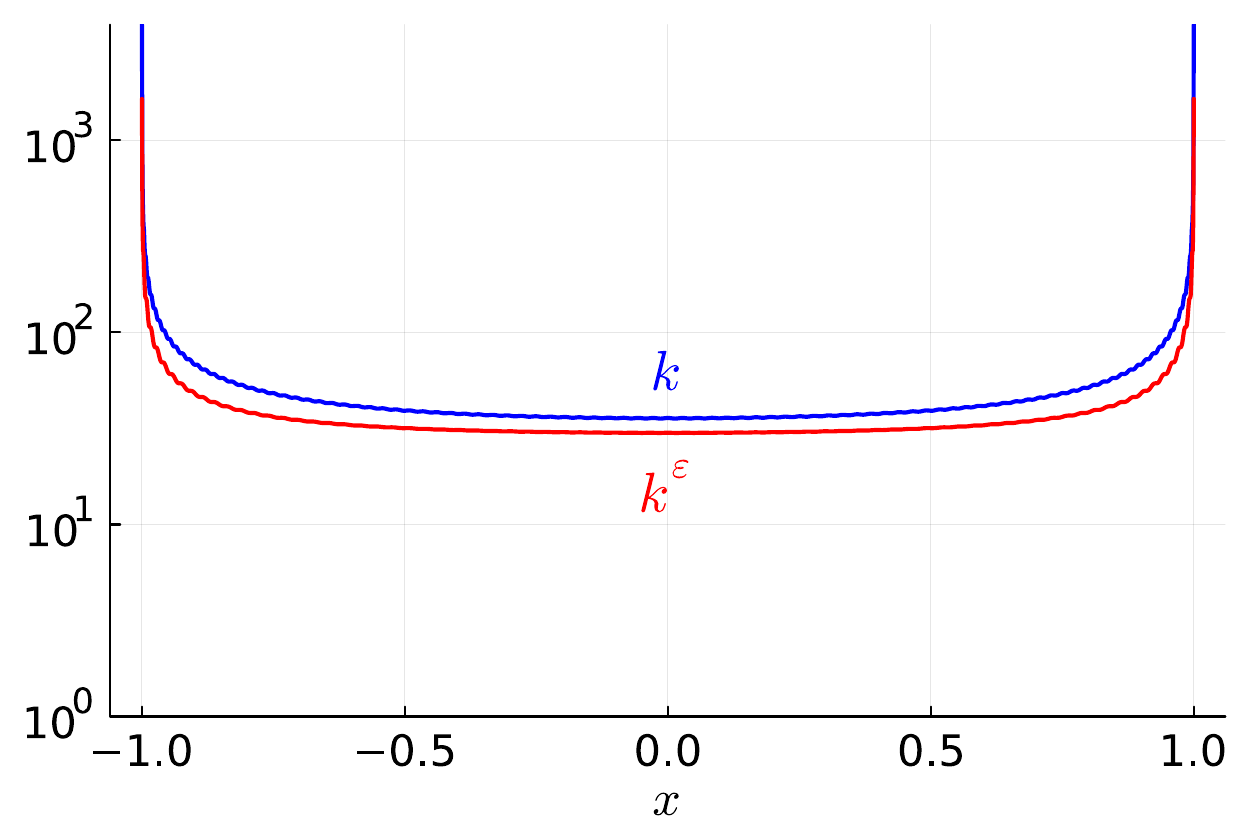}\hfill
        \includegraphics[width=.49\linewidth]{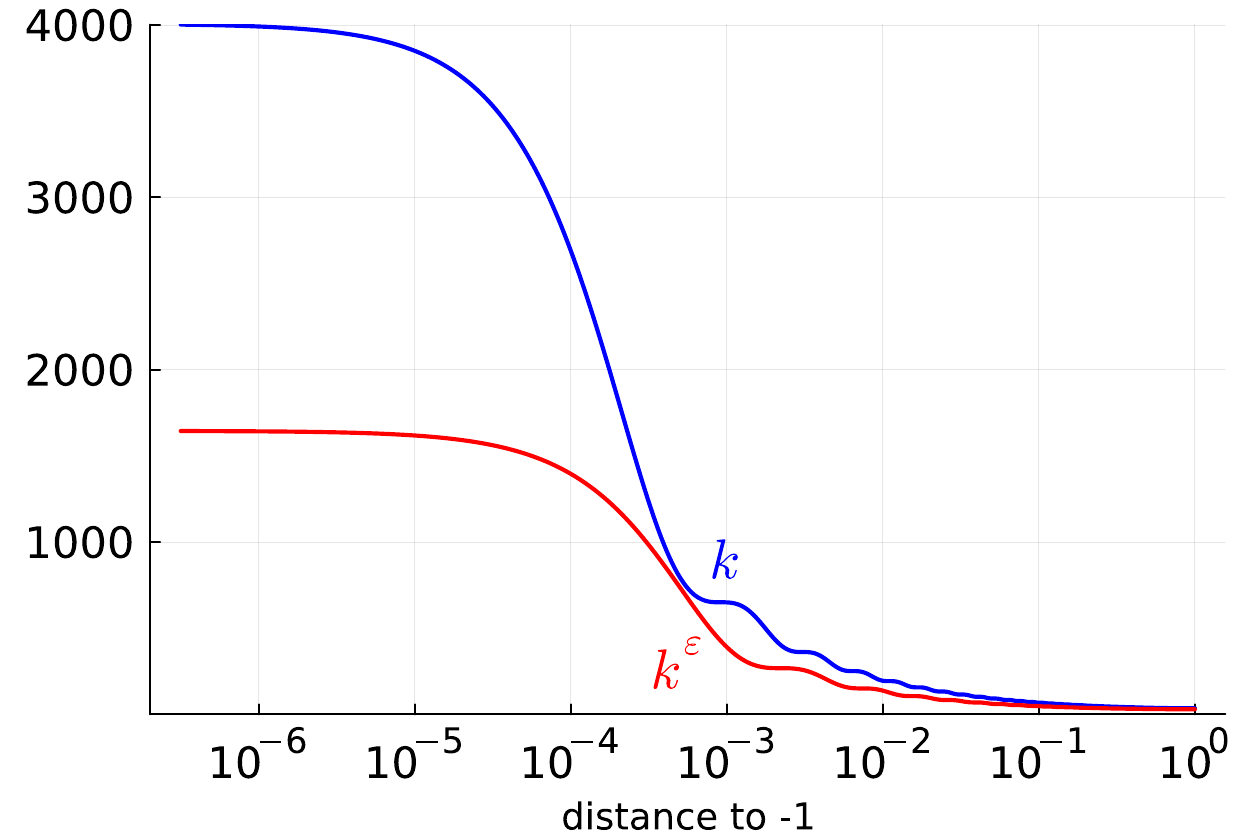}
        \caption{Fourier extension frame~\eqref{eq:numexampleapprox} for $n = \hat{n} = 101$ ($\hat{n}^\epsilon \approx 76.8$)}
    \end{subfigure}
    \begin{subfigure}{\linewidth}
        \includegraphics[width=.49\linewidth]{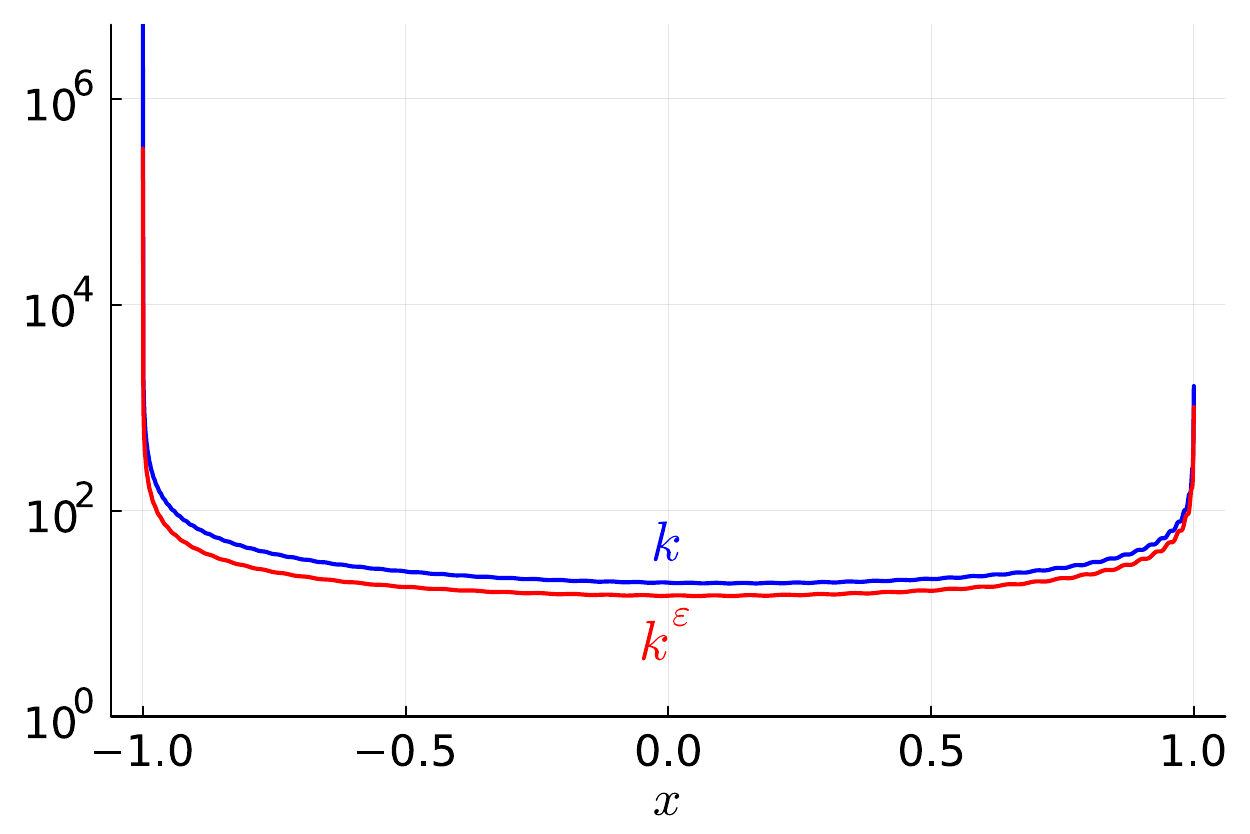}\hfill
        \includegraphics[width=.49\linewidth]{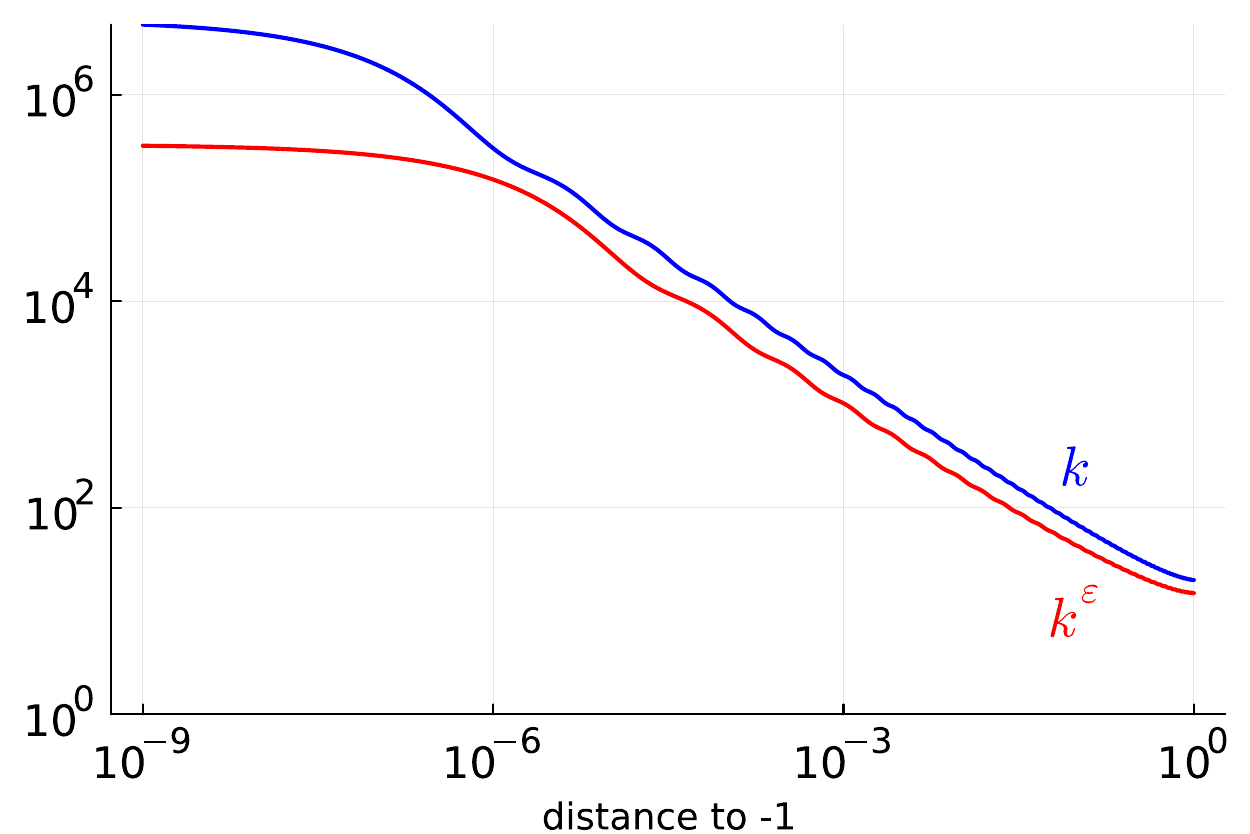}
        \caption{Legendre + weighted Legendre basis~\eqref{eq:sumframe} for $n = \hat{n} = 80$ ($\hat{n}^\epsilon \approx 54.8 $)}
    \end{subfigure}
    \caption{Comparison between $k$ and $k^\epsilon$ with $\epsilon = 10^{-14}$, for the numerical examples of Section~\ref{sec:numericalexamples}. We observe that $k^\epsilon$ and $k$ may differ significantly in their pointwise behaviour. }
    \label{fig:christoffel}
\end{figure}

\subsection{Backward stable computation of the Christoffel function} \label{sec:computabilitychrist}
In practice, the Christoffel function is often evaluated numerically to compute a near-optimal sampling distribution for a specific basis. This naturally raises the question of the expected accuracy of such numerical evaluations. All omitted proofs can be found in Appendix~\ref{sec:app2}.

Consider the following alternative characterization for the inverse (regularized) Christoffel function. 

\begin{proposition} \label{prop:grammatrixform}
    Let $\Phi \subset L^2(X,\rho)$ and $\epsilon > 0$. It holds that
    \begin{equation} \label{eq:christcharacterization1}
        k(x) = \phi(x) G^\dagger \phi(x)^* \qquad \text{and} \qquad k^\epsilon(x) = \phi(x) (G + \epsilon^2 I)^{-1} \phi(x)^*, 
    \end{equation}
    where $\phi(x) = \begin{bmatrix} \phi_1(x) & \dots & \phi_n(x) \end{bmatrix} \in \mathbb{C}^{1 \times n}$.
\end{proposition}

Since the Gram matrix is generally unknown in the case of a non-orthogonal basis, a workaround is to approximate it using a discrete Gram matrix computed over a dense grid~\cite{dolbeault2022optimal}. Note that this computation can be performed offline, i.e., it can be computed only once for every $\Phi$ and independently of the function to be approximated. 

More specifically, given a sampling operator\footnote{To avoid confusion with the sampling operator $\mathcal{M}$ introduced in \S\ref{sec:discretization}, we denote the present operator by $\mathcal{M}_{dg}$. The subscript emphasizes that the grid may be dense since the points are usually distributed suboptimally. Accordingly, we use $m_{dg}$ to denote the number of sample points rather than $m$.} $\mathcal{M}_{dg} \from L^\infty(X) \to \mathbb{C}^{m_{dg}}$ and the associated system matrix $A \in \mathbb{C}^{m_{dg} \times n}$ representing the operator $\mathcal{M}_{dg}\mathcal{T}$, we define 
\begin{equation} \label{eq:discretechristoffel}
    k^d(x) = \phi(x) (A^*A)^\dagger \phi(x)^* \qquad \text{and} \qquad k^{d,\epsilon}(x) = \phi(x) (A^*A + \epsilon^2 I)^{-1} \phi(x)^*.
\end{equation}
The following theorem shows that we can expect this discrete version to lie close to the exact one, whenever the sampling operator satisfies a two-sided norm equivalence. In particular, for a relative accuracy of $0 \leq \eta < 1$, it is required that
\begin{equation} \label{eq:normequivalence1}
    \frac{1}{1+\eta} \|v\|_{L^2_\rho}^2 \leq \|v\|_{\mathcal{M}_{dg}}^2 \leq \frac{1}{1-\eta} \|v\|_{L^2_\rho}^2, \qquad \forall v \in \SPAN(\Phi)
\end{equation}
and 
\begin{equation} \label{eq:normequivalence2}
    \frac{1}{1+\eta} \left( \|\mathcal{T}c\|_{L^2_\rho}^2 + \epsilon^2 \|c\|_2^2 \right) \leq \|\mathcal{T}c\|_{\mathcal{M}_{dg}}^2 + \epsilon^2 \|c\|_2^2 \leq \frac{1}{1-\eta} \left(\|\mathcal{T}c\|_{L^2_\rho}^2 + \epsilon^2 \|c\|_2^2 \right), \quad \forall c \in \mathbb{C}^n
\end{equation}
in the unregularized and regularized case, respectively. For the unregularized case, this result is analogous to~\cite[Lemma 3.3]{dolbeault2022optimal}. Furthermore, we show that the discrete Christoffel functions can be evaluated without the use of the normal equations $A^*A$ via~\eqref{eq:christcharacterization2}.

\begin{proposition} \label{prop:discretechristoffel}
    For any $0 \leq \eta < 1$, it holds that
    \begin{equation*} 
        \eqref{eq:normequivalence1} \; \Rightarrow \; \left\vert \frac{k(x) - k^d(x)}{k(x)} \right\vert \leq \eta \qquad \text{and} \qquad \eqref{eq:normequivalence2} \; \Rightarrow \; \left\vert \frac{k^\epsilon(x) - k^{d,\epsilon}(x)}{k^\epsilon(x)} \right\vert \leq \eta, \qquad \forall x \in X.
    \end{equation*}
    Furthermore,
    \begin{equation} \label{eq:christcharacterization2}
        k^d(x) = \norm{\phi(x) A^\dagger}_2^2 \qquad \text{and} \qquad k^{d,\epsilon}(x) = \norm{\phi(x) \begin{bmatrix} A \\ \epsilon I \end{bmatrix}^\dagger}_2^2.
    \end{equation}
\end{proposition}

We can now investigate whether the discrete Christoffel functions can be evaluated accurately using numerical algorithms. As discussed in Section~\ref{sec:backwardstability}, numerical algorithms are typically modeled as being backward stable, reflecting the presence of rounding errors both in the data and during execution.

Consider the mapping
\[
    (A, v) \mapsto \norm{v \begin{bmatrix} A \\ \epsilon I\end{bmatrix}^\dagger}_2^2
\]
which evaluates $k^{d,\epsilon}(x)$ when $v = \phi(x)$. Thm.~\ref{thm:computability} below shows that a backward stable algorithm yields accurate results whenever the regularization parameter $\epsilon$ is sufficiently large relative to $u$. The proof relies on perturbation results for pseudo-inverses, whose behaviour is well controlled in the presence of regularization.

One may ask whether an analogous result holds in the unregularized case, i.e., for
\begin{equation} \label{eq:discchristmap}
    (A, v) \mapsto \norm{v A^\dagger}_2^2
\end{equation}
which evaluates $k^{d}(x)$ when $v = \phi(x)$. In this case, it cannot be expected that $(A+\Delta A)^\dagger$ and $A^\dagger$ behave similarly when $A$ is numerically rank-deficient. Indeed, if $\sigma_{\min}(A) \lesssim \|\Delta A\|_2$, the perturbation $\Delta A$ generally alters the behaviour of the pseudo-inverse significantly. Thm.~\ref{thm:computability2} below formalizes this intuition. In conclusion, while $k^\epsilon$ can be evaluated numerically to high accuracy, no such guarantees exist for the evaluation of $k$ in the presence of numerical rank-deficiency.

\begin{theorem} \label{thm:computability}
    Given $A \in \mathbb{C}^{m_{dg} \times n}$, $v \in \mathbb{C}^{1 \times n}$ and $\epsilon > 0$, let
    \[
        y = \norm{v \begin{bmatrix} A \\ \epsilon I\end{bmatrix}^\dagger}_2^2  \qquad \text{and} \qquad \hat{y} = \norm{(v+\Delta v) \begin{bmatrix} (A+\Delta A) \\ \epsilon I\end{bmatrix}^\dagger}_2^2 
    \]
    for some $\|\Delta A\|_2 \leq u \Calg \|A\|_2$ and $\|\Delta v\|_2 \leq u \Calg \|v\|_2$ with $\Calg > 0$. Let $\epsilon = \alpha u \Calg \|A\|_2$ for some $0 < \alpha < (u\Calg)^{-1}$, then 
    \[
        \left\vert \frac{y - \hat{y}}{y} \right\vert \leq \frac{4+2\sqrt{2}}{\alpha} + \frac{6+4\sqrt{2}}{\alpha^2} \, .
    \]
\end{theorem}
\begin{proof}
    For notational convenience, we define
    \[
        \Aext = \begin{bmatrix} A \\ \epsilon I \end{bmatrix} \qquad \text{and} \qquad \Aexthat = \begin{bmatrix} A + \Delta A \\ \epsilon I \end{bmatrix}.
    \]
    Since $y,\hat{y} > 0$, we have
    \begin{equation} \label{eq:auxsplit}
        \left\vert \frac{y - \hat{y}}{y} \right\vert =  \frac{\left\vert \sqrt{y} - \sqrt{\hat{y}} \right\vert}{\sqrt{y}}  \frac{\sqrt{y} + \sqrt{\hat{y}}}{\sqrt{y}} \leq \frac{\left\vert \sqrt{y} - \sqrt{\hat{y}} \right\vert}{\sqrt{y}} \left( 2 + \frac{\lvert \sqrt{y} - \sqrt{\hat{y}} \rvert}{\sqrt{y}} \, \right).
    \end{equation}
    Using the (reverse) triangle inequality, we find that
    \begin{align*}
        \left\vert \sqrt{y} - \sqrt{\hat{y}} \right\vert &= \left\vert \norm{v \Aext^\dagger}_2 - \norm{(v+\Delta v)\Aexthat^\dagger}_2 \right\vert \\
        &\leq \norm{v \Aext^\dagger - (v+\Delta v)\Aexthat^\dagger}_2 \\
        &\leq \norm{v \left(\Aext^\dagger - \Aexthat^\dagger \right)}_2 + \norm{\Delta v \Aexthat^\dagger}_2.
    \end{align*}
    Consider the change-of-variables $w = v \Aext^\dagger$. Since $\Aext$ has full column rank, $\Aext^\dagger \Aext = I$ such that $v = w \Aext$, leading to
    \begin{align*}
        \frac{ \left\vert \sqrt{y} - \sqrt{\hat{y}} \right\vert}{\sqrt{y}} &\leq \frac{\norm{w \Aext\left(\Aext^\dagger - \Aexthat^\dagger \right)}_2 + \norm{\Delta v \Aexthat^\dagger}_2}{\norm{w}_2} \\
        &\leq \norm{\Aext\left(\Aext^\dagger - \Aexthat^\dagger \right)}_2 + \frac{\|\Delta v\|_2}{\|w\|_2} \norm{\Aexthat^\dagger}_2 \\
        &\leq \norm{\Aext \Aext^\dagger - \Aexthat \Aexthat^\dagger}_2 + \norm{(\Aexthat - \Aext) \Aexthat^\dagger}_2 + u \Calg \norm{\Aext}_2 \norm{\Aexthat^\dagger}_2 \\
        &\leq \left(2\|\Delta A\|_2 + u \Calg \norm{\Aext}_2 \right) \norm{\Aexthat^\dagger}_2 \\
        &\leq \frac{u \Calg}{\epsilon} \left( 2\|A\|_2 + \sqrt{\|A\|_2^2 + \epsilon^2} \right)
    \end{align*}
    where in the penultimate step we use~\cite[Thm.\ 2.3 and 2.4]{stewart1977perturbation} to bound the first term, and in the last step we use that $\norm{\Aexthat^\dagger}_2 \leq 1/\epsilon$ and $\norm{\Aext}_2 \leq \sqrt{\|A\|_2^2 + \epsilon^2}$. Substituting $\epsilon = \alpha u \Calg \|A\|_2$ and using $\alpha \leq (u\Calg)^{-1}$, we get
    \[
        \frac{ \left\vert \sqrt{y} - \sqrt{\hat{y}} \right\vert}{\sqrt{y}} \leq \frac{1}{\alpha} (2 + \sqrt{1 + (\epsilon/\|A\|_2)^2}) \leq \frac{2 + \sqrt{2}}{\alpha},
    \]
    which gives the final result when combined with~\eqref{eq:auxsplit}.
\end{proof}

\begin{theorem} \label{thm:computability2}
   Let $\Phi \subset L^\infty(X)$ and let $\mathcal{M}_{dg} \from L^\infty(X) \to \mathbb{C}^{m_{dg}}$ be an ideal sampler satisfying $A^*A = G$ where $A \in \mathbb{C}^{m_{dg} \times n}$ represents $\mathcal{M}_{dg} \mathcal{T}$. Let $\mathcal{T} = \sum_{i=1}^{\hat{n}} \sigma_i u_i(x) v_i^*$ be the singular value decomposition of the synthesis operator associated with $\Phi$. Then, for every
    \begin{equation} \label{eq:smallsigmaset}
        x \in \bigcup_{i \; : \; 0 < \sigma_i \leq u \sigma_1} \left\{ x \in  X \; : \; u_i(x) \neq 0 \right\}
    \end{equation}
    it holds that
    \[
        \sup_{\Delta A \; : \; \|\Delta A\|_2 \leq u \|A\|_2}\norm{\phi(x) (A+\Delta A)^\dagger}_2^2 = \infty.
    \]
\end{theorem}
\begin{proof}
    Since $A^*A = G$, $A$ has a (thin) singular value decomposition of the form $Q \Sigma V^*$, where $\Sigma \in \mathbb{R}^{\hat{n}\times\hat{n}}$ and $V \in \mathbb{C}^{n\times\hat{n}}$ are shared with the singular value decomposition of the synthesis operator $\mathcal{T}$, i.e., $v_i$ is the $i$-th column of $V$ and $(\Sigma)_{i,i} = \sigma_i$. For fixed $x$ satisfying~\eqref{eq:smallsigmaset}, consider any $j$ such that $0 < \sigma_j \leq u \sigma_1$ and $u_j(x) \neq 0$, and choose
    \[
        \Delta A = (-1 + \delta) \sigma_j q_j v_j^* \qquad \text{so that} \qquad A +\Delta A= \sum_{i=1, j\neq i}^{\hat{n}} \sigma_i q_i v_i^* + \delta \sigma_j q_j v_j^*,
    \]
    where $q_i$ is the $i$-th column of $Q \in \mathbb{C}^{m_{dg} \times \hat{n}}$ and $0 < \delta < 1$. Observe that $\|\Delta A\|_2 = (1-\delta)\sigma_j \leq u \|A\|_2$. Furthermore,
    \begin{align*}
        \norm{\phi(x) (A+\Delta A)^\dagger}_2^2 &= \norm{ \left( \sum_{i=1}^{\hat{n}} \sigma_i u_i(x) v_i^* \right) \left( \sum_{i=1}^{\hat{n}} \frac{1}{\sigma_i} v_i q_i^* + \frac{1-\delta}{\delta \sigma_j} v_j q_j^* \right) }_2^2 \\
        &= \norm{ \sum_{i=1}^{\hat{n}} u_i(x) q_i^* + \frac{1-\delta}{\delta} u_j(x) q_j^* }_2^2 \\
        &\geq \left(\norm{\frac{1-\delta}{\delta} u_j(x) q_j^*}_2 - \norm{ \sum_{i=1}^{\hat{n}} u_i(x) q_i^*}_2 \right)^2 \\
        &= \left( \frac{1-\delta}{\delta} \lvert u_j(x) \rvert - \sqrt{k(x)} \right)^2.
    \end{align*}
    This lower bound grows unboundedly as $\delta$ approaches zero, since $k(x)$ is finite and $\lvert u_j(x) \rvert \neq 0$.
\end{proof}


\section{Discretization of Fourier extension frames} \label{sec:fourierextension}

Approximation of functions on irregular domains $X \subset \mathbb{R}^d$ arises in many computational problems. Often, an orthonormal or Riesz basis is not available for such domains, such that numerical methods resort to using an orthonormal basis defined on a surrounding regular domain, such as a bounding box. Crucially, these bases are no longer orthonormal on the underlying domain $X$ itself; instead, they form a frame. The resulting redundancy leads to numerical rank-deficiency.

To illustrate the key ideas in a concrete setting, we focus on the univariate case. Specifically, for approximation on $X = [-W,W]$ with $W < 1/2$, one may use the Fourier basis defined on the extension $[-1/2,1/2]$:
\begin{equation} \label{eq:fourierextension}
    \Phi = \{ \exp(2 \pi i x k) \}_{k = -(n-1)/2}^{(n-2)/2},
\end{equation}
for odd $n$. Smooth functions, which may be non-periodic on $[-W,W]$, can be approximated by elements in $\SPAN(\Phi)$ with exponential convergence in $n$~\cite{huybrechsFourierExtensionNonperiodic2010}. Regularized approximation has been studied extensively in~\cite{adcockNumericalStabilityFourier2014}, where it is shown that regularization yields a slower decay of the approximation error but does not prevent convergence down to unit roundoff. Furthermore, note that $\hat{n} = n$, due to $\Phi$ being linearly independent.

Our goal is to analyze how regularization affects the discretization problem using the results of Section~\ref{sec:randomizedsampling}. This analysis hinges on a good understanding of the singular value decomposition of the synthesis operator $\mathcal{T}$. While little is known about this decomposition for general finite bases, in the case of the Fourier extension frame the singular vectors are the discrete prolate spheroidal wave functions (DPSWFs), which have been extensively studied in signal processing. For a detailed explanation of this connection, see~\cite[\S3.1]{matthysenFunctionApproximationArbitrary2018a}. DPSWFs are an indispensable tool for bandlimited extrapolation, introduced and thoroughly analyzed by Slepian and his collaborators~\cite{slepianProlateSpheroidalWave1978}. In this context, a notion of effective dimension is not new; see \cite{landauProlateSpheroidalWave1962} and~\cite[\S2.3]{daubechiesTenLecturesWavelets1992}.

\subsection{Typical singular value decomposition of a truncated frame} \label{sec:ridgesvd}
An infinite-dimensional sequence $\{\phi_i\}_{i=1}^\infty$ is called a frame for $L^2(X,\rho)$ if there exist constants $A,B > 0$ such that 
\begin{equation} \label{eq:frameinequality}
    A \|f\|_{L^2_\rho}^2 \leq \sum_{i=1}^\infty \lvert \langle f, \phi_i \rangle_{L^2_\rho} \rvert^2 \leq B \|f\|_{L^2_\rho}^2
\end{equation}
for all $f \in L^2(X,\rho)$~\cite[Def.\ 5.1.1]{christensen2003introduction}. This definition generalizes the notion of a Riesz basis by permitting redundancy. 

In~\cite[\S4]{fna1}, it is shown that sufficiently long finite subsequences of linearly independent and overcomplete frames are numerically rank-deficient\footnote{A set is linearly independent if every finite subsequence is linearly independent. A frame is overcomplete if there exists a sequence $c \in \ell^2(\mathbb{N})$ such that $\sum_{i=1}^\infty c_i \phi_i = 0$~\cite[\S3-4]{christensen2003introduction}.}. 
Since the Fourier extension sequence forms a linearly independent, overcomplete frame (with frame bounds $A = B = 1$), it follows that the truncated family~\eqref{eq:fourierextension} becomes numerically rank-deficient for sufficiently large $n$.

An important consequence of~\eqref{eq:frameinequality} is that for every $f \in L^2(X,\rho)$ there exist coefficients $c \in \ell^2(\mathbb{N})$ such that
\[
    f = \sum_{i=1}^\infty c_i \phi_i \qquad \text{and} \qquad \|c\|_2 \leq \|f\|_{L^2_\rho} / \sqrt{A}, \qquad \text{\cite[Lm.\ 5.5.5]{christensen2003introduction}}.
\]
Thus, although a frame expansion is generally non-unique, at least one representation has controlled norm of the coefficients. Since small coefficient norms are essential for small numerical errors as discussed in Part I, this property explains the relevance of frames for numerical approximation. Further examples of frames in numerical approximation can be found in~\cite{fna1,adcockFramesNumericalApproximation2020} and references therein. 

Figure~\ref{fig:typicalsvd} shows a typical singular value profile of the synthesis operator associated with a finite subsequence of a frame with $A = B = 1$. While the spectrum of the synthesis operator for the infinite-dimensional frame is contained in $\{0\} \cup [A,B]$, the finite-dimensional truncation exhibits additional singular values in the interval $(0,A)$. These spurious singular values arise from spectral pollution~\cite{daviesSpectralPollution2004} and are the cause for numerical rank-deficiency.

\begin{figure}
    \centering
    \includegraphics[width=.5\linewidth]{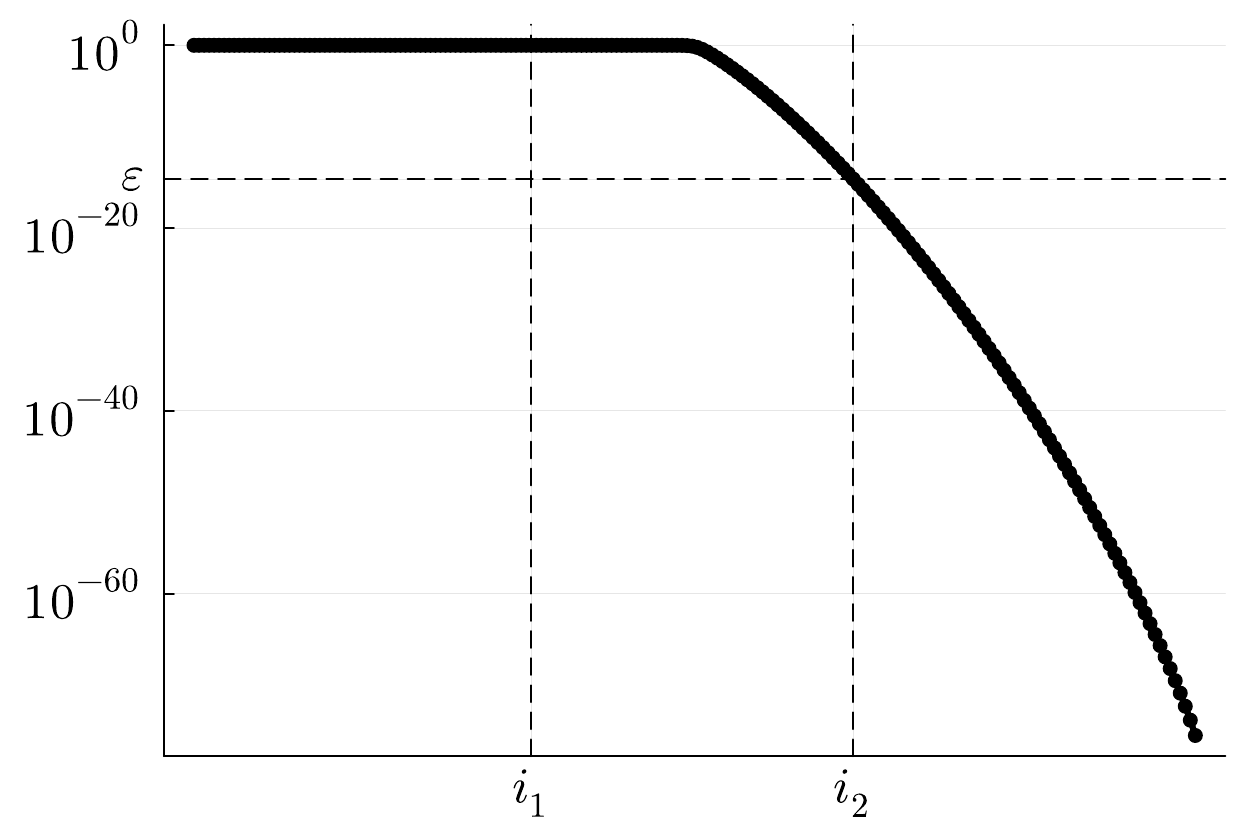}
    \caption{Typical singular value profile of the synthesis operator $\mathcal{T}$ associated with a finite subsequence of a linearly independent, overcomplete frame with frame bounds $A=B=1$. The spectrum can be divided into three regions based on the regularization parameter $\epsilon$, which is proportional to the unit roundoff $u$: a well-behaved region where $\sigma_i^2 \geq 1-\epsilon^2\text{ for }i \leq i_1$, a near-null region where $\sigma_i^2 \leq \epsilon^2 \text{ for } i > i_2$, and the intermediate plunge region between $i_1$ and $i_2$.}
    \label{fig:typicalsvd}
\end{figure}

One can split the singular value profile into different regimes based on the regularization parameter $\epsilon$, by identifying $i_1$ and $i_2$ so that 
\begin{equation} \label{eq:k1}
    \sigma_i^2 \geq 1-\epsilon^2  \text{ for } i \leq i_1
\end{equation}
and 
\begin{equation} \label{eq:k2}
    \sigma_i^2 \leq \epsilon^2  \text{ for } i > i_2.
\end{equation}
The region in between $i_1$ and $i_2$ is often referred to as the \textit{plunge region}~\cite{daubechiesTenLecturesWavelets1992}. These regions have a different influence on the effective dimension $\hat{n}^\epsilon$ and the the inverse regularized Christoffel function $k^\epsilon$. The proof of the theorem below can be found in Appendix~\ref{sec:app2}.

\begin{theorem} \label{thm:influenceregularization}
    Given $i_1$ and $i_2$ satisfying~\eqref{eq:k1} and~\eqref{eq:k2}, define $\Stail = \sum_{i = i_2 + 1}^{\hat{n}} \left( \sigma_i/\epsilon  \right)^2$ and $k^\epsilon_{\text{trunc}}(x) = \sum_{i=1}^{i_2} \lvert u_i(x) \rvert^2$ using the singular value decomposition of $\mathcal{T} = \sum_{i=1}^{\hat{n}} \sigma_i u_i(x) v_i^*$. Then,
    \begin{equation} \label{eq:boundonnlambda}
        \hat{n}^\epsilon  \leq i_2 + \Stail
    \end{equation}
    and 
    \begin{equation} \label{eq:boundonridgeleveragescores1}
        \frac{1}{2} k^\epsilon_{\text{trunc}}(x) \leq k^\epsilon (x) \leq k^\epsilon_{\text{trunc}}(x) + \Stail \max_{ i_2 < i \leq \hat{n}} \lvert u_i(x) \rvert^2.
    \end{equation}
    Furthermore, if $\Phi$ is linearly independent and uniformly bounded, $\|\phi_{i}\|_{L^\infty} \leq 1 \; (1 \leq i \leq n)$, then 
    \begin{equation} \label{eq:boundonridgeleveragescores2}
        \| k^\epsilon \|_{L^\infty} \leq \frac{n}{1-\epsilon^2} + ((i_2 - i_1) + \Stail) \max_{ i_1 < i \leq \hat{n}} \| u_i \|_{L^\infty}^2.
    \end{equation}
\end{theorem}

It becomes clear that regularization reduces the influence of small singular values and the associated weakly represented directions. The quantity $\Stail$, which depends heavily on the decay of the singular values, plays a central role. If $\Stail$ is bounded independently of $n$, then $\hat{n}^\epsilon$ scales with $i_2$. In practice, the decay of the singular values typically exhibits a mild dependence on $n$ so that $\Stail$ grows slowly with $n$; see e.g.\ Section~\ref{sec:effectivedofs}. 

We include a comparison with a truncated variant of the inverse regularized Christoffel function\footnote{In a fully discrete setting, $k^\epsilon_{\text{trunc}}$ is linked to leverage scores associated with the best rank-$i_2$ subspace~\cite{alaoui2015fast}.}, $k^\epsilon_{\text{trunc}}$, since many computational approaches for Christoffel sampling are—often implicitly—based on this quantity rather than on $k^\epsilon$. However,~\eqref{eq:boundonridgeleveragescores1} does not provide a bound of the form $k^\epsilon_{\text{trunc}} \lesssim k^\epsilon$, such that it remains unclear whether $k^\epsilon_{\text{trunc}}$ is associated with a near-optimal sampling strategy. In any case, our analysis indicates that the weakly represented directions $u_i$, corresponding to $i > i_2$, are not negligible for the behaviour of $k^\epsilon$.

\subsection{Analysis of the effective dimension} \label{sec:effectivedofs}
The following theorem formalizes the dependence of $\hat{n}^\epsilon$ on $n$, $W$ and $\epsilon$ for the one-dimensional Fourier extension frame~\eqref{eq:fourierextension}

\begin{theorem} \label{thm:effectivedof}
    For the basis~\eqref{eq:fourierextension} with $W < 1/2$,
    \begin{equation} \label{eq:thm11_1}
        i_2 = \lceil 2 n W \rceil + 2 + \frac{2}{\pi^2}\log\left(\frac{8}{\epsilon^2}\right)\log(4n)
    \end{equation}
    satisfies~\eqref{eq:k2}. For this choice of $i_2$, one has
    \begin{equation} \label{eq:thm11_2}
        \Stail \leq \frac{2}{\pi^2}\log(4n)
    \end{equation}
    such that the effective dimension $\hat{n}^\epsilon $ can be bounded by 
    \begin{equation} \label{eq:thm11_3}
        \hat{n}^\epsilon  \leq \lceil 2nW \rceil + 2 + \frac{2}{\pi^2} \left( \log\left(\frac{8}{\epsilon^2}\right) + 1 \right) \log(4n).
    \end{equation}
\end{theorem}
\begin{proof}
    Note that $\sigma_i^2$ equals the $i$-th eigenvalue of the Gram matrix $G$. Furthermore, $G$ is exactly the \textit{prolate matrix}~\cite[eq.\ (1)]{karnikImprovedBoundsEigenvalues2021}. Therefore, $\sigma_i^2 = \lambda_{i-1}$, for all $i = 1, \dots, n$, where the latter is analyzed in \cite{karnikImprovedBoundsEigenvalues2021}. It follows from~\cite[Corollary 1]{karnikImprovedBoundsEigenvalues2021} that~\eqref{eq:k2} holds for~\eqref{eq:thm11_1}. Moreover,~\eqref{eq:thm11_2} can be deduced by~\cite[Corollary 2]{karnikImprovedBoundsEigenvalues2021}. The final result is a consequence of Thm.~\ref{thm:influenceregularization}. 
\end{proof}

A numerical verification of Thm.~\ref{thm:effectivedof} is shown on Figure~\ref{fig:dofs} for $n = 141$, $\epsilon = 100 \, \machepssp \approx 10^{-5}$ and $\epsilon = 100 \, \machepsdp \approx 2 \times 10^{-14}$, where $\machepssp$ and $\machepsdp$ denote the unit round-offs associated with IEEE single and double precision floating point numbers, respectively. For implementation details, we refer to~\cite{githubrepo}.

Both theory and numerical verification show that the effective dimension essentially scales as $2nW$. This formalizes the intuition that increased redundancy, i.e., smaller $W$, results in a reduced amount of required information. Observe that $\hat{n}^\epsilon  = \mathcal{O}(n)$ such that, asymptotically, the number of samples needed for regularized and unregularized weighted least squares approximation is similar, namely $\mathcal{O}(n \log(n))$.

\begin{figure}
    \centering
    \begin{subfigure}{.49\linewidth}
        \centering 
        \includegraphics[width=\linewidth]{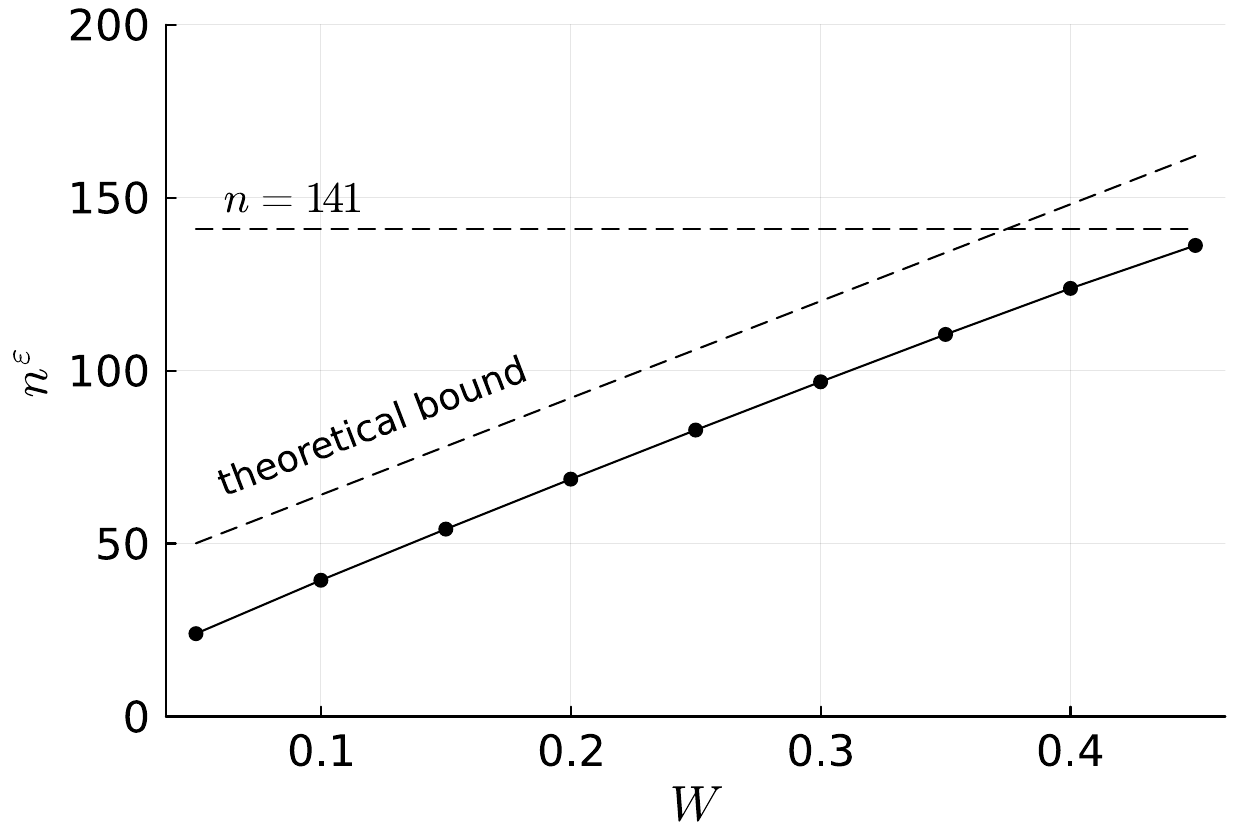}
        \caption{$\epsilon = 100 \, \machepssp \approx 10^{-5}$}
    \end{subfigure}
    \begin{subfigure}{.49\linewidth}
        \centering 
        \includegraphics[width=\linewidth]{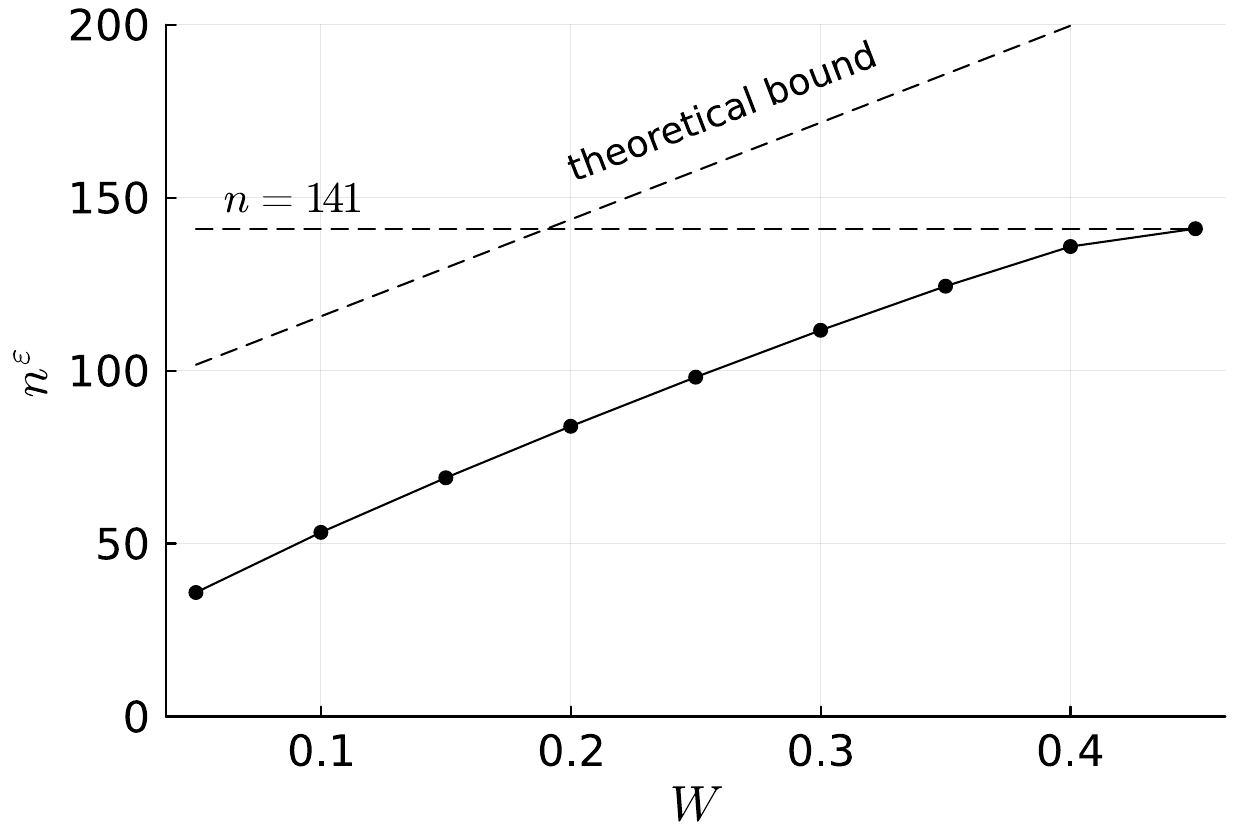}
        \caption{$\epsilon = 100 \, \machepsdp \approx 2 \times 10^{-14}$}
    \end{subfigure}
    \caption{The effective dimension $\hat{n}^\epsilon $ for the univariate Fourier extension frame~\eqref{eq:fourierextension} as a function of the domain parameter $W$, for $n = 141$. The dashed lines illustrate $\hat{n}^\epsilon \leq n = 141$ and the theoretical bound $\hat{n}^\epsilon  \leq \lceil 2nW \rceil + 2 + \frac{2}{\pi^2} \left( \log(8/\epsilon ^2) + 1 \right) \log(4n)$ stated in Thm.~\ref{thm:effectivedof}.}
    \label{fig:dofs} 
\end{figure}

\subsection{Analysis of uniformly random sampling} \label{sec:uniformsampling}
In many applications, one cannot freely choose the sample distribution. Another interesting question to answer in this regard is how many uniformly random samples are needed for accurate approximation. As follows from Prop.~\ref{prop:christoffel} and Thm.~\ref{thm:christoffel2}, this number is proportional to the maximum of $k$ and $k^\epsilon$ for unregularized and regularized least squares fitting, respectively. The following theorem shows that $\|k\|_{L^\infty} = \mathcal{O}(n^2)$.

\begin{theorem} \label{thm:uniformsampling1}
    For the basis~\eqref{eq:fourierextension} with $W < 1/2$, $k$ satisfies
    \[
        \|k\|_{L^\infty} \geq k(W) \sim n^2 \frac{\pi}{2} \cot(W\pi) \qquad \text{as $n \to \infty$}.
    \]
\end{theorem}
\begin{proof}
    In~\cite{webb2020pointwise} a related kernel is being analyzed as a means to determine the pointwise convergence behaviour of Fourier extensions. More specifically, in~\cite{webb2020pointwise}, the kernel $K_N(x,y)$ is analyzed, which we relate to the inverse Christoffel function $k$ in our setting via
    \[
        k(x) = K_n(x/W, x/W)/W, \qquad x \in [-W,W].
    \]
    We can, therefore, restrict ourselves to the analysis of $K_N(x,x)$ for $x \in [-1,1]$. The result can now be obtained by combining~\cite[Theorem 4.1]{webb2020pointwise} with~\cite[eq. 8]{webb2020pointwise}, which is valid for large degree $n$ and for $x \in [1-\delta,1]$ for some $\delta > 0$. After expanding all quantities for large $n$ and plugging in $x = 1$, one obtains the asymptotic result via lengthy but straightforward calculations. The inequality $\|k\|_{L^\infty} \geq k(W)$ holds since $\sup_{x \in X} \lvert k(x) \rvert = \max_{x \in X} \lvert k(x) \rvert$.
\end{proof}

Using~\eqref{eq:boundonridgeleveragescores2}, we obtain a bound for $k^\epsilon$ that depends on the size of the plunge region $i_2 - i_1$ and $\Stail$. The latter is shown to be $\mathcal{O}(\log(n))$ in Thm.~\ref{thm:effectivedof}. The following theorem proves that the size of the plunge region is $\mathcal{O}(\log(n))$ as well.

\begin{theorem} \label{thm:uniformsampling2}
    For the basis~\eqref{eq:fourierextension} with $W < 1/2$,
    \begin{equation} \label{eq:i1}
        i_1 = \lfloor 2nW \rfloor - 1 - \frac{2}{\pi^2}\log\left(\frac{8}{\epsilon^2}\right)\log(4n)
    \end{equation}
    satisfies~\eqref{eq:k1}. Moreover, the maximum of $k^\epsilon$ can be bounded by 
    \begin{equation} \label{eq:fourierextension_k}
        \| k^\epsilon \|_{L^\infty} \leq \frac{n}{1-\epsilon^2} + \left( 4 + \frac{2}{\pi^2} \left( 2\log\left( \frac{8}{\epsilon^2} \right) + 1 \right) \log(4n) \right) \max_{ i_1 < i \leq n} \| u_i \|_{L^\infty}^2.
    \end{equation}
\end{theorem}
\begin{proof}
    Similarly as for the proof of Thm.~\ref{thm:effectivedof}, observe that $\sigma_i^2 = \lambda_{i-1}$, for all $i = 1, \dots, n$, where the latter is analyzed in \cite{karnikImprovedBoundsEigenvalues2021}. It follows from~\cite[Corollary 1]{karnikImprovedBoundsEigenvalues2021} that~\eqref{eq:k1} holds for~\eqref{eq:i1}. By combining the results from Thm.~\ref{thm:effectivedof} with~\eqref{eq:boundonridgeleveragescores2} and~\eqref{eq:i1}, we arrive at~\eqref{eq:fourierextension_k}.
\end{proof}

It remains to investigate the behaviour of the basis functions $u_i$. In Appendix~\ref{sec:app3}, we summarize known asymptotic results due to Slepian, which indicate that $\| u_i \|_{L^\infty} = \mathcal{O}(\sqrt{n})$ as $n \to \infty$ for all $i$. Consequently,
\[
    \| k^\epsilon \|_{L^\infty} = \mathcal{O}(n \log(n)).
\]
Numerical verification is provided in Figure~\ref{fig:uniformsampling}, which confirms that $\|k\|_{L^\infty} = \mathcal{O}(n^2)$, while $\|k^\epsilon\|_{L^\infty}$ appears to grow linearly in $n$. This difference translates into the need for log-linear or quadratic oversampling when computing a least squares approximation in finite-precision or exact arithmetic, respectively, using uniformly random samples. 

As a concrete example, we approximate $f = 1/(1-0.32x)$ on $[-W,W]$ with $W = 0.3$ in the Fourier extension frame defined by~\eqref{eq:fourierextension} using uniformly random samples. The approximation is computed via a truncated singular value decomposition with threshold $\epsilon = 100 \, \machepsdp \approx 2 \times 10^{-14}$. Figure~\ref{fig:uniformsampling2} shows the $L^2$-error, demonstrating that linear oversampling suffices for near-best numerical errors. By contrast, a theoretical analysis ignoring regularization (cf.\ Thm.~\ref{thm:uniformsampling1}) would suggest the need for quadratic oversampling.

The approximation seems to converge root-exponentially; a more detailed analysis of the convergence behaviour is given in~\cite{adcockNumericalStabilityFourier2014}. Without regularization, convergence is exponential—see~\cite{huybrechsFourierExtensionNonperiodic2010} and the example in Section~\ref{sec:numexample1}—but this comes at the cost of quadratic sample complexity. Therefore, we observe root-exponential convergence in the number of sample points $m$, both in the presence and absence of numerical effects. This provides an instructive example in which the negative impact of finite precision (reduced accuracy) is counterbalanced by a positive effect (reduced data requirements). A similar phenomenon has been reported for deterministic, equispaced sample points in polynomial extension frames~\cite{adcockFastStableApproximation2023}.

\begin{figure}
    \centering
    \begin{subfigure}{.49\linewidth}
        \centering
        \includegraphics[width=\linewidth]{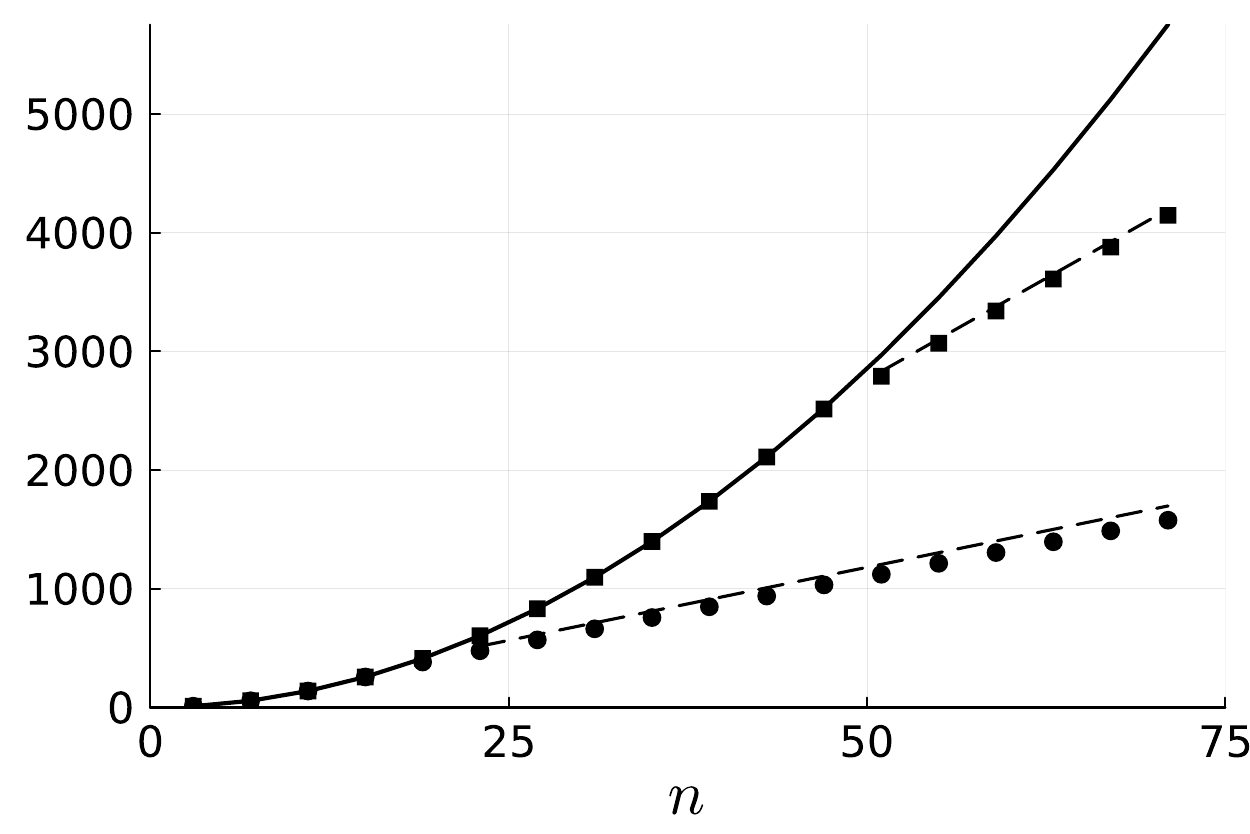}
        \caption{$W = 0.3$}
    \end{subfigure}\hfill%
    \begin{subfigure}{.49\linewidth}
        \centering
        \includegraphics[width=\linewidth]{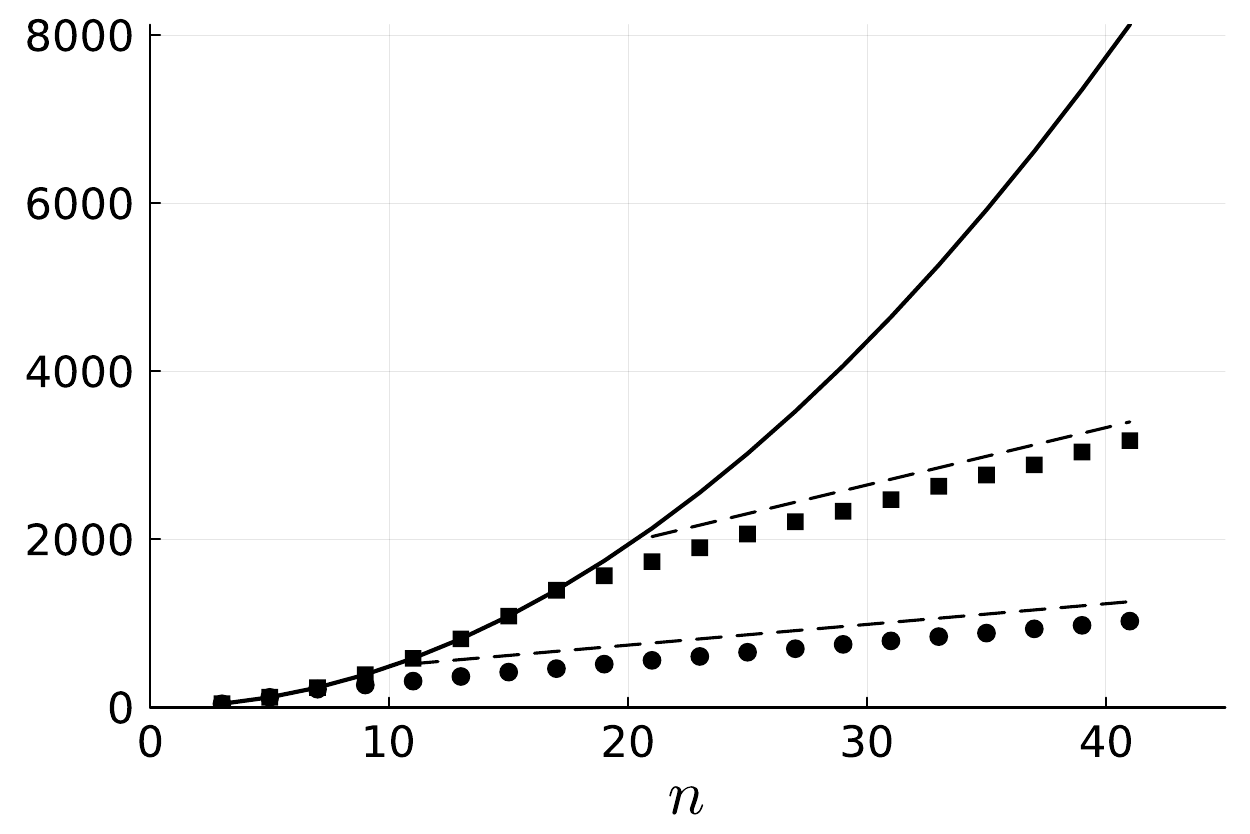}
        \caption{$W = 0.1$}
    \end{subfigure}
    \caption{The maximum of the inverse (regularized) Christoffel function for the univariate Fourier extension frame~\eqref{eq:fourierextension} as a function of $n$. Full line: $\|k\|_{L^\infty}$, dots: $\|k^\epsilon\|_{L^\infty}$ for single precision $\epsilon = 100 \, \machepssp \approx 10^{-5}$, and squares: $\|k^\epsilon\|_{L^\infty}$ for double precision
    $\epsilon = 100 \, \machepsdp \approx 2 \times 10^{-14}$. The linear behaviour of $k^\epsilon$ versus the quadratic behaviour of $k$ implies a difference between log-linear and quadratic sampling rates for least squares fitting with uniformly random samples. The dashed lines mark an empirical estimate of the form $5 \log_{10}(1/\epsilon) + C$, $C \in \mathbb{R}$. }
    \label{fig:uniformsampling}
\end{figure}

\begin{figure}
    \centering
    \includegraphics[width=.6\linewidth]{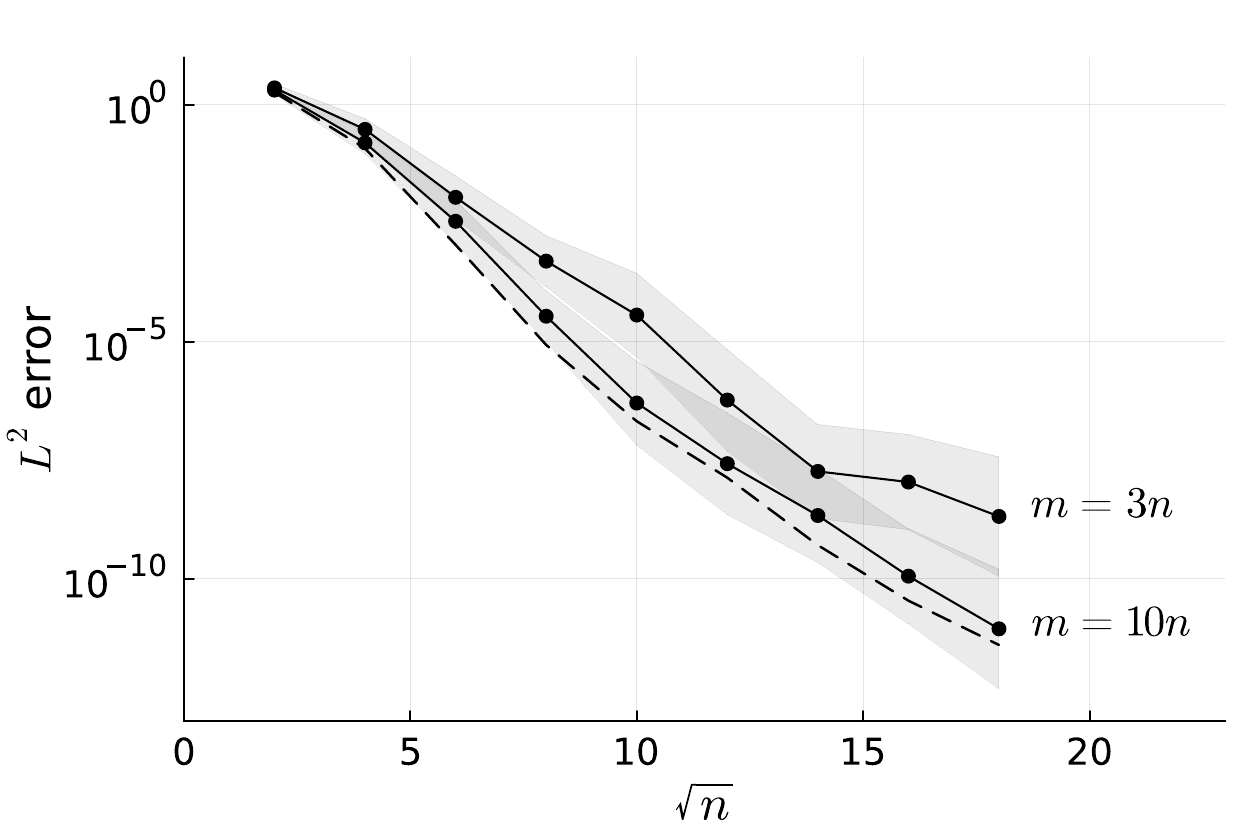}
    \caption{Error of the TSVD approximation of $f = 1/(1 - 0.32x)$ in the Fourier extension frame defined by~\eqref{eq:fourierextension} for $W = 0.3$, using $m$ uniformly random samples and a truncation threshold $\epsilon = 100 \, \machepsdp \approx 2 \times 10^{-14}$. The dots indicate the geometric mean and the shaded area shows the variance over ten repetitions. The dashed line corresponds to approximation using $\mathcal{O}(n^2)$ equispaced samples. The figure illustrates that linear oversampling suffices for near-best numerical errors. }
    \label{fig:uniformsampling2}
\end{figure}

\section*{Acknowledgements}
We greatly appreciate the fruitful discussions and helpful suggestions provided by Ben Adcock, Felix Bartel, Nick Dewaele, Karlheinz Gröchenig, Christopher Musco, Thijs Steel, Alex Townsend, Nick Trefethen and Raf Vandebril.

\section*{Funding}
The first author is a Ph.D fellow of the Research Foundation Flanders (FWO), funded by grant 11P2T24N.

\appendix

\section{Accuracy under backward stable computation}\label{sec:app1}
\begin{proof}[Proof of Thm.~\ref{thm:backwardstab1}]
    Due to the definition of $\hat{c}$, we obtain for every $c \in \mathbb{C}^n$ that
    \begin{align*}
        \|A\hat{c} - b\|_2 &\leq \|(A+\Delta A)\hat{c} - (b+\Delta b)\|_2 + \|\Delta A\|_2 \|\hat{c}\|_2 + \|\Delta b\|_2 \\
        &\leq \|(A+\Delta A)c - (b+\Delta b)\|_2 + \|\Delta A\|_2 \|\hat{c}\|_2 + \|\Delta b\|_2 \\
        &\leq \|Ac - b\|_2 + \|\Delta A\|_2 (\|\hat{c}\|_2 + \|c\|_2) + 2 \|\Delta b\|_2 \\
        &\leq \|Ac - b\|_2 + \Calg u \left( \|A\|_2 \left(\|\hat{c}\|_2 + \|c\|_2 \right) + 2\|b\|_2\right)
    \end{align*}
    giving the first result. Furthermore, note that 
    \[
        \|\hat{c}\|_2 = \|(A+\Delta A)^\dagger (b+\Delta b)\|_2 \leq \frac{\|b+\Delta b\|_2}{\sigma_{\min}(A+\Delta A)} \leq \frac{(1+\Calg u) \|b\|_2}{\sigma_{\min}(A+\Delta A)},
    \]
    where $\sigma_{\min}$ denotes the smallest nonzero singular value. When $\kappa(A) > (\Calg u)^{-1}$, we have $$\sigma_{\min}(A+\Delta A) \geq \sigma_{\min}(A) - \|\Delta A\|_2 \geq \sigma_{\min}(A) - \Calg u \|A\|_2,$$ yielding the bound on $\|\hat{c}\|_2$. For the bound on the residual, consider $c = A^\dagger b$ in the first inequality and substitute the bound on $\|\hat{c}\|_2$. This results in
    \begin{align*}
        \|A\hat{c} - b\|_2 &\leq  \|AA^\dagger b - b\|_2 + \Calg u \left( \|A\|_2 \left(\|\hat{c}\|_2 + \|A^\dagger b\|_2 \right) + 2\|b\|_2\right) \\
        &\leq  \|AA^\dagger b - b\|_2 + \Calg u \left( \frac{(1+\Calg u)\|A\|_2 }{\sigma_{\min}(A) - \Calg u \|A\|_2} + \kappa(A) + 2\right) \|b\|_2.
    \end{align*}
\end{proof}

\begin{proof}[Proof of Thm.~\ref{thm:backwardstab2}]
    Due to~\eqref{thm:backwardstab2:eq}, we obtain for every $c \in \mathbb{C}^n$ that 
    \begin{align*}
        e_\epsilon(\hat{c}; A, b) = \|A\hat{c} - b\|_2 + \epsilon \|\hat{c}\|_2 &\leq \|(A+\Delta A)\hat{c} - (b+\Delta b)\|_2 + \epsilon \|\hat{c}\|_2 + \|\Delta A\|_2 \|\hat{c}\|_2 + \|\Delta b\|_2 \\ 
        &\leq \|(A+\Delta A)\hat{c} - (b+\Delta b)\|_2 + 2 \epsilon \|\hat{c}\|_2 + \|\Delta b\|_2  \\
        &\leq 2C \left( \|(A+\Delta A)c - (b+\Delta b)\|_2 + \epsilon \|c\|_2 \right) + \|\Delta b\|_2 \\
        &\leq 2C \|Ac - b\|_2 + 3C \epsilon \|c\|_2 + (1+2C)\|\Delta b\|_2.
    \end{align*}
\end{proof}

\section{Properties of the (regularized) Christoffel function} \label{sec:app2}
\begin{proof}[Proof of Prop.~\ref{prop:extremalproperty}]
    Consider the singular value decomposition $\mathcal{T} = \sum_{i=1}^{\hat{n}} \sigma_i u_i v_i^*$ and define $a \in \mathbb{C}^{\hat{n}}$ via $(a)_i = \langle v_i, c \rangle_2$ such that $c = \sum_{i=1}^{\hat{n}} a_i v_i$. We find that the denominator equals
    \[
        \|\mathcal{T}c \|_{L^2_\rho}^2 + \epsilon^2 \|c\|_2^2 = \norm{\sum_{i=1}^{\hat{n}} a_i \sigma_i u_i }_{L^2_\rho}^2 + \epsilon^2 \|a\|_2^2 = \sum_{i=1}^{\hat{n}} (\sigma_i^2 + \epsilon^2) \vert a_i \vert^2 
    \]
    using the orthonormality of $\{u_i\}_{i=1}^{\hat{n}}$. Moreover, the numerator equals 
    \[
        \vert \phi(x) c \vert^2 = \left\vert \sum_{i=1}^{\hat{n}} a_i \sigma_i u_i(x) \right\vert^2 
    \]
    using $\phi(x) = \sum_{i=1}^{\hat{n}} \sigma_i u_i(x) v_i^*$. Consider now $b \in \mathbb{C}^{\hat{n}}$ defined by $(b)_i = \sqrt{\sigma_i^2 + \epsilon^2} \, \overline{a_i}$, then 
    \[
        \frac{\left\vert \sum_{i=1}^{\hat{n}} a_i \sigma_i u_i(x) \right\vert^2 }{ \sum_{i=1}^{\hat{n}} (\sigma_i^2 + \epsilon^2) \vert a_i \vert^2 } = \frac{\left| \langle b, d \rangle_2 \right|^2}{\|b\|_2^2} \qquad \text{where} \qquad d = \begin{bmatrix} \frac{\sigma_1}{\sqrt{\sigma_1^2 + \epsilon^2}} u_1(x) & \dots & \frac{\sigma_{\hat{n}}}{\sqrt{\sigma_{\hat{n}}^2 + \epsilon^2}} u_{\hat{n}}(x) \end{bmatrix}^\top
    \]
    for fixed $x \in X$. Maximizing the right-hand side of~\eqref{eq:extremalprop} over $c$ corresponds to maximizing the expression above over $b$. This maximizer equals $b^\star = d$ and, hence, we obtain 
    \[
        \frac{\left| \langle b^\star, d \rangle_2 \right|^2}{\|b^\star\|_2^2} = \frac{\left| \langle d, d \rangle_2 \right|^2}{\|d\|_2^2} = \|d\|_2^2 = \sum_{i=1}^{\hat{n}} \frac{\sigma_i^2}{\sigma_i^2 + \epsilon^2} \lvert u_i(x) \rvert^2 = k^\epsilon(x).
    \]
\end{proof}

\begin{proof}[Proof of Prop.~\ref{prop:grammatrixform}]
    $G$ is positive semidefinite and, hence, there exists a (partial) eigenvalue decomposition of the form $V \Sigma^2 V^*$, where $V\in\mathbb{C}^{n \times \hat{n}}$ is unitary and $\Sigma \in \mathbb{C}^{\hat{n} \times \hat{n}}$ is diagonal. Moreover, this implies that the synthesis operator $\mathcal{T}$ associated with $\Phi$ has a singular value decomposition of the form $\mathcal{T} = \sum_{i=1}^{\hat{n}} \sigma_i u_i v_i^*$ where $\{u_i\}_{i=1}^{\hat{n}}$ is an orthonormal basis for $\SPAN(\Phi)$, $v_i$ are the columns of $V$ and $\sigma_i = \Sigma_{i,i}$. From here we get $\phi(x) = \sum_{i=1}^{\hat{n}} \sigma_i u_i(x) v_i^*$, leading to
    \begin{align*}
        \phi(x) G^\dagger \phi(x)^* &= \left( \sum_{i=1}^{\hat{n}} \sigma_i u_i(x) v_i^* \right) V \Sigma^{-2} V^* \left( \sum_{i=1}^{\hat{n}} \sigma_i \overline{u_i(x)} v_i \right) \\
        &= \left( \sum_{i=1}^{\hat{n}} \sigma_i u_i(x) v_i^* V \right) \Sigma^{-2} \left( \sum_{i=1}^{\hat{n}} \sigma_i \overline{u_i(x)} V^* v_i \right) \\
        &= \begin{bmatrix} \sigma_1 u_1(x) & \dots & \sigma_{\hat{n}} u_{\hat{n}}(x) \, \end{bmatrix} \Sigma^{-2} \begin{bmatrix} \sigma_1 \overline{u_1(x)} & \dots & \sigma_{\hat{n}} \overline{u_{\hat{n}}(x)} \end{bmatrix}^\top \\
        &= \sum_{i=1}^{\hat{n}} \lvert u_i(x) \rvert^2,
    \end{align*}
    which equals $k(x)$ following its definition. From a very similar derivation using the eigenvalue decomposition of $G + \epsilon^2 I = V_{\text{ext}} (\Sigma_{\text{ext}}^2 + \epsilon^2 I) V_{\text{ext}}^*$ where $V_{\text{ext}}, \Sigma_{\text{ext}} \in \mathbb{C}^{n \times n}$, one obtains the characterization of $k^\epsilon$. 
\end{proof}

\begin{proof}[Proof of Prop.~\ref{prop:discretechristoffel}]
    We have 
    \begin{align*}
        \eqref{eq:normequivalence1} \; \Rightarrow \; \frac{1}{1+\eta} G \preceq A^*A \preceq \frac{1}{1-\eta} G \; &\Rightarrow \; (1-\eta)G^\dagger \preceq (A^*A)^\dagger \preceq (1+\eta)G^\dagger \\ &\Rightarrow \; (1-\eta)k(x) \leq k^d(x) \leq (1+\eta)k(x)
    \end{align*}
    for all $x \in X$, from which the final result follows immediately. A similar derivation can be done in the regularized case, leading to
    \begin{align*}
        \eqref{eq:normequivalence2} \; \Rightarrow \; \frac{1}{1+\eta} (G+\epsilon^2 I) \preceq A^*A + \epsilon^2 I \preceq \frac{1}{1-\eta} (G + \epsilon^2 I) \; \Rightarrow \; (1-\eta)k^\epsilon(x) \leq k^{d,\epsilon}(x) \leq (1+\eta)k^\epsilon(x).
    \end{align*}
    For the alternative characterization, note that $A^\dagger = (A^*A)^\dagger A^*$ and, hence, 
    \[
        A^\dagger (A^\dagger)^* = (A^*A)^\dagger A^* A ((A^* A)^\dagger)^* = (A^*A)^\dagger A^* A (A^* A)^\dagger = (A^* A)^\dagger
    \]
    where we also use that $(A^* A)^\dagger$ is Hermitian and $M^\dagger M M^\dagger = M^\dagger$. Therefore, 
    \[
        k^d(x) = \phi(x) (A^* A)^\dagger \phi(x)^* = \phi(x) A^\dagger (A^\dagger)^* \phi(x)^* = \norm{\phi(x) A^\dagger}_2^2. 
    \]
    Analogously, denoting $\Aext = \begin{bmatrix} A \\ \epsilon I \end{bmatrix}$,
    \[
        k^\epsilon(x) = \phi(x) (A^*A+\epsilon^2 I )^{-1} \phi(x)^* = \phi(x) \left( \Aext^* \Aext \right)^{-1} \phi(x)^* = \norm{\phi(x) \Aext^\dagger}_2^2. 
    \]
\end{proof}

\begin{proof}[Proof of Thm.~\ref{thm:influenceregularization}]
    Equation~\eqref{eq:boundonnlambda} follows from splitting the summation that defines $\hat{n}^\epsilon$ as 
    \[
        \hat{n}_\epsilon  = \sum_{i=1}^{\hat{n}} \frac{\sigma_i^2}{\sigma_i^2 + \epsilon ^2} = \sum_{i=1}^{i_2} \frac{\sigma_i^2}{\sigma_i^2 + \epsilon ^2} + \sum_{i=i_2+1}^{\hat{n}} \frac{\sigma_i^2}{\sigma_i^2 + \epsilon ^2} \leq \sum_{i=1}^{i_2} 1 + \sum_{i=i_2+1}^{\hat{n}} \frac{\sigma_i^2}{\epsilon ^2} = i_2 + \Stail.
    \]
    For~\eqref{eq:boundonridgeleveragescores1}, consider 
    \begin{equation*}
        k^\epsilon_\text{trunc} (x) = \sum_{i=1}^{i_2} \lvert u_i(x) \rvert^2 \leq 2 \sum_{i=1}^{i_2} \frac{\sigma_i^2}{\sigma_i^2 + \epsilon^2}\lvert u_i(x) \rvert^2 \leq 2 \sum_{i=1}^{\hat{n}} \frac{\sigma_i^2}{\sigma_i^2 + \epsilon^2} \lvert u_i(x) \rvert^2 = 2k_n^\epsilon(x)
    \end{equation*}
    and
    \begin{equation*}
        k^\epsilon (x) = \sum_{i=1}^{\hat{n}} \frac{\sigma_i^2}{\sigma_i^2 + \epsilon ^2} \lvert u_i(x) \rvert^2 \leq \sum_{i=1}^{i_2} \lvert u_i(x) \rvert^2 + \sum_{i=i_2 + 1}^{\hat{n}} \frac{\sigma_i^2}{\epsilon ^2} \lvert u_i(x) \rvert^2 \leq k^\epsilon_\text{trunc}(x) + \Stail  \max_{ i_2 < i \leq \hat{n}} \lvert u_i(x) \rvert^2.
    \end{equation*}
    In order to obtain~\eqref{eq:boundonridgeleveragescores2}, we split the above as 
    \begin{align*}
        \| k^\epsilon \|_{L^\infty} &\leq \esssup_{x \in X} \left( \sum_{i=1}^{i_1} \lvert u_i(x) \rvert^2 + \sum_{i=i_1+1}^{i_2} \lvert u_i(x) \rvert^2 + \Stail \max_{ i_2 < i \leq \hat{n}} \lvert u_i(x) \rvert^2 \right) \\ &\leq \esssup_{x\in X} \sum_{i=1}^{i_1} \lvert u_i(x) \rvert^2 + ((i_2 - i_1) + \Stail) \max_{ i_1 < i \leq \hat{n}} \| u_i\|_{L^\infty}^2.
    \end{align*}
    Define $\Sigma_1 \in \mathbb{C}^{i_1 \times i_1}$ via $\Sigma_{i,i} = \sigma_i$ and let $V_1 \in \mathbb{C}^{n \times i_1}$ contain the vectors $\{v_i\}_{i=1}^{i_1}$ as columns (coming from the SVD of $\mathcal{T}$), then the first term can be bounded via
    \[
        \sum_{i=1}^{i_1} \lvert u_i(x) \rvert^2 = \| \phi(x) V_1 \Sigma_1^{-1} \|_2^2  \leq \| \phi(x) \|_2^2 \| V_1 \|_2^2 \| \Sigma_1^{-1} \|_2^2 = \frac{\| \phi(x) \|_2^2}{\sigma_{i_1}^2} \leq \frac{\| \phi(x) \|_2^2}{1-\epsilon^2},
    \]
    where $\phi(x) = \begin{bmatrix} \phi_1(x) & \dots & \phi_n(x) \end{bmatrix}$. The final result follows from
    \[
        \esssup_{x \in X} \frac{\| \phi(x) \|_2^2}{1-\epsilon^2} = \frac{1}{1-\epsilon^2} \esssup_{x \in X} \sum_{i=1}^{n} \lvert \phi_i(x) \rvert^2 \leq \frac{1}{1-\epsilon^2} \sum_{i=1}^{n} \|\phi_i\|_{L^\infty}^2 \leq \frac{n}{1-\epsilon^2}.
    \]
\end{proof}

\section{Asymptotics of DPSWFs} \label{sec:app3}
In this section, we aim to motivate that the maxima $\|u_i\|_{L^\infty}$ associated to the one-dimensional Fourier extension frame~\eqref{eq:fourierextension} are $\mathcal{O}(\sqrt{n} \kern1pt)$ as $n \to \infty$. In order to do so, observe that
\begin{equation} \label{eq:linktodpswf}
    \|u_i\|_{L^\infty} = \|U_{i-1}(n,W)\|_{L^\infty}/\sqrt{\lambda_{i-1}(n,W)}, 
\end{equation}
where $U_k(n,W)$ and $\lambda_k(n,W)$ are the discrete prolate spheroidal wave functions (DPSWFs) and the associated eigenvalues analyzed by Slepian in~\cite{slepianProlateSpheroidalWave1978}. We recall and examine their asymptotic behaviour for the different regimes identified in~\cite[\S2.4-2.5]{slepianProlateSpheroidalWave1978}. Numerical evidence confirms that~\eqref{eq:linktodpswf} is bounded by $\mathcal{O}(\sqrt{n})$, as is the case for the prolate spheroidal wave functions~\cite[Prop.\ 3.1]{bonamiGaussianBoundsHeat2023}, the continuous analogues of the DPSWFs. 

Given the asymptotic expansions outlined in~\cite[section 2.4-2.5]{slepianProlateSpheroidalWave1978}, we can deduce that as $n \to \infty$:
\begin{itemize}
    \item for $i \leq 2nW + 1$,
    \begin{equation}
        \|u_i\|_{L^\infty} = \mathcal{O}(\sqrt{n}),
    \end{equation}
    \item for $i = \lfloor 2nW + (b/\pi) \log(n) \rfloor + 1$ with fixed $b > 0$,
    \begin{equation}
        \|u_i\|_{L^\infty} \sim r(E \beta) \sqrt{\beta } \exp(\pi E \beta / 4) \sqrt{\frac{3\pi (1 + \exp(\pi b))}{\beta(1 + \exp(\pi E \beta)) \ln(n)}} \sqrt{n},
    \end{equation}
    where $r$, $E$ and $\beta$ are defined by~\cite[eq. (52)-(54)]{slepianProlateSpheroidalWave1978},
    \item for $i = \lfloor 2nW(1 + \epsilon) \rfloor + 1$ with fixed $0 < \epsilon < 1/2W- 1$,
    \begin{equation}
        \|u_i\|_{L^\infty} \sim L_2^{-1/2}\pi  (1 - \cos(2\pi W)^2)^{-1/4} \sqrt{n},
    \end{equation}
    where $L_2$ is defined by~\cite[eq. (47)]{slepianProlateSpheroidalWave1978} with $A = -\cos(2\pi W)$ and $k = n - (i - 1)$,
    \item for $i = n - l$ with fixed $0 \leq l$,
    \begin{equation}
        \|u_i\|_{L^\infty} \sim \sqrt{\pi} \left( \frac{2}{1 - \cos(2\pi W)}\right)^{1/4} \sqrt{n}.
    \end{equation}
\end{itemize}

The result for $i \leq 2nW + 1$ follows from considering
\begin{align*}
    \|u_i\|_{L^\infty[-W,W]} &\leq \|u_i\|_{L^\infty[-1/2,1/2]} \\ &= \|U_{i-1}(n,W)\|_{L^\infty[-1/2,1/2]}/\sqrt{\lambda_{i-1}(n,W)} \\
    &\leq \sqrt{n}/\sqrt{\lambda_{i-1}(n,W)},
\end{align*}
where we used the Nikolskii inequality~\cite[Theorem 1]{nesselNikolskiitypeInequalitiesTrigonometric1978} with $\|U_{i-1}(n,W)\|_{L^\infty[-1/2,1/2]} = 1$. 
Furthermore,
\[
    1/\sqrt{\lambda_{i-1}(n,W)} \leq 1/\sqrt{\lambda_{\lfloor 2nW \rfloor }(n,W)} \sim \sqrt{2}
\]
using~\cite[eq.\ (60)]{slepianProlateSpheroidalWave1978}. Hence, we get $\|u_i\|_{L^\infty[-W,W]} = \mathcal{O}(\sqrt{n})$, where the proportionality factor is at most approximately $\sqrt{2}$.

For the case of $i \geq 2nW + 1$, it follows from~\cite[Corollary 2]{saidNonasymptoticBoundsDiscrete2023} and the fact that the DPSWFs are even or odd that 
\[
    \|u_i\|_{L^\infty} = \frac{\| U_{i-1}(n,W) \|_{L^\infty}}{\sqrt{\lambda_{i-1}(n,W)}} = \frac{\lvert U_{i-1}(n,W; \omega = 2\pi W) \rvert}{\sqrt{\lambda_{i-1}(n,W)}}. 
\]
Hence, the maximum is achieved at the boundary. Furthermore, by combining~\cite[eq.\ (50) and (60)]{slepianProlateSpheroidalWave1978} for $U_k(n,W)$ and $\lambda_k(n,W)$ for $k = i-1 = \lfloor 2nW + (b/\pi) \log(n) \rfloor$ with fixed $b$, one obtains 
\[
    \frac{\lvert U_k(n,W; \omega = 2\pi W) \rvert}{\sqrt{\lambda_k(n,W)}} \sim r(E \beta) \sqrt{\beta } \exp(\pi E \beta / 4) \sqrt{\frac{3\pi (1 + \exp(\pi b))}{\beta(1 + \exp(\pi E \beta)) \ln(n)}} \sqrt{n}
\]
Similarly, for $k = i - 1 = \lfloor 2nW(1 + \epsilon) \rfloor$ with fixed $0 < \epsilon < 1/2W-1$ we derive from~\cite[eq.\ (42), (56), (57) and (63)]{slepianProlateSpheroidalWave1978} that
\[
    \frac{\lvert U_k(n,W; \omega = 2\pi W) \rvert}{\sqrt{\lambda_k(n,W)}} \sim L_2^{-1/2}\pi  (1 - \cos(2\pi W)^2)^{-1/4} \sqrt{n},
\]
while for $k = i-1 = n-l$ with fixed $l \geq 1$ we derive from~\cite[eq.\ (13), (38) and (58)]{slepianProlateSpheroidalWave1978} that
\[ 
    \frac{\lvert U_k(n,W; \omega = 2\pi W) \rvert}{\sqrt{\lambda_k(n,W)}} \sim \sqrt{\pi} \left( \frac{2}{1 - \cos(2\pi W)}\right)^{1/4} \sqrt{n}.
\]

\bibliographystyle{abbrv}
\bibliography{ref2}

\end{document}